\documentclass[twoside,11pt]{article}

\usepackage[abbrvbib, preprint]{jmlr2e}

\usepackage{jmlr2e}

\usepackage{amsmath,natbib}

\usepackage{amssymb,latexsym}
\usepackage[usenames]{color}
\usepackage[]{graphicx}
\usepackage{mathrsfs}   
\usepackage{verbatim}
\usepackage{url}
\usepackage{bm}
\usepackage{dsfont}
\usepackage{extarrows}
\usepackage{multirow}
\usepackage{enumerate}
\usepackage{tikz}
\usetikzlibrary{fit}					
\usepackage{setspace}

\newtheorem{condition}[theorem]{Condition}
\numberwithin{equation}{section}

\allowdisplaybreaks 

\newcommand{\sint}{{\textstyle\int}}

\let\liminf\relax
\let\limsup\relax
\DeclareMathOperator*{\liminf}{liminf}
\DeclareMathOperator*{\limsup}{limsup}
\DeclareMathOperator*{\lspan}{span}

\DeclareMathOperator*{\Bernoulli}{Bernoulli}

\DeclareMathOperator*{\Poisson}{Poisson}

\newcommand{\Normal}{\mathcal{N}}

\DeclareMathOperator*{\argmin}{arg\,min}
\DeclareMathOperator*{\argmax}{arg\,max}

\newcommand{\R}{\mathbb{R}}
\newcommand{\N}{\mathbb{N}}
\newcommand{\Z}{\mathbb{Z}}
\newcommand{\E}{\mathrm{E}}
\renewcommand{\Pr}{\mathbb{P}}
\newcommand{\I}{\mathds{1}}

\newcommand{\X}{\mathcal{X}}
\newcommand{\Y}{\mathcal{Y}}

\renewcommand{\epsilon}{\varepsilon}
\newcommand{\T}{\mathtt{T}}

\newcommand{\branch}[4]{
\left\{
	\begin{array}{ll}
        #1  & #2 \\[0.5em]
		#3 & #4
	\end{array}
\right.
}

\def\app#1#2{%
  \mathrel{%
    \setbox0=\hbox{$#1\sim$}%
    \setbox2=\hbox{%
      \rlap{\hbox{$#1\propto$}}%
      \lower1.3\ht0\box0%
    }%
    \raise0.25\ht2\box2%
  }%
}

\bibpunct[, ]{(}{)}{;}{a}{,}{,}
\graphicspath{{figures/}}

\jmlrheading{1}{2021}{1-25}{5/20}{10/00}{miller20a}{Jeffrey W. Miller}

\ShortHeadings{Generalized posteriors}{Miller}
\firstpageno{1}

\begin{document}

\title{Asymptotic normality, concentration, and coverage of generalized posteriors}

\author{\name Jeffrey W. Miller \email jwmiller@hsph.harvard.edu \\
       \addr Department of Biostatistics\\
       Harvard University\\
       Boston, MA 02115, USA}

\editor{TBD}

\maketitle

\begin{abstract}
Generalized likelihoods are commonly used to obtain consistent estimators with attractive computational and robustness properties. Formally, any generalized likelihood can be used to define a generalized posterior distribution, but an arbitrarily defined ``posterior'' cannot be expected to appropriately quantify uncertainty in any meaningful sense. In this article, we provide sufficient conditions under which generalized posteriors exhibit concentration, asymptotic normality (Bernstein--von Mises), an asymptotically correct Laplace approximation, and asymptotically correct frequentist coverage. We apply our results in detail to generalized posteriors for a wide array of generalized likelihoods, including pseudolikelihoods in general, the Gaussian Markov random field pseudolikelihood, the fully observed Boltzmann machine pseudolikelihood, the Ising model pseudolikelihood, the Cox proportional hazards partial likelihood, and a median-based likelihood for robust inference of location. Further, we show how our results can be used to easily establish the asymptotics of standard posteriors for exponential families and generalized linear models. We make no assumption of model correctness so that our results apply with or without misspecification.
\end{abstract}



\section{Introduction}
\label{section:introduction}

Many statistical estimation methods are based on maximizing a generalized likelihood function
such as a pseudolikelihood, partial likelihood, or composite likelihood.
Generalized likelihood functions are often advantageous in terms of computation or robustness while still having consistency guarantees,
even though they do not necessarily correspond to the standard likelihood of a probabilistic model.

Formally, any generalized likelihood can be used to construct a generalized posterior
proportional to the generalized likelihood times a prior.
Generalized posteriors have been proposed based on a variety of generalized likelihoods, including 
composite likelihoods \citep{smith2009extended,pauli2011bayesian,ribatet2012bayesian,friel2012bayesian},
restricted likelihoods \citep{pettitt1983likelihood,doksum1990consistent,hoff2007extending,lewis2014bayesian},
partial likelihoods \citep{raftery1996accounting,sinha2003bayesian,kim2009bayesian,ventura2016pseudo},
substitution likelihoods \citep{lavine1995approximate,dunson2005approximate},
modular likelihoods \citep{liu2009modularization,jacob2017better},
quasi-likelihoods \citep{ventura2010default},
generalized method of moments likelihoods \citep{yin2009bayesian},
loss-based likelihoods \citep{jiang2008gibbs,zhang2006information,holmes2016general}, and more.

Generalized posteriors have been used in a range of applications, including
spatial statistics \citep{ribatet2012bayesian},
social networks \citep{friel2012bayesian},
neural networks \citep{hyvarinen2006consistency},
protein modeling \citep{zhou2009bayesian}, 
computer model emulators of physical processes \citep{liu2009modularization},
copula models \citep{hoff2007extending},
survival analysis \citep{raftery1996accounting},
infrastructure networks \citep{bouranis2017efficient},
longitudinal studies \citep{yin2009bayesian}, and
survey sampling \citep{williams2018bayesian}.
Although various theoretical guarantees have been provided in various cases, 
bespoke theory has been needed to verify whether a given generalized posterior will be valid for statistical inference.



In this article, we provide new theoretical results on the asymptotic validity of generalized posteriors.
We provide a range of sufficient conditions for concentration (Section~\ref{section:concentration}),
Bernstein--von Mises asymptotic normality and the Laplace approximation (Section~\ref{section:BVM}),
and asymptotic frequentist coverage of credible sets (Section~\ref{section:coverage}) for generalized posteriors.
For generalized posteriors derived from composite likelihoods---a large class covering essentially all the examples in this article---we informally discuss what can be expected
in terms of consistency and coverage (Section~\ref{section:composite-likelihood}).
We show how our results can easily be applied to many standard posteriors,
including i.i.d.\ exponential family models and (non-i.i.d.)\ generalized linear models for regression (Section~\ref{section:standard-applications}).
We then apply our results to generalized posteriors for an array of generalized likelihoods,
including pseudolikelihoods in general, the Gaussian Markov random field pseudolikelihood,
the fully observed Boltzmann machine pseudolikelihood, the Ising model pseudolikelihood,
the Cox proportional hazards partial likelihood,
and a median-based likelihood for robust inference of location (Section~\ref{section:generalized-applications}).
Finally, we provide a discussion of previous work (Section~\ref{section:previous-work}).

\subsection{Novelty and overview of the results}

In some sense, new Bernstein--von Mises (BvM) theorems are never surprising since they only verify what we already expect to happen
if things are sufficiently nice.
Thus, the utility of a BvM result is directly related to the ease and generality with which it can be applied.
The main novelty of this article is that we provide results that are not only general, but are also relatively easy to apply in practice. 

More specifically, the results in this article are novel in the following respects: (a) we provide rigorous results on generalized posteriors
for non-i.i.d.\ data without any assumption of model correctness (in fact, in our main results, we do not even require there to be a probability model -- true or assumed),
(b) we provide sufficient conditions that are relatively easy to verify when they hold,
and (c) we apply our results to a number of non-trivial examples, providing precise and concrete sufficient conditions for each example.

Standard BvM theorems are only applicable to standard posteriors under correctly specified i.i.d.\ probabilistic models
\citep{van2000asymptotic,ghosh2003bayesian}.
\citet{kleijn2012bernstein} generalize by establishing a Bernstein--von Mises theorem under misspecification, 
but their result still only applies to standard posteriors, and they focus mainly on the i.i.d.\ case.
In contrast, our main results in Sections~\ref{section:concentration} and \ref{section:BVM} do not involve a probability model at all
and are applicable to arbitrary distributions of the form $\pi_n(\theta) \propto \exp(-n f_n(\theta))\pi(\theta)$,
where $\pi$ and the sequence of functions $f_n$ are required to satisfy certain conditions.
By treating the problem in this generality, we provide results for i.i.d.\ and non-i.i.d.\ cases with or without misspecification;
see the examples in Sections~\ref{section:standard-applications} and~\ref{section:generalized-applications}.
Additionally, BvM theorems often only show that the total variation distance converges to zero in probability;
in contrast, we prove it converges to zero almost surely.

For generalized posteriors, much of the previous work on asymptotic normality
tends to rely on unspecified regularity conditions or only establishes weak convergence, that is, convergence in distribution
\citep{doksum1990consistent,lazar2003bayesian,greco2008robust,pauli2011bayesian,ribatet2012bayesian,ventura2016pseudo}.
In contrast, we show convergence in total variation distance
and we provide rigorous results with all assumptions explicitly stated.
Further, the usual regularity conditions in previous work include an assumption of concentration \citep{bernardo2000bayesian}; 
in contrast, we prove concentration.

In general, we make no assumption of model correctness.
However, to ensure that a generalized posterior is doing something reasonable,
it is desirable to have a guarantee of consistency---that is, concentration at the true parameter---if the assumed model is correct or at least partially correct.
To this end, in Section~\ref{section:composite-likelihood} we show that for any composite likelihood derived from a correct model, the resulting generalized posterior 
concentrates at the true parameter under fairly general conditions.
Since many generalized likelihoods can be viewed as composite likelihoods, this establishes consistency in a wide range of cases.
On the other hand, it is well-known that---except in special circumstances---the asymptotic frequentist coverage of composite likelihood-based posteriors 
is typically incorrect unless an adjustment is made \citep{pauli2011bayesian,ribatet2012bayesian};
see Section~\ref{section:composite-likelihood} for more details.

For each main result in Sections~\ref{section:concentration} and \ref{section:BVM},
we provide a range of alternative sufficient conditions, from more abstract to more concrete.
The more abstract versions are more generally applicable, whereas the more concrete versions have conditions that are easier to verify when applicable.
For instance, Theorem~\ref{theorem:BVM} is an abstract BvM theorem involving a quadratic representation condition;
meanwhile, Theorem~\ref{theorem:altogether} is a more concrete BvM theorem involving conditions on derivatives
that are roughly analogous to the conditions in classical BvM theorems. 
We also provide versions of the theorems based on convexity of $f_n$
(see Theorems~\ref{theorem:concentration}(\ref{item:concentration-convex}) and \ref{theorem:altogether}(\ref{item:altogether-convex})),
which is usually easy to verify when it applies and simplifies the other required conditions.

See Section~\ref{section:previous-work} for a detailed technical discussion of how our assumptions, results, and proof techniques compare with those in previous work.


\section{Posterior concentration}
\label{section:concentration}

Theorem~\ref{theorem:f-balls} is a general concentration result for generalized posteriors $\Pi_n$ on a measurable space $(\Theta,\mathcal{A})$. 
The basic structure of the proof of Theorem~\ref{theorem:f-balls} follows that of Schwartz's theorem \citep{schwartz1965bayes,ghosh2003bayesian}.
Although Theorem~\ref{theorem:f-balls} is useful for theoretical purposes, in practice, one typically needs to establish concentration on neighborhoods in a relevant topology on $\Theta$. To this end, Theorem~\ref{theorem:concentration} provides a range of sufficient conditions for concentration on metric space neighborhoods of a point $\theta_0 \in \Theta$.

\begin{condition}
\label{condition:concentration}
Let $f_n:\Theta\to\R$ for $n\in\N$ be a sequence of functions on a probability space $(\Theta,\mathcal{A},\Pi)$.
For all $n$, assume $z_n < \infty$ where $z_n = \int_\Theta \exp(- n f_n(\theta))\Pi(d\theta)$, and 
define the probability measure
$$ \Pi_n(d\theta) = \exp(- n f_n(\theta))\Pi(d\theta)/z_n. $$
\end{condition}

Throughout, all arbitrarily defined functions and sets are assumed to be measurable,
and we denote $\N = \{1,2,\ldots\}$.
Here, $\exp(- n f_n(\theta))$ is interpreted as the ``likelihood'', possibly in some generalized sense,
$\Pi$ is the ``prior'', and $\Pi_n$ is the ``posterior''.

Our main theorems in Sections~\ref{section:concentration} and \ref{section:BVM} do not involve a probability model and do not even require that there be data.
Instead, our results apply to arbitrary deterministic sequences of distributions $\Pi_n$ satisfying certain conditions.
Consequently, the mode of convergence in these theorems is not probabilistic in any sense.
In the applications in Sections~\ref{section:standard-applications} and \ref{section:generalized-applications} that involve probability models, we show that the conditions hold with probability $1$, and 
in this way we obtain almost sure convergence.

\begin{theorem}
\label{theorem:f-balls}
Assume Condition~\ref{condition:concentration}.
If $\theta_0\in\Theta$ and there exists $f:\Theta\to\R$ such that
\begin{enumerate}
\item\label{item:f-balls-cvg} $f_n(\theta)\to f(\theta)$ as $n\to\infty$ for all $\theta\in\Theta$, 
\item\label{item:f-balls-pi} $\Pi(A_\epsilon)>0$ for all $\epsilon>0$, where $A_\epsilon =\{\theta\in\Theta:f(\theta) < f(\theta_0) +\epsilon\}$, and 
\item\label{item:f-balls-min} $\liminf_n \inf_{\theta\in A_\epsilon^c} f_n(\theta)> f(\theta_0)$ for all $\epsilon>0$,
\end{enumerate}
then $\Pi_n(A_\epsilon)\to 1$ as $n\to\infty$, for any $\epsilon>0$.
\end{theorem}

See Section \ref{section:concentration-proof} for the proof. 
In Section~\ref{section:previous-work}, we discuss the interpretation of this result in relation to Schwartz's theorem.

When $\Theta$ is a metric space, the collection of functions $(f_n)$ is said to be \textit{equicontinuous} if for any $\epsilon>0$ there exists
$\delta>0$ such that for all $n\in\N$, $\theta,\theta'\in\Theta$, if $d(\theta,\theta')<\delta$ then $|f_n(\theta) - f_n(\theta')|<\epsilon$.
For a function $f:E\to\R$ where $E\subseteq\R^D$, we denote the gradient by $f'(\theta)$
(that is, $f'(\theta) = (\frac{\partial f}{\partial\theta_i}(\theta))_{i=1}^D\in\R^D$) and the Hessian by $f''(\theta)$
(that is, $f''(\theta) = (\frac{\partial^2 f}{\partial\theta_i\partial\theta_j}(\theta))_{i,j=1}^D\in\R^{D\times D}$)
when these derivatives exist.
We use the following definition of convexity to allow the possibility that the domain $E\subseteq\R^D$ is non-convex:
$f:E\to\R$ is \textit{convex} if for all $\theta,\theta'\in E$ and all $t\in[0,1]$ such that $t \theta + (1-t) \theta'\in E$,
we have $f(t \theta + (1-t) \theta') \leq t f(\theta) + (1-t) f(\theta')$.

\begin{theorem}
\label{theorem:concentration}
Assume Condition~\ref{condition:concentration}.
Suppose $(\Theta,d)$ is a metric space and $\mathcal{A}$ is the resulting Borel sigma-algebra.
Fix $\theta_0\in\Theta$ and denote $N_\epsilon = \{\theta\in\Theta:d(\theta,\theta_0)<\epsilon\}$.
If $\Pi(N_\epsilon)>0$ for all $\epsilon>0$, $f_n\to f$ pointwise on $\Theta$ for some $f:\Theta\to\R$, and
any one of the following three sets of conditions hold, then for any $\epsilon>0$, $\Pi_n(N_\epsilon)\to 1$ as $n\to\infty$.
\begin{enumerate}
\item\label{item:concentration-metric} $f$ is continuous at $\theta_0$ and $\liminf_n \inf_{\theta\in N_\epsilon^c} f_n(\theta) > f(\theta_0)$ for all $\epsilon>0$.
\item\label{item:concentration-equicontinuous} $(f_n)$ is equicontinuous on some compact set $K\subseteq\Theta$, $\theta_0$ is an interior point of $K$, $f(\theta)>f(\theta_0)$ for all $\theta\in K\setminus\{\theta_0\}$, and $\liminf_n \inf_{\theta\in K^c} f_n(\theta) > f(\theta_0)$.
\item\label{item:concentration-convex} $f_n$ is convex for each $n$, $\Theta\subseteq\R^D$ with the Euclidean metric, $\theta_0$ is an interior point of $\Theta$, and either
    \begin{enumerate}
        \item\label{item:concentration-convex-min} $f(\theta)>f(\theta_0)$ for all $\theta\in\Theta\setminus\{\theta_0\}$, or
        \item $f'$ exists in a neighborhood of $\theta_0$, $f'(\theta_0) = 0$, and $f''(\theta_0)$ exists and is positive definite.
    \end{enumerate}
\end{enumerate}
Further, 2 $\Rightarrow$1 and 3 $\Rightarrow$1 under the assumptions of the theorem.
\end{theorem}

See Section \ref{section:concentration-proof} for the proof. 
Note that if $\Theta$ is compact, then case \ref{item:concentration-equicontinuous} with $K=\Theta$ simplifies to
$(f_n)$ being equicontinuous and $f(\theta)>f(\theta_0)$ for all $\theta\in \Theta\setminus\{\theta_0\}$.
This can be used to prove consistency results based on classical Wald-type conditions such as in \citet[Section 1.3.4]{ghosh2003bayesian}.



\section{Asymptotic normality and the Laplace approximation}
\label{section:BVM}

Theorem~\ref{theorem:BVM} establishes general sufficient conditions under which a generalized posterior
exhibits asymptotic normality and an asymptotically correct Laplace approximation, along with concentration at $\theta_0$.
As in Section~\ref{section:concentration}, $\pi(\theta)$ can be interpreted as the prior density
and $\pi_n(\theta)\propto \exp(-n f_n(\theta))\pi(\theta)$ can be thought of as the ``posterior'' density. 
The points $\theta_n$ can be viewed as maximum generalized likelihood estimates.
The proof of Theorem~\ref{theorem:BVM} is concise, but some of the conditions of the theorem are a bit abstract.
Thus, we also provide Theorem~\ref{theorem:altogether} to give more concrete sufficient conditions which, when satisfied, are usually easier to verify.
Theorem~\ref{theorem:altogether} is the main result used in the examples in the rest of the paper.

We remphasize that unlike previous work on BvM, the results in this section only involve conditions on $f_n$ and $\pi$,
and do not involve any assumptions at all regarding the data; indeed, we do not even require that there be any data.
Thus, the limits in these theorems are not probabilistic in any sense -- they are simply limits of deterministic sequences.
When we apply the theorems to statistical models in Sections~\ref{section:standard-applications} and \ref{section:generalized-applications}, we handle the randomness in the data by showing that the conditions of the theorems hold almost surely, which implies almost sure convergence.

We also highlight two supporting results that are employed in the proof of Theorem~\ref{theorem:altogether}.
Theorem~\ref{theorem:natural-sufficient} provides concrete sufficient conditions under which
the quadratic representation (condition \ref{item:BVM-rep}) in Theorem \ref{theorem:BVM} holds.
Theorem~\ref{theorem:regular-convergence} is a pure real analysis result on uniform convergence of $f_n$, $f_n'$, and $f_n''$, which we 
believe is interesting in its own right.

Given $x_0\in\R^D$ and $r>0$, we write $B_r(x_0)$ to denote the open ball of radius $r$ at $x_0$, that is, $B_r(x_0)=\{x\in\R^D:|x-x_0|<r\}$.
We use $|\cdot|$ to denote the Euclidean norm.
Given positive sequences $(a_n)$ and $(b_n)$, we write $a_n\sim b_n$ to denote that $a_n/b_n\to 1$ as $n\to\infty$.
We write $\Normal(x\mid \mu,C)$ to denote the normal density with mean $\mu$ and covariance matrix $C$.

\begin{theorem}
\label{theorem:BVM}
Fix $\theta_0\in\R^D$ and let $\pi:\R^D\to\R$ be a probability density with respect to Lebesgue measure
such that $\pi$ is continuous at $\theta_0$ and $\pi(\theta_0)>0$.
Let $f_n:\R^D\to\R$ for $n\in\N$ and assume:
\begin{enumerate}
\item\label{item:BVM-rep} $f_n$ can be represented as
\begin{align}\label{equation:f_n}
f_n(\theta) = f_n(\theta_n) + \tfrac{1}{2}(\theta - \theta_n)^\T H_n (\theta - \theta_n) + r_n(\theta - \theta_n)
\end{align}
where $\theta_n\in\R^D$ such that $\theta_n\to\theta_0$, $H_n \in \R^{D\times D}$ symmetric such that $H_n \to H_0$ for some positive definite $H_0$, and $r_n:\R^D\to\R$ has the following property: there exist $\epsilon_0,c_0>0$ such that for all $n$ sufficiently large, for all $x\in B_{\epsilon_0}(0)$, 
we have $|r_n(x)|\leq c_0|x|^3$; and
\item\label{item:BVM-liminf} for any $\epsilon>0$,
$\liminf_n \inf_{\theta \in B_\epsilon(\theta_n)^c} \big(f_n(\theta) - f_n(\theta_n)\big)>0.$
\end{enumerate}
Then, defining $z_n =\int_{\R^D} \exp(-n f_n(\theta))\pi(\theta) d\theta$ and $\pi_n(\theta) = \exp(-n f_n(\theta))\pi(\theta)/z_n$, we have
\begin{align}\label{equation:concentration}
{\textstyle\int_{B_\epsilon(\theta_0)}} \pi_n(\theta) d\theta \xrightarrow[n\to\infty]{} 1  \text{ for all $\epsilon > 0$},
\end{align}
that is, $\pi_n$ concentrates at $\theta_0$, 
\begin{align}\label{equation:Laplace}
z_n \sim \frac{\exp(-n f_n(\theta_n))\pi(\theta_0)}{|\det H_0|^{1/2}}\Big(\frac{2\pi}{n}\Big)^{D/2}
\end{align}
as $n\to\infty$ (Laplace approximation), and
letting $q_n$ be the density of $\sqrt{n} (\theta-\theta_n)$ when $\theta\sim \pi_n$,
\begin{align}\label{equation:normality}
\int_{\R^D} \Big\vert q_n(x) -\Normal(x \mid 0, H_0^{-1})\Big\vert d x \xrightarrow[n\to\infty]{} 0,
\end{align}
that is, $q_n$ converges to $\Normal(0, H_0^{-1})$ in total variation.
\end{theorem}

See Section \ref{section:BVM-proof} for the proof. 
The virtue of Theorem~\ref{theorem:BVM} is not its technical depth -- indeed, it is fairly straightforward to prove 
using the generalized dominated convergence theorem.
Rather, the utility of this result is that it is formulated in such a way that it can be broadly applied to generalized posteriors.
In Section~\ref{section:previous-work}, we compare Theorem~\ref{theorem:BVM} to previous Bernstein--von Mises results.

Roughly speaking,
condition~\ref{item:BVM-rep} of Theorem~\ref{theorem:BVM} is that $f_n(\theta)$ can be approximated by a quadratic form in a neighborhood of $\theta_n$.
This is similar to a second-order Taylor expansion
where the constants in the bound on the remainder (namely, $\varepsilon_0$ and $c_0$) need to work for all $n$ sufficiently large;
however, unlike in Taylor's theorem, differentiability of $f_n$ is not assumed.
Since $\theta_n\to\theta_0$, the idea of condition \ref{item:BVM-rep} is that the log posterior density approaches a quadratic form near $\theta_0$.
The assumption that $H_0$ is positive definite is necessary to ensure that in the limit, the exponentiated quadratic form can be normalized to a probability density, namely, $\mathcal{N}(x\mid 0,H_0^{-1})$.
Note that in the special case of a correctly specified i.i.d.\ probability model,
$H_0$ typically coincides with the Fisher information matrix at $\theta_0$.

Condition~\ref{item:BVM-liminf} of Theorem~\ref{theorem:BVM} 
ensures that, asymptotically, the posterior puts negligible mass outside a neighborhood of $\theta_0$,
and thus, the locally normal part near $\theta_0$ is all that remains in the limit.
Condition~\ref{item:BVM-liminf} is stronger than necessary, but it is not clear how to 
adapt the usual probabilistic separation conditions (such as uniformly consistent tests) 
to generalized posteriors, especially since we seek almost sure convergence.

Throughout, we use the Euclidean--Frobenius norms on vectors $v\in\R^D$, matrices $M\in\R^{D\times D}$, and tensors $T\in\R^{D^3}$, 
that is, $|v| = (\sum_i v_i^2)^{1/2}$, $\|M\| = (\sum_{i,j} M_{i j}^2)^{1/2}$, and $\|T\|= (\sum_{i,j,k} T_{i j k}^2)^{1/2}$.
Convergence and boundedness for vectors, matrices, and tensors is defined with respect to these norms.
A collection of functions $h_n:E\to F$, where $F$ is a normed space, is \textit{uniformly bounded} if the set $\{\|h_n(x)\| : x\in E,\,n\in\N\}$ is bounded, and is \textit{pointwise bounded} if $\{\|h_n(x)\| : n\in\N\}$ is bounded for each $x\in E$.
Let $f'''(\theta)$ denote the tensor of third derivatives, that is,
$f'''(\theta) = (\frac{\partial^3 f}{\partial\theta_i\partial\theta_j\partial\theta_k}(\theta))_{i,j,k=1}^D\in\R^{D^3}$.


\begin{theorem}
\label{theorem:altogether}
Let $\Theta\subseteq\R^D$.  Let $E\subseteq\Theta$ be open (in $\R^D$) and bounded.
Fix $\theta_0\in E$ and 
let $\pi:\Theta\to\R$ be a probability density with respect to Lebesgue measure
such that $\pi$ is continuous at $\theta_0$ and $\pi(\theta_0)>0$.  Let $f_n:\Theta\to\R$ have continuous third derivatives on $E$.
Suppose $f_n \to f$ pointwise for some $f:\Theta\to\R$, $f''(\theta_0)$ is positive definite, and $(f_n''')$ is uniformly bounded on $E$.
If either of the following two conditions is satisfied:
\begin{enumerate}
\item\label{item:altogether-liminf} $f(\theta) > f(\theta_0)$ for all $\theta\in K\setminus\{\theta_0\}$ and
    $\liminf_n \inf_{\theta\in \Theta\setminus K} f_n(\theta)>f(\theta_0)$ for some compact $K\subseteq E$ with $\theta_0$ in 
    the interior of $K$, or
\item\label{item:altogether-convex} each $f_n$ is convex and $f'(\theta_0)=0$,
\end{enumerate}
then there is a sequence $\theta_n\to\theta_0$ such that $f_n'(\theta_n) = 0$ for all $n$ sufficiently large, $f_n(\theta_n)\to f(\theta_0)$,
Equation~\ref{equation:concentration} (concentration at $\theta_0$) holds, 
Equation \ref{equation:Laplace} (Laplace approximation) holds,
and Equation \ref{equation:normality} (asymptotic normality) holds,
where $H_0 = f''(\theta_0)$.
Further, condition \ref{item:altogether-convex} implies condition \ref{item:altogether-liminf} under the assumptions of the theorem.
\end{theorem}

See Section \ref{section:BVM-proof} for the proof.
While Theorem~\ref{theorem:BVM} is more general, Theorem~\ref{theorem:altogether} provides conditions that are easier to verify when applicable.
The set $E$ simply serves as a neighborhood of $\theta_0$ on which $f_n$ is well-behaved. 
The assumption that $f''(\theta_0)$ is positive definite ensures that $f$ is locally convex at $\theta_0$,
but not necessarily globally convex.
See Section~\ref{section:previous-work} for comparison with previous work. 
The following result is used in the proof of Theorem~\ref{theorem:altogether}.

\begin{theorem}\label{theorem:natural-sufficient}
Let $E\subseteq\R^D$ be open and convex, and let $\theta_0\in E$. Let $f_n:E\to\R$ have continuous third derivatives, and assume:
\begin{enumerate}
\item there exist $\theta_n\in E$ such that $\theta_n\to\theta_0$ and $f_n'(\theta_n) = 0$ for all $n$ sufficiently large,
\item $f_n''(\theta_0)\to H_0$ as $n\to\infty$ for some positive definite $H_0$, and
\item $(f_n''')$ is uniformly bounded.
\end{enumerate}
Then, letting $H_n = f_n''(\theta_n)$, condition \ref{item:BVM-rep} of Theorem \ref{theorem:BVM} is satisfied for all $n$ sufficiently large.
\end{theorem}

See Section \ref{section:BVM-proof} for the proof. 
The main tool used in the proof of Theorem \ref{theorem:altogether} is the following result, which provides somewhat more than we require.
A collection of functions $h_n:E\to F$, where $E$ and $F$ are subsets of normed spaces,
is \textit{equi-Lipschitz} if there exists $c>0$ such that for all
$n\in\N$, $x,y\in E$, we have $\|h_n(x) - h_n(y)\|\leq c\|x - y\|$.

\begin{theorem}[Regular convergence]
\label{theorem:regular-convergence}
Let $E\subseteq \R^D$ be open, convex, and bounded.
For $n\in\N$, let $f_n:E\to\R$ have continuous third derivatives, and suppose $(f_n''')$ is uniformly bounded.
If $(f_n)$ is pointwise bounded, then $(f_n)$, $(f_n')$, and $(f_n'')$ are all equi-Lipschitz and uniformly bounded.
If $f_n\to f$ pointwise for some $f:E\to\R$, then $f'$ and $f''$ exist, 
$f_n\to f$ uniformly, $f_n'\to f'$ uniformly, and $f_n''\to f''$ uniformly.
\end{theorem}

Note that if $f_n\to f$ pointwise then $(f_n)$ is pointwise bounded; thus, if $f_n\to f$ pointwise then we also get the equi-Lipschitz and uniform bounded result.
See Section \ref{section:regular-convergence-proof} for proof.

\section{Coverage}
\label{section:coverage}


For a generalized posterior to provide useful quantification of uncertainty, it is important that it be reasonably well-calibrated in terms of frequentist coverage.
Ideally, we would like $\Pi_n$ to have correct frequentist coverage in the sense that posterior credible sets of probability $\rho$ have frequentist coverage $\rho$.
Obviously, an arbitrarily chosen generalized posterior cannot be expected to have correct coverage.
Thus, in Theorem~\ref{theorem:coverage}, we provide simple conditions under which a generalized posterior has correct frequentist coverage, asymptotically.

Unfortunately, it seems that, like misspecified models, many common choices of generalized posterior do not exhibit correct coverage, even asymptotically.
In Section~\ref{section:cl-coverage}, we discuss why this problem occurs in the context of composite likelihood-based posteriors,
which is a very general class that includes nearly all of the examples in this paper.
Due to this, we do not apply our main coverage result (Theorem~\ref{theorem:coverage})
to the examples in Sections~\ref{section:standard-applications} and \ref{section:generalized-applications} for the simple reason that we do not expect it to hold,
except under correct specification or in special circumstances.
Nonetheless, we present the theorem here in order to help find those special circumstances when they do arise,
and to provide a foundation for future work that may generalize upon this result.
For instance, having correct coverage for each univariate component of the parameter, marginally, rather than jointly, is a less stringent property that would still be very useful in practice.
Alternatively, one could aim for conservative coverage, which would be more achievable as well.

To interpret Theorem~\ref{theorem:coverage}, we think of 
$\theta_n$ as a maximum generalized likelihood estimate, $\theta_0$ as the ``true'' parameter we want to cover, $\Pi_n$ as the generalized posterior distribution, 
$S_n$ as a credible set of asymptotic probability $\rho$, $Q_n$ as a centered and scaled version of $\Pi_n$, and $R_n$ as a centered and scaled version of $S_n$.
Roughly, the theorem says that if $Q_n$ converges in total variation to the asymptotic distribution of $-\sqrt{n}(\theta_n - \theta_0)$,
and $R_n$ converges pointwise, 
then asymptotically, $S_n$ contains the true parameter $100\rho$ percent of the time.
In other words, if the conditions of the theorem hold, then asymptotically, $\Pi_n$ has correct frequentist coverage in the sense that posterior credible sets of probability $\rho$ have frequentist coverage~$\rho$.

Typically, when things work out nicely, $\theta_n$ is $\sqrt{n}$-consistent and asymptotically normal and a BvM result holds for $\Pi_n$, 
in which case the result says that $\Pi_n$ has correct coverage asymptotically 
if the covariance matrices of these two normal distributions are equal.
In other words, if $\sqrt{n}(\theta_n - \theta_0) \xrightarrow[]{\mathrm{D}} \Normal(0,C_1)$
and $Q_n \to \Normal(0,C_2)$ in total variation distance, then 
$\Pi_n$ has correct asymptotic frequentist coverage if $C_1=C_2$. 
In this case, the only other condition is that $R_n$ converges to a set $R$ with finite nonzero Lebesgue measure,
because it is guaranteed that $Q(\partial R) = 0$.
(Note that if $X\sim \Normal(0,C_1)$ then $-X\sim\Normal(0,C_1)$ also.)
This result is precisely what one would expect; thus, the purpose of the theorem is to make this rigorous under easy-to-verify conditions.

We give $\R^D$ the Euclidean topology and the resulting Borel sigma-algebra, $\mathcal{B}$,
and we use $m(\cdot)$ to denote Lebesgue measure on $\R^D$.
We write $\partial R$ to denote the boundary of a set $R\in\R^D$, that is, $\partial R = \bar{R}\setminus R^\circ$,
where $\bar{R}$ is the closure and $R^\circ$ is the interior of $R$.
Given $R,R_1,R_2,\ldots\subseteq \R^D$, we write $R_n \to R$ to denote that for all $x\in\R^D$, $\I(x\in R_n) \to \I(x\in R)$ as $n\to\infty$.
Define $d(x,A) = \inf_{y\in A} \|x - y\|$ for $x\in \R^D$ and $A\subseteq\R^D$.

\begin{theorem}
\label{theorem:coverage}
Let $\theta_1,\theta_2,\ldots \in\R^D$ be a sequence of random vectors, and let $\theta_0\in\R^D$ be fixed.
Let $\Pi_1,\Pi_2,\ldots$ be a sequence of random probability measures on $\R^D$,
possibly dependent on $\theta_1,\theta_2,\ldots$.
Let $S_1,S_2,\ldots\subseteq\R^D$ be a sequence of random convex measurable sets such that $\Pi_n(S_n) \xrightarrow[]{\mathrm{a.s.}} \rho$
for some fixed $\rho\in(0,1)$.
For $A\in\mathcal{B}$, define $Q_n(A) = \int \I\big(\sqrt{n}(\theta - \theta_n)\in A\big)\Pi_n(d\theta)$ 
and define $R_n = \{\sqrt{n}(\theta - \theta_n) : \theta\in S_n\}$.
Suppose there is a fixed probability measure $Q$ and a fixed set $R\subseteq\R^D$ such that
\begin{enumerate}
\item\label{item:coverage1} $-\sqrt{n}(\theta_n - \theta_0) \xrightarrow[]{\mathrm{D}} Q$ as $n\to\infty$
(where $\xrightarrow[]{\mathrm{D}}$ denotes convergence in distribution),
\item\label{item:coverage2} $\sup_{A\in\mathcal{B}} |Q_n(A) - Q(A)| \xrightarrow[]{\mathrm{a.s.}} 0$ as $n\to\infty$
(that is, $Q_n \xrightarrow[]{\mathrm{a.s.}} Q$ in total variation),

\item\label{item:coverage3} $R_n \xrightarrow[]{\mathrm{a.s.}} R$ as $n\to\infty$ (that is, a.s., for all $x\in\R^D$, $\I(x\in R_n) \to \I(x\in R)$), and
\item\label{item:coverage4} $Q(\partial R) = 0$ and $0 < m(R) < \infty$, where $m$ denotes Lebesgue measure on $\R^D$.
\end{enumerate}
Then $\Pr(\theta_0\in S_n)\to \rho$ as $n\to\infty$.
\end{theorem}

See Section \ref{section:coverage-proof} for the proof. 
If $Q$ has a density with respect to Lebesgue measure, then the condition that $Q(\partial R) = 0$ is automatically satisfied, since 
the assumptions imply that $R$ is convex and thus $m(\partial R) = 0$.
In Theorem~\ref{theorem:coverage}, the assumption that the confidence sets $S_n$ are convex is not essential. 
The only place that convexity is used is to ensure that the conclusion of Lemma~\ref{lemma:squeeze} holds.
Indeed, Theorem~\ref{theorem:coverage} still holds if $S_1,S_2,\ldots$ are not assumed to be convex and
condition~\ref{item:coverage3} is replaced by the conclusion of Lemma~\ref{lemma:squeeze} (that is,
for any $\epsilon > 0$, if $A = \{x\in \R^D : d(x,R^c) > \epsilon \}$ and $B = \{x\in\R^D : d(x,R) \leq \epsilon\}$
then for all $n$ sufficiently large, $A\subseteq R_n\subseteq B$). 
We chose to state the theorem in this way because credible sets are often convex by construction, and
pointwise convergence of $R_n$ to $R$ is considerably easier to verify than the conclusion of Lemma~\ref{lemma:squeeze}.  
The following lemmas are used in the proof of Theorem~\ref{theorem:coverage}, and may be useful in their own right.
See Section \ref{section:coverage-proof} for their proofs.

\begin{lemma}
\label{lemma:convergent-set}
Let $X_1,X_2,\ldots\in\R^D$ be random vectors such that $X_n\xrightarrow[]{\mathrm{D}} X$ for some random vector $X$.
Let $R_1,R_2,\ldots\subseteq\R^D$ be random convex measurable sets, possibly dependent on $X_1,X_2,\ldots$.
Assume there exists some fixed $R \subseteq \R^D$ 
with $0 < m(R) < \infty$ and $\Pr(X \in \partial R) = 0$
such that $R_n \to R$ almost surely as $n\to\infty$.
Then $\Pr(X_n \in R_n) \to \Pr(X\in R)$ as $n\to\infty$.
\end{lemma}

The probability $\Pr(X_n\in R_n)$ should be interpreted as $\int \I(X_n(\omega)\in R_n(\omega)) P(d\omega)$,
that is, $X_n$ and $R_n$ are jointly integrated over and $\Pr(X_n\in R_n)$ is a non-random quantity.

\begin{lemma}
\label{lemma:squeeze}
Let $R_1,R_2,\ldots\subseteq\R^D$ be convex sets.
Assume $R_n \to R$ for some $R \subseteq \R^D$ with $0 < m(R) < \infty$.
For any $\epsilon > 0$,
if $A = \{x\in \R^D : d(x,R^c) > \epsilon \}$ and $B = \{x\in\R^D : d(x,R) \leq \epsilon\}$
then for all $n$ sufficiently large, $A\subseteq R_n\subseteq B$.
\end{lemma}


\section{Composite likelihood-based posteriors}
\label{section:composite-likelihood}

Composite likelihoods (CLs) \citep{lindsay1988composite} represent a large class of generalized likelihoods 
that encompasses essentially all of the examples in Sections~\ref{section:standard-applications} and \ref{section:generalized-applications}.
The theory of maximum composite likelihood estimation is well-established \citep{lindsay1988composite,molenberghs2005models,varin2011overview}.
Theoretical results for CL-based generalized posteriors have been provided 
\citep{pauli2011bayesian,ribatet2012bayesian,ventura2016pseudo,greco2008robust,lazar2003bayesian}, subject to the caveats discussed in the introduction.
The purpose of this section is to discuss how these previous results on CL-based generalized posteriors, or CL-posteriors for short,
can be strengthened using our results in Sections~\ref{section:concentration} to \ref{section:coverage}.
Roughly speaking, CL-posteriors derived from a correctly specified model can generally be expected to be consistent,
but not necessarily correctly calibrated with respect to frequentist coverage.


Let $y$ denote the full data set, which may take any form such as a sequence, a graph, a database, or any other data structure.
Suppose $\{P_\theta : \theta\in\Theta\}$ is an assumed model for the distribution of $y$ given $\theta$, where $\Theta\subseteq\R^D$. 
For $j=1,\ldots,k$, suppose $s_j(y)$ and $t_j(y)$ are functions of the data
and, when $Y\sim P_\theta$, suppose the conditional distribution of $s_j(Y)$ given $t_j(Y)$ has density $p_\theta(s_j | t_j)$ with respect to a common dominating measure $\lambda_j$ for all values of $\theta$ and $t_j$.
Define the \textit{composite likelihood} \citep{lindsay1988composite},
$$ \mathcal{L}^\mathrm{CL}(\theta) = \prod_{j = 1}^k p_\theta(s_j | t_j). $$
A few examples are given here and in Section~\ref{section:generalized-applications}; see \citet{varin2011overview} for more examples.

\begin{example}[i.i.d.\ likelihood]
If $y=(y_1,\ldots,y_n)$, $s_j(y) = y_j$, and $t_j(y) = 0$, then $\mathcal{L}^\mathrm{CL}(\theta) = \prod_{j=1}^n p_\theta(y_j)$
is simply the likelihood for an i.i.d.\ model.
\end{example}

\begin{example}[pseudolikelihood]
If $y=(y_1,\ldots,y_n)$, $s_j(y) = y_j$, and $t_j(y) = y_{-j} := (y_1,\ldots,y_{j-1},y_{j+1},\ldots,y_n)$, then $\mathcal{L}^\mathrm{CL}(\theta) = \prod_{j=1}^n p_\theta(y_j | y_{-j})$ is a pseudolikelihood \citep{besag1975statistical}.
\end{example}

\begin{example}[restricted likelihood]
If $k=1$, $t_1(y)=0$, and $s_1(y)$ is an insufficient statistic, 
then $\mathcal{L}^\mathrm{CL}(\theta)$ is a restricted likelihood \citep{lewis2014bayesian}.
For instance, if $s_1(y)$ consists of ranks or selected quantiles, then $\mathcal{L}^\mathrm{CL}(\theta)$
is a rank likelihood \citep{pettitt1983likelihood,hoff2007extending}
or a quantile-based likelihood \citep{doksum1990consistent}, respectively.
\end{example}

Due to the special structure of composite likelihoods,
one can make some general observations about CL-posteriors
of the form $\pi_n(\theta) \propto \mathcal{L}^\mathrm{CL}(\theta)\pi(\theta)$.
First, a reassuring property is that if the model is correctly specified, 
then CL-posteriors are consistent under fairly general conditions; we discuss this next.

\subsection{Consistency of CL-posteriors under correct specification}
\label{section:cl-consistency}

Throughout this article, we make no assumption of model correctness in the main results (Sections~\ref{section:concentration} to \ref{section:coverage})
or the applications (Sections~\ref{section:standard-applications} and \ref{section:generalized-applications}).
However, for interpretability, it is important to have a guarantee of consistency if the assumed model is correct or at least partially correct.
Here, we show that in many cases of interest, if the model is correctly specified---or at least, if the conditional densities $p_\theta(s_j|t_j)$ are correctly specified---then the CL-posterior concentrates at the true parameter.
The analogue of this result for maximum CL estimators is well-known \citep{lindsay1988composite,varin2011overview};
also see \citet{pauli2011bayesian} and \citet{ribatet2012bayesian}.

First, observe that if $Y\sim P_{\theta_0}$, $S_j = s_j(Y)$, and $T_j = t_j(Y)$, then for all $\theta\in\Theta$,
\begin{align}
\label{equation:cl-ineq}
\E\big(\log p_{\theta_0}(S_j | T_j)\big) \geq \E\big(\log p_{\theta}(S_j | T_j)\big)
\end{align}
because the conditional relative entropy $\E\big(\log(p_{\theta_0}(S_j | T_j)/p_\theta(S_j | T_j))\big)$ is nonnegative;
this is referred to as the information inequality by \citet{lindsay1988composite}.
Now, suppose that for each $n\in\{1,2,\ldots\}$, we have a data set $Y^{n}$, model $\{P_\theta^{n} : \theta\in\Theta\}$ (where $\Theta$ does not depend on $n$), and functions $s_j^{n}$, $t_j^{n}$ for $j=1,\ldots,k_n$.
Further, suppose the assumed model is correct, 
such that $Y^{n}\sim P_{\theta_0}^{n}$ where the true parameter $\theta_0$ is shared across all $n$.
Define
$$ f_n(\theta) = -\frac{1}{n} \log \mathcal{L}_n^\mathrm{CL}(\theta) = -\frac{1}{n} \sum_{j = 1}^{k_n} \log p_\theta^{n}(S_j^{n} | T_j^{n}) $$
and $\pi_n(\theta) \propto  \exp(-n f_n(\theta))\pi(\theta) = \mathcal{L}_n^\mathrm{CL}(\theta) \pi(\theta)$
where $S_j^{n} = s_j^{n}(Y^{n})$ and $T_j^{n} = t_j^{n}(Y^{n})$.
In many cases of interest (see Sections~\ref{section:standard-applications} and \ref{section:generalized-applications}), 
we have that with probability $1$, for all $\theta\in\Theta$, 
$\lim_{n\to\infty} f_n(\theta) = f(\theta)$ where $f(\theta) = \lim_{n\to\infty} \E\big(f_n(\theta)\big)$.
Then, by Equation~\ref{equation:cl-ineq}, $f(\theta_0) \leq f(\theta)$ for all $\theta\in\Theta$,
in other words, $\theta_0$ is a minimizer of $f$.
Further, in many cases, $f$ has a unique minimizer, and $\pi_n$ concentrates at the unique minimizer;
in particular, this holds if the conditions of Theorem~\ref{theorem:concentration} or Theorem~\ref{theorem:altogether} are met.
Therefore, in such cases, the CL-posterior $\pi_n$ concentrates at the true parameter, $\theta_0$.

\subsection{Coverage of CL-posteriors under correct specification}
\label{section:cl-coverage}

Although CL-posteriors have appealing consistency properties, 
they do not generally have correct asymptotic frequentist coverage, except in special circumstances \citep{pauli2011bayesian,ribatet2012bayesian}.
Continuing in the notation of Section~\ref{section:cl-consistency},
suppose $Y^n\sim P_{\theta_0}^{n}$, 
let $\pi_n(\theta) \propto \exp(-n f_n(\theta))\pi(\theta) = \mathcal{L}_n^\mathrm{CL}(\theta)\pi(\theta)$ be the CL-posterior, and 
let $\theta_n = \argmax_\theta \mathcal{L}_n^\mathrm{CL}(\theta) = \argmin_\theta f_n(\theta)$ be the maximum composite likelihood estimator.
If Theorem~\ref{theorem:altogether} applies with probability $1$, then $Q_n \xrightarrow[]{\mathrm{a.s.}} \Normal(0,H_0^{-1})$ in total variation distance, 
where $H_0 = f''(\theta_0)$ and $Q_n$ is the distribution of $\sqrt{n} (\theta-\theta_n)$ when $\theta\sim\pi_n$.
This strengthens previous BvM results for CL-posteriors by showing almost sure convergence (rather than convergence in probability)
with respect to total variation distance (rather than in the weak topology).

To use Theorem~\ref{theorem:coverage}, we also need to know the asymptotic distribution of $\theta_n$.
The asymptotics of $\theta_n$ are well-known \citep{lindsay1988composite,varin2011overview},
but for completeness we provide an informal derivation (see below).
Define $G_j^n = \nabla_\theta\big\vert_{\theta=\theta_0}  \log p_\theta^{n}(S_j^{n} | T_j^{n})$.
It turns out that $-\sqrt{n}(\theta_n - \theta_0) \approx  \Normal(0, A_n^{-1} J_n A_n^{-1})$
under regularity conditions, 
where $A_n = \frac{1}{n}\sum_{j=1}^{k_n} \mathrm{Cov}(G_j^n)$
and $J_n = \frac{1}{n}\mathrm{Cov}\big(\sum_{j=1}^{k_n} G_j^n\big)$.
Typically, $A_n \to H_0$ and $J_n \to J_0$ for some $J_0$, so that
$$-\sqrt{n}(\theta_n - \theta_0) \xlongrightarrow[]{\mathrm{D}} \Normal(0, H_0^{-1} J_0 H_0^{-1}).$$

Hence, under typical conditions, the asymptotic distribution of $-\sqrt{n}(\theta_n - \theta_0)$ and the limit of $Q_n$ are the same if and only if $H_0 = J_0$. 
Therefore, under these conditions, if $H_0 = J_0$ then 
the CL-posterior $\pi_n$ has correct asymptotic frequentist coverage by Theorem~\ref{theorem:coverage}.
For instance, if for each $n$, $G_1^n,\ldots,G_{k_n}^n$ are pairwise uncorrelated, 
then $A_n = J_n$ and hence $H_0 = J_0$.
However, in many cases of interest, $H_0\neq J_0$ and the CL-posterior needs to be affinely transformed to have correct coverage
\citep{ribatet2012bayesian,pauli2011bayesian,friel2012bayesian,stoehr2015calibration};
also see \citet{williams2018bayesian} for a similar technique in survey sampling.

For completeness, here we provide a rough sketch of the derivation of the asymptotic distribution of $\theta_n$; 
see \citet{lindsay1988composite} and \citet{varin2011overview}.
By a first-order Taylor approximation applied to each entry of $f_n'(\theta)\in\R^D$, when $\theta_n$ is near $\theta_0$ we have
$0 = f_n'(\theta_n) \approx f_n'(\theta_0) + f_n''(\theta_0) (\theta_n - \theta_0)$,
and thus, $-\sqrt{n}(\theta_n - \theta_0) \approx f_n''(\theta_0)^{-1} (\sqrt{n} f_n'(\theta_0))$,
assuming $f_n''(\theta_0)\in\R^{D\times D}$ exists and is invertible and the error terms are negligible.
When $n$ is large, we typically have $f_n''(\theta_0) \approx \E f_n''(\theta_0)$ (for instance, due to a law of large numbers result), and thus,
$f_n''(\theta_0) \approx \E f_n''(\theta_0) = \frac{1}{n}\sum_{j=1}^{k_n} \E (G_j^n {G_j^n}^\T) = \frac{1}{n}\sum_{j=1}^{k_n} \mathrm{Cov}(G_j^n) = A_n$
since $\E(G_j^n) = 0$ 
and $\E \big(\nabla^2_\theta\big\vert_{\theta=\theta_0} \log p_\theta^n(S_j^{n} | T_j^{n})\big) 
= -\E (G_j^n {G_j^n}^\T)$,
as long as we can interchange the order of integrals and derivatives.
Further, assuming a central limit theorem holds,
$\sqrt{n} f_n'(\theta_0) = -\frac{1}{\sqrt{n}} \sum_{j=1}^{k_n} G_j^n \approx \Normal(0, J_n)$
where $J_n = \frac{1}{n}\mathrm{Cov}\big(\sum_{j=1}^{k_n} G_j^n\big)$.
Thus, under appropriate conditions, $-\sqrt{n}(\theta_n - \theta_0) \approx \Normal(0, A_n^{-1} J_n A_n^{-1})$.

\section{Applications to standard posteriors}
\label{section:standard-applications}

In this section, we illustrate how our results can be used to easily prove posterior concentration, the Laplace approximation,
and asymptotic normality for standard models such as exponential families, linear regression, and generalized linear models including logistic regression and Poisson regression. 
We do not assume that the model is correctly specified; thus, this section can be compared to the misspecified setting of 
\citet{kleijn2012bernstein}.
Even in these standard models, our results go beyond the existing theory of \citet{kleijn2012bernstein} by showing almost sure convergence
and employing conditions that we believe are easier to verify; see Section~\ref{section:previous-work-misspecification} for a detailed comparison.
Further, these ``toy'' examples serve to illustrate our general results in familiar settings, 
enabling one to compare our assumptions with commonly used assumptions for these models.


\subsection{Exponential families}
\label{section:exponential-families}


Consider an exponential family with density $q(y|\eta) =\exp(\eta^\T s(y) -\kappa(\eta))$ with respect to a sigma-finite Borel measure
$\lambda$ on $\Y\subseteq\R^d$ where $s:\Y\to\R^k$, $\eta\in\mathcal{E}\subseteq\R^k$, and 
$\kappa(\eta) =\log\int_\Y\exp(\eta^\T s(y))\lambda(d y)$. 
Any exponential family on $\R^d$ can be put in this form by choosing $\lambda$ appropriately and possibly reparametrizing to $\eta$.
Let $Q_\eta(E) =\int_E q(y|\eta) \lambda(d y)$ and denote $\E_\eta s(Y) = \int_\Y s(y) Q_\eta(d y)$.
For any $m\in\N$, we give $\R^m$ the Euclidean metric and the resulting Borel sigma-algebra unless otherwise specified.

\begin{condition}
\label{condition:expfam}
Assume $q(y|\eta)$ is of the form above, 
$\mathcal{E} =\{\eta\in\R^k: |\kappa(\eta)|<\infty\}$, $\mathcal{E}$ is open, $\mathcal{E}$ is nonempty, and $\eta \mapsto Q_\eta$ is one-to-one
(that is, $\eta$ is identifiable).
\end{condition}

\begin{theorem}[Exponential families]
\label{theorem:expfams}
Consider a family $q(y|\eta)$ satisfying Condition~\ref{condition:expfam}.
Suppose $Y_1,Y_2,\ldots\in\Y$ are i.i.d.\ random vectors such that 
$\E s(Y_i) = \E_{\theta_0} s(Y)$ for some $\theta_0\in\Theta:=\mathcal{E}$.
Then for any open ball $E$ such that $\theta_0\in E$ and $\bar E\subseteq\Theta$, $f_n(\theta) := -\frac{1}{n}\sum_{i = 1}^n \log q(Y_i|\theta)$
satisfies the conditions of Theorem~\ref{theorem:altogether} with probability~$1$.
\end{theorem}

Condition~\ref{condition:expfam} is that the exponential family is full, regular, nonempty, identifiable, and in natural form;
these are standard conditions that hold for most commonly used exponential families \citep{hoffman1994probability,miller2014inconsistency}.
Recall that the maximum likelihood estimate (MLE) is obtained by matching the expected sufficient statistics to the observed sufficient statistics.
Thus, the assumption that $\E s(Y_i) = \E_{\theta_0} s(Y)$ for some $\theta_0$ is simply assuming that this moment matching is possible, asymptotically.
In many cases, this holds automatically since the moment space $\mathcal{M} := \{\E_\theta s(Y) : \theta\in\Theta\}$ is often equal to the full set of possible values of $\E s(Y_i)$, due to the fact that $\mathcal{M}$ is convex \citep[e.g.,][Prop.\ 19]{miller2014inconsistency}.
Thus, while exceptions can occur, the result holds very generally.

\vspace{1em}
\begin{proof}{\bf of Theorem~\ref{theorem:expfams}}~
Note that $f_n(\theta) = \kappa(\theta) -\theta^\T S_n$ where $S_n =\frac{1}{n}\sum_{i = 1}^n s(Y_i)$.
By standard exponential family theory \citep[e.g.,][Prop.\ 19]{miller2014inconsistency}, 
$\kappa$ is $C^\infty$ (that is, $\kappa$ has continuous derivatives of all order), $\kappa$ is convex on $\Theta$, 
$\kappa'(\theta) =\E_\theta s(Y)$, and $\kappa''(\theta)$ is symmetric positive definite for all $\theta\in\Theta$.
Let $s_0 =\E s(Y_i)$. Since $s_0 = \E s(Y_i) = \E_{\theta_0} s(Y) = \kappa'(\theta_0)$ and $\kappa'(\theta_0)$ is finite
(because $\kappa$ is $C^\infty$), $S_n\to s_0$ with probability $1$ by the strong law of large numbers.
Thus, letting $f(\theta) =\kappa(\theta) -\theta^\T s_0$, we have that with probability $1$, for all $\theta\in\Theta$,
$f_n(\theta) =\kappa(\theta) -\theta^\T S_n \to \kappa(\theta) -\theta^\T s_0 = f(\theta)$.
(Note that due to the almost sure convergence of the sufficient statistics, we not only have that for all $\theta$, with probability $1$,  $f_n(\theta) \to f(\theta)$,
but we have the stronger consequence that with probability $1$, for all $\theta$, $f_n(\theta) \to f(\theta)$,
which is needed for Theorem~\ref{theorem:altogether} to apply.)
Let $E$ be an open ball such that $\theta_0\in E$ and $\bar E\subseteq\Theta$.
Then $\kappa'''(\theta)$ is bounded on $\bar E$, since $\kappa'''(\theta)$ is continuous and $\bar E$ is compact.
Hence, $(f_n''')$ is uniformly bounded on $E$ because $f_n'''(\theta) = \kappa'''(\theta)$.
Therefore, with probability $1$, $f_n\to f$ pointwise, $f_n$ is convex and has continuous third derivatives on $\Theta$, 
$f'(\theta_0) = \kappa'(\theta_0) - s_0 = 0$, $f''(\theta_0) = \kappa''(\theta_0)$ is positive definite, 
and $(f_n''')$ is uniformly bounded on $E$.
\end{proof}

\subsection{Generalized linear models (GLMs)}

First, we state a general theorem for GLMs, then we show how it applies to commonly used GLMs.
Consider a regression model of the form $p(y_i\mid\theta,x_i) \propto_\theta q(y_i\mid\theta^\T x_i)$
for covariates $x_i\in\X\subseteq\R^D$ and coefficients $\theta\in \Theta\subseteq\R^D$,
where $q(y|\eta)=\exp(\eta s(y) - \kappa(\eta))$ is a one-parameter exponential family satisfying Condition~\ref{condition:expfam}.
Note that the proportionality is with respect to $\theta$, not $y_i$.
Assume $\Theta$ is open, $\Theta$ is convex, and $\theta^\T x\in\mathcal{E}$ for all $\theta\in \Theta$, $x\in\X$.

\begin{theorem}[GLMs]
\label{theorem:GLMs}
Suppose $(X_1,Y_1),(X_2,Y_2),\ldots\in\X\times\Y$ i.i.d.\ such that:
\begin{enumerate}
\item\label{item:GLMs-minimum} $f'(\theta_0)=0$ for some $\theta_0 \in \Theta$, where $f(\theta) = -\E \log q(Y_i\mid\theta^\T X_i)$,
\item\label{item:GLMs-moments} $\E|X_i s(Y_i)| < \infty$ and $\E|\kappa(\theta^\T X_i)| < \infty$ for all $\theta\in \Theta$,
\item\label{item:GLMs-linearly} for all $a\in\R^D$, if $a^\T X_i \overset{\mathrm{a.s.}}{=} 0$ then $a = 0$, and
\item\label{item:GLMs-triple} there is an open ball $E\subseteq\R^D$ such that $\theta_0\in E$, $\bar E\subseteq \Theta$, and
        for all $j,k,\ell\in\{1,\ldots,D\}$, $\E\big(\sup_{\theta\in\bar E}|\kappa'''(\theta^\T X_i) X_{i j} X_{i k} X_{i \ell}|\big)<\infty$. 
\end{enumerate}
Then for any open ball $E$ satisfying condition~\ref{item:GLMs-triple}, $f_n(\theta) := -\frac{1}{n}\sum_{i = 1}^n \log q(Y_i\mid\theta^\T X_i)$
satisfies the conditions of Theorem \ref{theorem:altogether} with probability $1$.
\end{theorem}

Condition~\ref{item:GLMs-minimum} of Theorem~\ref{theorem:GLMs} is essentially that the MLE exists, asymptotically.
Condition~\ref{item:GLMs-linearly} is that the support of the covariate vector $X_i$ is not contained in any proper subspace of $\R^D$;
this is necessary to ensure identifiability.
When $\E X_i X_i^\T$ exists and is finite, condition~\ref{item:GLMs-linearly} is equivalent to the assumption that $\E X_i X_i^\T$ is non-singular, which is commonly used to ensure identifiability in GLMs \citep[Example 16.8]{van2000asymptotic};
alternatively, it is sometimes assumed that $\frac{1}{n}\sum_{i=1}^n X_i X_i^\T$ is non-singular for all $n$ sufficiently large \citep{fahrmeir1985consistency}.
Conditions~\ref{item:GLMs-moments} and \ref{item:GLMs-triple} are moment conditions that are fairly easy to work with;
for instance, if the covariates are bounded and $\E s(Y_i)$ exists, then conditions~\ref{item:GLMs-moments} and \ref{item:GLMs-triple}
are always satisfied since $\kappa$ is $C^\infty$ smooth. 

For comparison, traditional theorems on the asymptotic normality of the MLE in a GLM typically
assume the model is correctly specified (whereas we do not), and they
assume conditions on the observed Fisher information $n f_n''(\theta)$, 
such as divergence of the smallest eigenvalue of $n f_n''(\theta_0)$ and bounds on the variability of $f_n''(\theta)$ near $\theta_0$ as $n\to\infty$
\citep{fahrmeir1985consistency}.
These Fisher information conditions are more closely analogous to condition~\ref{item:BVM-rep} of Theorem~\ref{theorem:BVM},
which is implied by our result in Theorem~\ref{theorem:GLMs}.
On the other hand, we show almost sure convergence of the posterior in TV distance,
whereas \citet{fahrmeir1985consistency} only show convergence in distribution of the MLE.

\vspace{1em}
\begin{proof}{\bf of Theorem~\ref{theorem:GLMs}}~
For all $\theta\in \Theta$, 
$f_n(\theta) = \frac{1}{n}\sum_{i = 1}^n \kappa(\theta^\T X_i)-\theta^\T S_n$
where $S_n =\frac{1}{n}\sum_{i = 1}^n X_i s(Y_i)$.
Thus, $f_n(\theta)$ is $C^\infty$ on $\Theta$ by the chain rule, since $\kappa(\eta)$ is $C^\infty$ on $\mathcal{E}$.
Further, $f_n(\theta)$ is convex since $\kappa(\eta)$ is convex.
Noting that 
$$ f(\theta) =  -\E \log q(Y_i\mid\theta^\T X_i) = \E(\kappa(\theta^\T X_i)) -\theta^\T \E(X_i s(Y_i)), $$
the assumed moment conditions (\ref{item:GLMs-moments}) ensure that for all $\theta\in \Theta$, with probability $1$, $f_n(\theta) \to f(\theta)$.
This implies that with probability $1$, for all $\theta\in \Theta$, $f_n(\theta) \to f(\theta)$, by the following argument.
For any countable set $C\subseteq \Theta$, we have that with probability $1$, for all $\theta\in C$, $f_n(\theta) \to f(\theta)$.
Hence, letting $C$ be a countable dense subset of $\Theta$, and using the fact that each $f_n$ is convex, we have that with probability $1$,
the limit $\tilde f(\theta) := \lim_n f_n(\theta)$ exists and is finite for all $\theta\in \Theta$ 
and $\tilde f$ is convex \citep[Theorem 10.8]{rockafellar1970convex}. Since $f$ is also convex, then $\tilde f$ and $f$ are continuous functions
\citep[Theorem 10.1]{rockafellar1970convex} that agree on a dense subset, so they are equal.

Choose $E$ according to condition~\ref{item:GLMs-triple}.
We show that with probability $1$, $(f_n''')$ is uniformly bounded on $E$.
Fix $j,k,\ell\in\{1,\ldots,D\}$, and define $T(\theta,x) = \kappa'''(\theta^\T x) x_j x_k x_\ell$ for $\theta\in \Theta$, $x\in\X$.
For all $x\in\X$, $\theta\mapsto T(\theta,x)$ is continuous, and for all $\theta\in \Theta$, $x\mapsto T(\theta,x)$ is measurable.
Since
$f_n'''(\theta)_{j k \ell} = \frac{1}{n}\sum_{i = 1}^n T(\theta,X_i)$,
condition~\ref{item:GLMs-triple} implies that with probability~$1$,
$(f_n'''(\theta)_{j k \ell})$ is uniformly bounded on $\bar E$, by the uniform law of large numbers \citep[Theorem 1.3.3]{ghosh2003bayesian}.
Letting $C_{j k \ell}(X_1,X_2,\ldots)$ be such a uniform bound for each $j,k,\ell$,
we have that with probability $1$, for all $n\in\N$, $\theta\in\bar E$,
$\|f_n'''(\theta)\|^2 = \sum_{j,k,\ell} f_n'''(\theta)_{j k \ell}^2 \leq \sum_{j,k,\ell} C_{j k \ell}(X_1,X_2,\ldots)^2 < \infty$.
Thus, $(f_n''')$ is a.s.\ uniformly bounded on $\bar E$, and hence on $E$.

By Theorem \ref{theorem:regular-convergence}, 
$f''(\theta_0) \overset{\mathrm{a.s.}}{=} \lim_{n\to\infty} f_n''(\theta_0) = \lim \frac{1}{n}\sum_{i=1}^n \kappa''(\theta_0^\T X_i) X_i X_i^\T$.
Since this limit exists and is finite almost surely, then by the strong law of large numbers, the limit must be equal to the expectation \citep[Theorem 4.23]{kallenberg2002foundations},
that is, $f''(\theta_0) = \E\big(\kappa''(\theta_0^\T X_i) X_i X_i^\T\big)$.
Thus, $f''(\theta_0)$ is positive definite, since for all nonzero $a\in\R^D$,
$a^\T f''(\theta_0) a =\E\big(\kappa''(\theta_0^\T X_i) a^\T X_i X_i^\T a\big) > 0$,
by the fact that $\kappa''(\eta)>0$ for all $\eta\in\mathcal{E}$ and by condition~\ref{item:GLMs-linearly}, $a^\T X_i X_i^\T a =|a^\T X_i|^2$ is strictly positive with positive probability.
\end{proof}

\subsubsection{Linear regression}

The linear regression model is $p(y_i\mid\theta,x_i) = \Normal(y_i\mid\theta^\T x_i, \sigma^2)$
for $y_i\in\Y:=\R$, $x_i\in\X := \R^D$, and $\theta\in \Theta := \R^D$. 
The model can equivalently be written as $p(y_i\mid \theta,x_i) \propto_\theta q(y_i\mid\theta^\T x_i)$ 
where $q(y|\eta) := \exp(\eta s(y) - \kappa(\eta))$ is a density with respect to $\lambda(d y) = \Normal(y\mid 0,\sigma^2)d y$
for $y\in\Y$ and $\eta\in\mathcal{E}:=\R$,
by defining $s(y) = y/\sigma^2$ and $\kappa(\eta) = \eta^2 / (2\sigma^2)$.

\begin{theorem}[Linear regression]
\label{theorem:linear-regression}
Suppose $(X_1,Y_1),(X_2,Y_2),\ldots\in\X\times\Y$ i.i.d.\ such that:
\begin{enumerate}
\item\label{item:linreg-moments} $\E|X_i Y_i| < \infty$, $\E\|X_i X_i^\T\| < \infty$, and
\item\label{item:linreg-linearly} for all $a\in\R^D$, if $a^\T X_i \overset{\mathrm{a.s.}}{=} 0$ then $a = 0$.
\end{enumerate}
Then $\theta_0 := (\E X_i X_i^\T)^{-1} \E X_i Y_i$ is well-defined and 
for any open ball $E$ such that $\theta_0\in E$, 
$f_n(\theta) := -\frac{1}{n}\sum_{i = 1}^n \log q(Y_i\mid\theta^\T X_i)$
satisfies the conditions of Theorem \ref{theorem:altogether} with probability $1$.
\end{theorem}

Condition~\ref{item:linreg-moments} is necessary to ensure that $\theta_0 := (\E X_i X_i^\T)^{-1} \E X_i Y_i$ is well-defined
and condition~\ref{item:linreg-linearly} is necessary to ensure identifiability, as in the case of Theorem~\ref{theorem:GLMs}.
Since $\kappa(\eta) = \eta^2 / (2\sigma^2)$, we have $f_n''(\theta) = \frac{1}{n}\sum_{i=1}^n X_i X_i^\T / \sigma^2$.
Thus, for the traditional MLE conditions of \citet{fahrmeir1985consistency},
it is not necessary to bound the variability of $f_n''(\theta)$, since $f_n''(\theta)$ does not depend on $\theta$ in the linear regression model.
Hence, these traditional MLE conditions reduce to assuming divergence of the smallest eigenvalue of $\sum_{i=1}^n X_i X_i^\T$,
which can be shown to be equivalent to condition~\ref{item:linreg-linearly} of Theorem~\ref{theorem:linear-regression}
when $X_1,X_2,\ldots$ are i.i.d.

\vspace{1em}
\begin{proof}{\bf of Theorem~\ref{theorem:linear-regression}}~
For any random vector $Z\in\R^k$, $\E|Z|<\infty$ if and only if $\E Z$ exists and is finite; likewise for matrices and tensors.
Thus, $\E X_i Y_i$ and $\E X_i X_i^\T$ exist and are finite. 
Further, $\E X_i X_i^\T$ is positive definite (and hence, invertible) since 
for all nonzero $a\in\R^D$, $a^\T (\E X_i X_i^\T) a = \E|a^\T X_i|^2 > 0$.
Condition~\ref{condition:expfam} is easily checked: $\mathcal{E} = \{\eta\in\R : |\kappa(\eta)| < \infty\}$ since $\eta^2/(2\sigma^2) < \infty$ for all $\eta\in\R$, $\mathcal{E}$ is open and nonempty, and the mean of a normal distribution is identifiable.
The GLM conditions are also straightforward to verify. $\Theta$ is open and convex, and $\theta^\T x\in\mathcal{E}$ for all $\theta\in \Theta$, $x\in\X$.
Condition \ref{item:GLMs-linearly} of Theorem \ref{theorem:GLMs} is satisfied by assumption,
and condition~\ref{item:GLMs-triple} of Theorem \ref{theorem:GLMs} is satisfied trivially since $\kappa'''(\eta) = 0$ for all $\eta\in\mathcal{E}$.
Condition~\ref{item:linreg-moments} of Theorem~\ref{theorem:linear-regression} implies that condition~\ref{item:GLMs-moments} of Theorem \ref{theorem:GLMs} holds,
since $\E|X_i s(Y_i)| = \E|X_i Y_i|/\sigma^2 < \infty$ and $\E|\kappa(\theta^\T X_i)| = \theta^\T (\E X_i X_i^\T) \theta / (2\sigma^2) < \infty$.
It is straightforward to verify that condition~\ref{item:GLMs-minimum} of Theorem \ref{theorem:GLMs} holds with $\theta_0 = (\E X_i X_i^\T)^{-1} \E X_i Y_i$.
\end{proof}

\subsubsection{Logistic regression}

The logistic regression model is $p(y_i\mid\theta,x_i) = \Bernoulli(y_i\mid\sigma(\theta^\T x_i))$
for $y_i\in\Y := \{0,1\}$, $x_i\in\X := \R^D$, and $\theta\in \Theta := \R^D$,
where $\sigma(\eta) = 1/(1 + e^{-\eta})$ for $\eta\in\mathcal{E} := \R$.
Thus, $p(y_i\mid \theta,x_i) = q(y_i\mid \theta^\T x_i)$ where
$q(y|\eta) := \exp(\eta y - \kappa(\eta))$ is a density with respect to 
$\lambda = \delta_0 + \delta_1$
for $y\in\Y$ and $\eta\in\mathcal{E}$,
by defining $\kappa(\eta) = \log(1 + e^\eta)$.
Here, $\delta_y$ denotes the unit point mass at $y$.

\begin{theorem}[Logistic regression]
\label{theorem:logreg}
Suppose $(X_1,Y_1),(X_2,Y_2),\ldots\in\X\times\Y$ i.i.d.\ such that:
\begin{enumerate}
\item\label{item:logreg-minimum} $f'(\theta_0)=0$ for some $\theta_0\in\Theta$, where $f(\theta) = -\E \log q(Y_i\mid\theta^\T X_i)$,
\item\label{item:logreg-moments} $\E|X_{i j} X_{i k} X_{i \ell}| < \infty$ for all $j,k,\ell\in\{1,\ldots,D\}$, and
\item\label{item:logreg-linearly} for all $a\in\R^D$, if $a^\T X_i \overset{\mathrm{a.s.}}{=} 0$ then $a = 0$.
\end{enumerate}
Then for any open ball $E\subseteq\Theta$ such that $\theta_0\in E$, $f_n(\theta) := -\frac{1}{n}\sum_{i = 1}^n \log q(Y_i\mid\theta^\T X_i)$
satisfies the conditions of Theorem \ref{theorem:altogether} with probability $1$.
\end{theorem}

Condition~\ref{item:logreg-minimum} is essentially that the MLE exists, asymptotically,
and condition~\ref{item:logreg-linearly} is necessary for identifiability (see the remarks following Theorem~\ref{theorem:GLMs}); these are both very mild.
Condition~\ref{item:logreg-moments} is a third-moment condition on the covariates,
which we use to bound $f_n'''(\theta)$; this is more stringent, but is reasonable in many practical applications.

\vspace{1em}
\begin{proof}
Condition~\ref{condition:expfam} is easily checked: $\mathcal{E} = \{\eta\in\R : |\kappa(\eta)| < \infty\}$, 
$\mathcal{E}$ is open and nonempty, and $\eta$ is identifiable since $\sigma(\eta)$ is one-to-one.
Trivially, $\Theta$ is open and convex, and $\theta^\T x\in\mathcal{E}$ for all $\theta\in \Theta$, $x\in\X$.
Conditions \ref{item:GLMs-minimum} and \ref{item:GLMs-linearly} of Theorem \ref{theorem:GLMs} are satisfied by conditions 
\ref{item:logreg-minimum} and \ref{item:logreg-linearly} of Theorem~\ref{theorem:logreg}, respectively.
Condition \ref{item:GLMs-triple} of Theorem \ref{theorem:GLMs} is satisfied due to condition \ref{item:logreg-moments} and the fact that
$|\kappa'''(\eta)|\leq 3$ for all $\eta\in\mathcal{E}$, because $\kappa''' = \sigma(1-\sigma)(1-2\sigma)^2 - 2\sigma^2(1-\sigma)^2$ and $0<\sigma(\eta)<1$.
Condition~\ref{item:logreg-moments} also implies that $\E|X_i|<\infty$, 
because $|X_i| \leq \sum_j|X_{i j}|$ and $\E|X_{i j}| < \infty$ for all $j$ \citep[6.12]{folland2013real}.
It follows that condition \ref{item:GLMs-moments} of Theorem \ref{theorem:GLMs} holds,
since $\E|X_i Y_i| \leq \E|X_i| < \infty$ and $\E|\kappa(\theta^\T X_i)| \leq \log 2 + \E|\theta^\T X_i| \leq \log 2 + |\theta|\E|X_i| < \infty$,
where we have used the inequality $|\kappa(\eta)| = \log(1 + e^\eta) \leq \log 2 + |\eta|$ for $\eta\in\R$.
\end{proof}

\subsubsection{Poisson regression}

The Poisson regression model is $p(y_i\mid\theta,x_i) = \Poisson(y_i\mid \exp(\theta^\T x_i))$
for $y_i\in\Y:=\{0,1,2,\ldots\}$, $x_i\in\X := \R^D$, and $\theta\in \Theta := \R^D$. 
Thus, $p(y_i\mid \theta,x_i) \propto_\theta q(y_i\mid \theta^\T x_i)$ where
$q(y|\eta) := \exp(\eta y - \kappa(\eta))$ is a density with respect to $\lambda := \sum_{y\in\Y} \delta_y / y!$
for $y\in\Y$ and $\eta\in\mathcal{E}:=\R$,
by defining $\kappa(\eta) = e^\eta$.

\begin{theorem}[Poisson regression]
\label{theorem:poisson}
Suppose $(X_1,Y_1),(X_2,Y_2),\ldots\in\X\times\Y$ i.i.d.\ such that:
\begin{enumerate}
\item\label{item:poisson-minimum} $f'(\theta_0)=0$ for some $\theta_0\in\Theta$, where $f(\theta) = -\E \log q(Y_i\mid\theta^\T X_i)$,
\item\label{item:poisson-moments} $\E|X_i Y_i| < \infty$ and $\E \exp(c|X_i|) < \infty$ for all $c>0$, and
\item\label{item:poisson-linearly} for all $a\in\R^D$, if $a^\T X_i \overset{\mathrm{a.s.}}{=} 0$ then $a = 0$.
\end{enumerate}
Then for any open ball $E\subseteq\Theta$ such that $\theta_0\in E$, $f_n(\theta) := -\frac{1}{n}\sum_{i = 1}^n \log q(Y_i\mid\theta^\T X_i)$
satisfies the conditions of Theorem \ref{theorem:altogether} with probability $1$.
\end{theorem}
\begin{proof}
As before, Condition~\ref{condition:expfam} is easily checked: $\mathcal{E} = \{\eta\in\R : |\kappa(\eta)| < \infty\}$, 
$\mathcal{E}$ is open and nonempty, and $\eta$ is identifiable. Trivially, $\Theta$ is open and convex, and $\theta^\T x\in\mathcal{E}$ for all $\theta\in \Theta$, $x\in\X$.
Conditions \ref{item:GLMs-minimum} and \ref{item:GLMs-linearly} of Theorem \ref{theorem:GLMs} are satisfied by conditions 
\ref{item:poisson-minimum} and \ref{item:poisson-linearly} of Theorem~\ref{theorem:poisson}.
Condition \ref{item:GLMs-moments} of Theorem \ref{theorem:GLMs} is satisfied 
due to condition \ref{item:poisson-moments} of Theorem~\ref{theorem:poisson},
since for all $\theta\in \Theta$, $\E|\kappa(\theta^\T X_i)| = \E \exp(\theta^\T X_i) \leq \E \exp(|\theta||X_i|) < \infty$.
For all $m\in\N$ and $j\in\{1,\ldots,D\}$, 
$\E|X_{i j}|^m \leq \E|X_i|^m = m! \E(|X_i|^m / m!) \leq m! \E \exp(|X_i|) < \infty.$
Further, letting $r>0$, $c = |\theta_0| + r$, and $E = B_r(\theta_0)$, we have that for all $\theta\in \bar E$,
$\kappa'''(\theta^\T X_i) = \exp(\theta^\T X_i) \leq \exp(c|X_i|)$.  Hence,
$$\E\big(\sup_{\theta\in\bar E}|\kappa'''(\theta^\T X_i) X_{i j} X_{i k} X_{i \ell}|\big)
\leq \E \big(e^{c|X_i|}|X_{i j} X_{i k} X_{i \ell}|\big)
\leq \Big(\E e^{4 c|X_i|} \E|X_{i j}|^4 \E|X_{i k}|^4 \E|X_{i \ell}|^4\Big)^{1/4}$$
by H\"older's inequality \citep[6.2]{folland2013real}; thus, condition \ref{item:GLMs-triple} of Theorem \ref{theorem:GLMs} is satisfied.
\end{proof}

\section{Applications to generalized posteriors}
\label{section:generalized-applications}


\subsection{Pseudolikelihood-based posteriors}
\label{section:pseudolik}

Pseudolikelihood \citep{besag1975statistical} is a powerful approach for many models in which the 
likelihood is difficult to compute due to intractability of the normalization constant.
Instead of the standard likelihood $p(y_1,\ldots,y_n \mid \theta)$,
the basic idea is to use a \textit{pseudolikelihood} $\mathcal{L}(\theta) = \prod_{i=1}^n p(y_i \mid y_{-i}, \theta)$
where $y_{-i} = (y_1,\ldots,y_{i-1},y_{i+1},\ldots,y_n)$.
Maximum pseudolikelihood estimates are used in many applications and have been shown to be 
consistent and asymptotically normal in a range of cases
\citep{besag1975statistical,geman1986markov,gidas1988consistency,comets1992consistency,jensen1994asymptotic,mase1995consistency,liang2003maximum,hyvarinen2006consistency}.
Usage of pseudolikelihoods for constructing generalized posteriors is much less common,
perhaps due to concerns about the validity of the resulting posterior
\citep[but see][]{zhou2009bayesian,bouranis2017efficient,pauli2011bayesian,ryden1998computational}.

In this section, we provide sufficient conditions for
concentration, asymptotic normality, and the Laplace approximation
for a large class of pseudolikelihood-based posteriors.
Specifically, we consider pseudolikelihoods in which each factor takes the form of a generalized linear model.
We provide a general result for pseudolikelihoods in this class, and
then consider three cases in particular:
Gaussian Markov random fields (Section~\ref{section:grf}),
fully visible Boltzmann machines (Section~\ref{section:boltzmann}), and
the Ising model on $\Z^m$ (Section~\ref{section:ising}).
Any pseudolikelihood is a composite likelihood, so as discussed in Section~\ref{section:composite-likelihood}, if the model is correct then we can expect consistency but not necessarily correct frequentist coverage.

\begin{condition}
\label{condition:pseudolik}
Suppose the data can be arranged in a sequence $y_1,y_2,\ldots\in\Y\subseteq\R^d$ and consider a pseudolikelihood of the form:
$$ \mathcal{L}_n^\mathrm{pseudo}(\theta) \propto \prod_{i = 1}^n q\big(y_i \mid \theta^\T \varphi_i(\vec{y})\big) $$
for $\theta\in\Theta\subseteq\R^D$, where $\varphi_i(\vec{y})\in\X\subseteq\R^D$ is a function of $\vec{y} = (y_1,y_2,\ldots)$
and $q(y|\eta) = \exp(\eta s(y) - \kappa(\eta))$
is a one-parameter exponential family satisfying Condition~\ref{condition:expfam}
for $y\in\Y$, $\eta\in\mathcal{E}$.
Assume $\Theta$ is open and convex, and $\theta^\T x\in\mathcal{E}$ for all $\theta\in \Theta$, $x\in\X$.
\end{condition}

\begin{theorem}
\label{theorem:pseudolik}
Assume the setup in Condition~\ref{condition:pseudolik}.
Let $\vec{Y} = (Y_1,Y_2,\ldots)$ be a sequence of random vectors in $\Y$ and define $X_i = \varphi_i(\vec{Y})$.
Suppose $(X_1,Y_1),(X_2,Y_2),\ldots$ are identically distributed, but not necessarily independent.
Define $f_n(\theta) = -\frac{1}{n} \sum_{i=1}^n \log q(Y_i \mid \theta^\T X_i)$ and $f(\theta) = -\E \log q(Y_i \mid \theta^\T X_i)$ for $\theta\in\Theta$.
Assume:
\begin{enumerate}
\item\label{item:pseudolik-convergence} for all $\theta\in \Theta$, $f(\theta)$ is finite and $f_n(\theta) \xrightarrow[]{\mathrm{a.s.}} f(\theta)$ as $n\to\infty$,
\item\label{item:pseudolik-minimum-diff} there exists $\theta_0\in\Theta$ such that $f'(\theta_0)=0$ 
    and $f''(\theta_0) = \E\big(\kappa''(\theta_0^\T X_i) X_i X_i^\T\big)$,
\item\label{item:pseudolik-linearly} for all $a\in\R^D$, if $a^\T X_i \overset{\mathrm{a.s.}}{=} 0$ then $a = 0$, and
\item\label{item:pseudolik-triple} with probability $1$, $(f_n''')$ is uniformly bounded on some open ball $E\subseteq \Theta$ containing $\theta_0$.
\end{enumerate}
Then for any $E$ as in condition~\ref{item:pseudolik-triple}, $f_n$ satisfies the conditions of Theorem \ref{theorem:altogether} with probability~$1$.
\end{theorem}

\begin{proof}
As in the proof of Theorem~\ref{theorem:GLMs}, $f_n$ is $C^\infty$, $f_n$ is convex, 
and by convexity, condition~\ref{item:pseudolik-convergence} implies that with probability $1$, for all $\theta\in \Theta$, $f_n(\theta)\to f(\theta)$.
By Theorem~\ref{theorem:regular-convergence}, $f''(\theta_0)$ exists and is finite.  
Thus, $f''(\theta_0)$ is positive definite since for all nonzero $a\in\R^D$,
$a^\T f''(\theta_0) a =\E\big(\kappa''(\theta_0^\T X_i) a^\T X_i X_i^\T a\big) > 0$ by conditions \ref{item:pseudolik-minimum-diff} and \ref{item:pseudolik-linearly}
and the fact that $\kappa''(\eta) > 0$.
\end{proof}

To explain the notation, the observed data consist of the first $n$ elements of single random sequence $\vec{Y} = (Y_1,Y_2,\ldots)$, where each $Y_i$ is a random vector.
In the Gaussian Markov random field and Ising model examples in Sections~\ref{section:grf} and \ref{section:ising} below,
$\vec{Y}$ contains the values at the vertices of a single infinite graph, arranged as a sequence.
Meanwhile, in the fully visible Boltzmann machine (Section~\ref{section:boltzmann}), we have i.i.d.\ samples of graphs.

\subsection{Gaussian Markov random fields}
\label{section:grf}

Gaussian Markov random fields (GMRFs) are widely used in spatial statistics and time-series \citep{banerjee2014hierarchical}.
Let $G$ be an infinite regular graph with vertices $v(1),v(2),\ldots$, and let $y_1,y_2,\ldots\in\R$ be variables associated with the vertices of $G$ such that $y_i$ is the value at $v(i)$.
Consider a model in which the conditional distribution of $y_i$ given $y_{-i}$ is $p_\theta(y_i | y_{-i}) = \Normal(y_i \mid \theta^\T \varphi_i(\vec{y}), \gamma^{-1})$ where $\theta\in\Theta:=\R^D$, $\varphi_i(\vec{y}) = (y_j : j \in N_i) \in \R^D$, and $N_i = \{j \in \N : v(j) \text{ is adjacent to } v(i)\}$.
This leads to the pseudolikelihood \citep{besag1975statistical}
$$ \mathcal{L}_n^\mathrm{GRF}(\theta) = \prod_{i=1}^n p_\theta(y_i | y_{-i}) = \prod_{i = 1}^n \Normal\big(y_i \mid \theta^\T \varphi_i(\vec{y}), \gamma^{-1}\big). $$
By defining $q(y | \eta) = \exp(\eta \gamma y - \kappa(\eta))$ for $y\in\R$ and $\eta\in\R$, where $\kappa(\eta) = \frac{1}{2} \gamma \eta^2$,
this pseudolikelihood can be written as
$\mathcal{L}_n^\mathrm{GRF}(\theta) \propto \prod_{i = 1}^n q\big(y_i \mid \theta^\T \varphi_i(\vec{y})\big)$.

\begin{theorem}
\label{theorem:grf}
Let $\vec{Y} = (Y_1,Y_2,\ldots)$ be a sequence of random variables in $\R$ and define $X_i = (Y_j : j \in N_i) \in \R^D$
where $N_i$ is defined as above.
Suppose $(X_1,Y_1),(X_2,Y_2),\ldots$ are identically distributed, but not necessarily independent. Assume:
\begin{enumerate}
\item\label{item:grf-convergence} $\frac{1}{n}\sum_{i=1}^n X_i Y_i \xrightarrow[]{\mathrm{a.s.}} \E X_i Y_i \in \R^D$
and $\frac{1}{n}\sum_{i=1}^n X_i X_i^\T \xrightarrow[]{\mathrm{a.s.}} \E X_i X_i^\T \in\R^{D\times D}$, and  
\item\label{item:grf-linearly} for all $a\in\R^D$, if $a^\T X_i \overset{\mathrm{a.s.}}{=} 0$ then $a = 0$.
\end{enumerate}
Then $\theta_0 := (\E X_i X_i^\T)^{-1} \E X_i Y_i$ is well-defined and 
for any open ball $E$ such that $\theta_0\in E$, 
$f_n(\theta) := -\frac{1}{n}\sum_{i = 1}^n \log q(Y_i\mid\theta^\T X_i)$
satisfies the conditions of Theorem \ref{theorem:altogether} with probability $1$.
\end{theorem}
\begin{proof}
We apply Theorem~\ref{theorem:pseudolik}.
Let $f(\theta) = -\E \log q(Y_i \mid \theta^\T X_i) = \frac{1}{2}\gamma \theta^\T (\E X_i X_i^\T) \theta - \gamma \theta^\T \E X_i Y_i$ for $\theta\in\R^D$.
Thus, $f''(\theta) = \gamma (\E X_i X_i^\T) = \E\big(\kappa''(\theta^\T X_i) X_i X_i^\T\big)$ since $\kappa''(\eta) = \gamma$.
By condition \ref{item:grf-convergence}, for all $\theta\in\R^D$, $f(\theta)$ is finite and $f_n(\theta)  \xrightarrow[]{\mathrm{a.s.}} f(\theta)$ as $n\to\infty$.
As in the case of linear regression (Theorem~\ref{theorem:linear-regression}), 
$E X_i X_i^\T$ is positive definite by condition~\ref{item:grf-linearly}, $f'(\theta_0) = 0$,
and $(f_n''')$ is a.s.\ uniformly bounded on all of $\R^D$ since $\kappa'''(\eta) = 0$.
\end{proof}

The setup of Theorem~\ref{theorem:grf} is quite general; 
note that the graph $G$ may consist of a single connected component (such as the $m$-dimensional integer lattice $\Z^m$) or it may consist of many disconnected components, each of which could contain finitely many or infinitely many vertices.
Further, the setup is that there is a single graph $G$, and more and more of the graph is observed as $n$ grows;
thus, it is necessary that $G$ be infinite in order to obtain an asymptotic result, as we do in Theorem~\ref{theorem:grf}.
The identically distributed assumption is quite general as well; for instance, it holds whenever the true distribution is stationary with 
respect to a set of transformations that can map $v(i)$ to $v(j)$ for any $i,j$.
Thus, this assumption is reasonable since stationarity is commonly assumed
\citep{banerjee2014hierarchical,lee2002markov,kervrann1995markov}; also see \citet{kunsch1981thermodynamics} for background.
Condition~\ref{item:grf-convergence} of Theorem~\ref{theorem:grf} is that a strong law of large numbers holds for $X_i Y_i$ and $X_i X_i^\T$;
in Theorem~\ref{theorem:grf-ergodic} we show that this holds whenever the true distribution
is a stationary, ergodic process on the integer lattice $\Z^m$, assuming a moment condition.
Condition~\ref{item:grf-linearly} of Theorem~\ref{theorem:grf} is simply that the support of the neighbor vector $X_i$ is not contained in any proper subspace of $\R^D$; see Theorem~\ref{theorem:GLMs} for further discussion of this non-degeneracy assumption.

Many applications in spatial statistics involve more complex models that do not satisfy the 
assumption of a regular graph with identically distributed neighborhoods $(X_i,Y_i)$ \citep{ferreira2007bayesian}.
While Theorem~\ref{theorem:grf} could be extended to handle such generalizations,
we chose to keep it relatively simple in order to capture the essential features of this class of models 
without being overburdened with details.

\begin{theorem}
\label{theorem:grf-ergodic}
Suppose $G$ is the $m$-dimensional lattice on $\Z^m$, and let $v:\N\to\Z^m$ be a bijection from $\N$ to $\Z^m$
such that $R(v(1))\leq R(v(2)) \leq \cdots$ where $R(j) = \max\{|j_1|,\ldots,|j_m|\}$ for $j\in \Z^m$.  
Let $T_1,\ldots,T_m$ denote the shift transformations on $\Z^m$.
Suppose $(Y_1,Y_2,\ldots)$ is a stochastic process such that the random field $(Y_{v^{-1}(j)} : j\in\Z^m)$ is stationary 
with respect to $T_1,\ldots,T_m$ and ergodic with respect to at least one of $T_1,\ldots,T_m$.
If $\E|Y_1|^4 < \infty$,
then condition~\ref{item:grf-convergence} of Theorem~\ref{theorem:grf} holds.
\end{theorem}
See Section \ref{section:generalized-proof} for the proof.

\subsection{Fully visible Boltzmann machines}
\label{section:boltzmann}

The Boltzmann machine is a stochastic recurrent neural network originally developed as a model of neural computation \citep{hinton1983optimal,ackley1985learning}.
Maximum pseudolikelihood estimation has been shown to be consistent for fully visible Boltzmann machines \citep{hyvarinen2006consistency}.
Here, we consider the corresponding pseudolikelihood-based generalized posteriors.
To our knowledge, Theorem~\ref{theorem:boltzmann} is the first result establishing a Bernstein--von Mises theorem for this model.

Define $p_{A,b}(y) \propto \exp(y^\T A y + b^\T y)$ for $y\in\Y := \{-1,1\}^d$, where $A\in\R^{d\times d}$ is a strictly upper triangular matrix and $b\in\R^d$.
Given samples from $p_{A,b}$, inference for $A$ and $b$ is complicated by the intractability of the normalization constant $Z_{A,b} = \sum_{y\in\Y} \exp(y^\T A y + b^\T y)$ since $|\Y| = 2^d$ is very large when $d$ is large.  Observe that we can write
\begin{equation}\label{equation:boltzmann}
p_{A,b}(y_j | y_{-j}) \propto_{y_j} \exp\big(\textstyle\sum_{k=1}^{j-1} A_{k j} y_k y_j + \sum_{k=j+1}^{d} A_{j k} y_j y_k + b_j y_j\big)
= \exp\big(y_j \theta^\T \varphi_j(y)\big)
\end{equation}
where $\theta=\theta(A,b)\in\R^D$ is a $D = d + d(d-1)/2$ dimensional vector concatenating $b$ and the strictly upper triangular entries of $A$,
and $\varphi_j(y)\in\{-1,0,1\}^D$ is a function that does not depend on $y_j$. 
Thus, we have $p_{A,b}(y_j | y_{-j}) = q\big(y_j \mid \theta^\T\varphi_j(y)\big)$
by defining $q(y_j|\eta) = \exp(\eta y_j - \kappa(\eta))$ for $y_j\in\{-1,1\}$ and $\eta\in\R$, where $\kappa(\eta) = \log(e^\eta + e^{-\eta})$.
Now, suppose we have $n$ samples $y_1,\ldots,y_n\in\Y=\{-1,1\}^d$ and 
for $\theta\in\Theta:=\R^D$,
consider the pseudolikelihood
$$ \mathcal{L}_n^\mathrm{Boltz}(\theta) = \prod_{i=1}^n \prod_{j=1}^d p_{A,b}(y_{i j} | y_{i,-j}) 
= \prod_{i = 1}^n \prod_{j = 1}^d q\big(y_{i j} \mid \theta^\T\varphi_j(y_i)\big).$$

\begin{theorem}
\label{theorem:boltzmann}
Let $Y_1,Y_2,\ldots\in\Y$ be i.i.d.\ random vectors and define $X_{i j} = \varphi_j(Y_i)$.
Define $f(\theta) = -\sum_{j=1}^d \E \log q(Y_{i j} \mid \theta^\T X_{i j})$ for $\theta\in\Theta$.
Assume:
\begin{enumerate}
\item\label{item:boltzmann-minimum} $f'(\theta_0)=0$ for some $\theta_0\in\Theta$, and
\item\label{item:boltzmann-linearly} for all nonzero $a\in\R^d$, $\mathrm{Var}(a^\T Y_i) > 0$.
\end{enumerate}
Then for any open ball $E$ such that $\theta_0\in E$, 
$f_n(\theta) := -\frac{1}{n}\sum_{i=1}^n \sum_{j=1}^d \log q\big(Y_{i j} \mid \theta^\T X_{i j}\big)$
satisfies the conditions of Theorem \ref{theorem:altogether} with probability~$1$.
\end{theorem}

The assumptions of Theorem~\ref{theorem:boltzmann} are extremely mild and can be expected to typically hold in practice.  
Condition~\ref{item:boltzmann-minimum} is simply that the maximum pseudolikelihood estimator exists, asymptotically
-- or more precisely, that the asymptotic pseudolikelihood function has a critical point.
Condition~\ref{item:boltzmann-linearly} is that there is no lower-dimensional affine subspace that contains $Y_i$ almost surely;
this is analogous to the non-degeneracy condition in Theorem~\ref{theorem:pseudolik}.

\vspace{1em}
\begin{proof}{\bf of Theorem~\ref{theorem:boltzmann}}~
Observe that 
$$f_n(\theta) = \frac{1}{n}\sum_{i=1}^n \sum_{j=1}^d \kappa(\theta^\T X_{i j}) \,-\, \theta^\T \Big(\frac{1}{n}\sum_{i=1}^n \sum_{j=1}^d X_{i j} Y_{i j}\Big)$$
and $f(\theta) = \sum_{j=1}^d \E\kappa(\theta^\T X_{i j}) \,-\, \theta^\T \big(\sum_{j=1}^d \E X_{i j} Y_{i j}\big)$.
As in the proof of Theorem~\ref{theorem:GLMs}, $f_n$ is $C^\infty$ and convex.
Since $\{-1,0,1\}^D$ is a finite set, $\sup \big\{|\kappa(\theta^\T x)| : x\in\{-1,0,1\}^D\big\} < \infty$ for all $\theta\in\Theta$. 
Also, $|X_{i j k} Y_{i j}| \leq 1$, and thus, 
$f(\theta)$ is finite and 
$f_n(\theta) \xrightarrow[]{\mathrm{a.s.}} f(\theta)$ by the strong law of large numbers.
(Note that the standard strong law of large numbers applies here since the data consist of $n$ i.i.d.\ samples from the Boltzmann machine,
rather than the first $n$ elements of a single sample as in the GMRF example.)
Due to convexity, this implies that with probability $1$, for all $\theta\in \Theta$, $f_n(\theta)\to f(\theta)$ as $n\to\infty$. 

Let $E$ be an open ball containing $\theta_0$.  Then for all $\theta\in E$, 
$|f_n'''(\theta)_{k\ell m}| \leq c d$ where $c = \sup \{ |\kappa'''(\theta^\T x)| : x\in\{-1,0,1\}^D, \theta\in\bar{E}\}$, and $c<\infty$
because $\kappa'''$ is continuous and $\bar{E}$ is compact.
Thus, for all $\theta\in E$, $\|f_n'''(\theta)\|^2 = \sum_{k,\ell,m} |f_n'''(\theta)_{k\ell m}|^2 \leq c^2 d^2 D^3$. 
Hence, $(f_n''')$ is uniformly bounded on $E$.

Now, we show that $f''(\theta_0)$ is positive definite.
First, $f''(\theta_0) = \sum_{j=1}^d \E\big(\kappa''(\theta_0^\T X_{i j}) X_{i j} X_{i j}^\T\big)$
because differentiating under the integral sign is justified by the bounds
$|\kappa(\eta)| \leq |\eta| + \log 2$,
$|\kappa'(\eta)| \leq 1$, 
$|\kappa''(\eta)| \leq 2$, and
$|X_{i j k}| \leq 1$ \citep[2.27]{folland2013real}.
Let $\theta\in\R^D$ be nonzero and let $A,b$ be the corresponding parameters such that $\theta = \theta(A,b)$.
Then by Equation~\ref{equation:boltzmann}, $A^\T Y_i + A Y_i + b = (\theta^\T X_{i 1},\ldots,\theta^\T X_{i d})^\T \in \R^d$.
If $A \neq 0$, then $\mathrm{Var}(\theta^\T X_{i j'}) > 0$ for some $j'$ by condition~\ref{item:boltzmann-linearly},
and hence, $\theta^\T f''(\theta_0)\theta = \sum_{j=1}^d \E\big(\kappa''(\theta_0^\T X_{i j}) |\theta^\T X_{i j}|^2\big) > 0$ because $\kappa''(\eta)>0$ and $\Pr(|\theta^\T X_{i j'}|>0) > 0$.
Meanwhile, if $A = 0$, then $b_{j'}\neq 0$ for some $j'$ (because $\theta\neq 0$), and again $\theta^\T f''(\theta_0) \theta > 0$ because $|\theta^\T X_{i j'}| = |b_{j'}| > 0$.  
Therefore, $f''(\theta_0)$ is positive definite.
\end{proof}

\subsection{Ising model}
\label{section:ising}

The Ising model is a classical model of ferromagnetism in statistical mechanics
and has gained widespread use in many other applications such as
spatial statistics \citep{banerjee2014hierarchical} and image processing \citep{geman1984stochastic}. 
Pseudolikelihood-based posteriors for the Ising model and Potts model, more generally, have been used by \citet{zhou2009bayesian} for protein modeling.

Consider the $m$-dimensional integer lattice $\Z^m$ and let $v:\N\to\Z^m$ be a bijection from $\N$ to $\Z^m$. 
Let $y_1,y_2,\ldots\in\Y:=\{-1,1\}$ be variables associated with the points of $\Z^m$ such that $y_i$ is the value at $v(i)$.
The Ising model is a Markov random field with singleton potentials $\exp(\theta_1 y_i)$ for each $i\in\N$ and pairwise potentials $\exp(\theta_2 y_i y_j)$ for each pair $i,j\in\N$ such that $v(i)$ and $v(j)$ are adjacent in $\Z^m$, that is, such that $|v(i) - v(j)| = 1$.
This motivates the use of the pseudolikelihood \citep{besag1975statistical},
$$ \mathcal{L}_n^\mathrm{Ising}(\theta) = \prod_{i = 1}^n 
\frac{\exp({\textstyle \theta_1 y_i + \theta_2 \sum_{j \in N_i} y_i y_j})}{\sum_{y\in\Y}\exp({\textstyle \theta_1 y + \theta_2 \sum_{j \in N_i} y y_j})} $$
for $\theta \in \Theta := \R^2$,
where $N_i = \{j \in \N : v(j) \text{ is adjacent to } v(i)\}$.
By defining $q(y|\eta) = \exp(\eta y - \kappa(\eta))$ for $y\in\{-1,1\}$ and $\eta\in\R$, where $\kappa(\eta) = \log(e^\eta + e^{-\eta})$,
the Ising model pseudolikelihood can be written as 
$\mathcal{L}_n^\mathrm{Ising}(\theta) = \prod_{i=1}^n q(y_i \mid \theta_1 + \theta_2 \sum_{j \in N_i} y_j)$.


\begin{theorem}
\label{theorem:ising}
Let $\vec{Y} = (Y_1,Y_2,\ldots)$ be a sequence of random variables in $\{-1,1\}$ and define $X_i = \big(1,\,\sum_{j\in N_i} Y_j\big)^\T\in\R^2$.
Suppose $(X_1,Y_1),(X_2,Y_2),\ldots$ are identically distributed, but not necessarily independent.
Define $f_n(\theta) = -\frac{1}{n} \sum_{i=1}^n \log q(Y_i \mid \theta^\T X_i)$ and $f(\theta) = -\E \log q(Y_i \mid \theta^\T X_i)$ for $\theta\in\Theta$.
Assume:
\begin{enumerate}
\item\label{item:ising-convergence} for all $\theta\in \Theta$, $f_n(\theta) \xrightarrow[]{\mathrm{a.s.}} f(\theta)$ as $n\to\infty$,
\item\label{item:ising-minimum} $f'(\theta_0)=0$ for some $\theta_0\in \Theta$, and
\item\label{item:ising-linearly} $\mathrm{Var}\big(\sum_{j \in N_i} Y_j\big) > 0$.
\end{enumerate}
Then for any open ball $E$ such that $\theta_0\in E$, $f_n$ satisfies the conditions of Theorem \ref{theorem:altogether} with probability~$1$.
\end{theorem}

Condition~\ref{item:ising-convergence} is that a strong law of large numbers holds for the log-likelihood terms.
In Theorem~\ref{theorem:ising-ergodic}, we show that condition~\ref{item:ising-convergence} holds
whenever the true distribution is a stationary, ergodic process on $\Z^m$ satisfying a certain moment condition.
Condition~\ref{item:ising-minimum} is that a maximum pseudolikelihood estimate exists, asymptotically.
Condition~\ref{item:ising-linearly} is simply that the distribution of the neighbors is not degenerate, 
in the sense that their support is not restricted to an affine subspace orthogonal to the vector $(1,\ldots,1)^\T$.

\vspace{1em}
\begin{proof}{\bf of Theorem~\ref{theorem:ising}}~
We apply Theorem~\ref{theorem:pseudolik}.
Define $\X = \big\{(1,z)^\T : z \in \{-2 m, \ldots, 2 m\}\big\}$, noting that $X_i \in \X$.
It is easy to check that Condition~\ref{condition:pseudolik} holds.
For all $\theta\in \Theta$, $f(\theta)$ is finite since $|\X\times\Y| < \infty$.
If $a^\T X_i \overset{\mathrm{a.s.}}{=} 0$ then $a = 0$, since
$a^\T X_i = a_1 + a_2 \sum_{j \in N_i} Y_j$ and $\mathrm{Var}\big(\sum_{j \in N_i} Y_j\big) > 0$.
Let $E$ be an open ball containing $\theta_0$, and let $c = \sup \{ |\kappa'''(\theta^\T x)| : x\in\X, \theta\in \bar{E} \}$.
Then $c < \infty$ since $\kappa'''$ is continuous, $|\X|$ is finite, and $\bar{E}$ is compact. 
Therefore, for all $\theta \in E$, 
$|f_n'''(\theta)_{j k \ell}| \leq \frac{1}{n}\sum_{i = 1}^n |\kappa'''(\theta^\T X_i) X_{i j} X_{i k} X_{i \ell}| \leq c (2 m)^3$,
and thus, $(f_n''')$ is a.s.\ uniformly bounded on $E$.
Finally, $f''(\theta_0) = \E\big(\kappa''(\theta_0^\T X_i) X_i X_i^\T\big)$ 
because differentiating under the integral sign is justified by the bounds
$|\kappa(\eta)| \leq |\eta| + \log 2$,
$|\kappa'(\eta)| \leq 1$, 
$|\kappa''(\eta)| \leq 2$, and
$|X_{i j}| \leq 2 m$ \citep[2.27]{folland2013real}.
\end{proof}

\begin{theorem}
\label{theorem:ising-ergodic}
Let $v:\N\to\Z^m$ be a bijection 
such that $R(v(1))\leq R(v(2)) \leq \cdots$ where $R(j) = \max\{|j_1|,\ldots,|j_m|\}$ for $j\in \Z^m$.  
Let $T_1,\ldots,T_m$ denote the shift transformations on $\Z^m$.
Suppose $(Y_1,Y_2,\ldots)$ is a stochastic process such that the random field $(Y_{v^{-1}(j)} : j\in\Z^m)$ is stationary 
with respect to $T_1,\ldots,T_m$ and ergodic with respect to at least one of $T_1,\ldots,T_m$.
Assume that $\mathrm{Var}\big(\log q(Y_1 \mid \theta_1 + \theta_2 \sum_{j \in N_1} Y_j)\big)<\infty$ for all $\theta\in\Theta$.
Then condition~\ref{item:ising-convergence} of Theorem~\ref{theorem:ising} holds.
\end{theorem}
The proof is the same as Theorem~\ref{theorem:grf-ergodic}, except with $Z_i = \log q(Y_i \mid \theta_1 + \theta_2 \sum_{j \in N_i} Y_j) - \E\big(\log q(Y_i \mid \theta_1 + \theta_2 \sum_{j \in N_i} Y_j)\big)$.

\subsection{Cox proportional hazards model}
\label{section:cox}

The Cox proportional hazards model \citep{cox1972regression} is widely used for survival analysis.
The proportional hazards model assumes the hazard function for subject $i$ is $\lambda_0(y) \exp(\theta^\T x_i)$
for $y\geq 0$, where $\lambda_0(y)\geq 0$ is a baseline hazard function shared by all subjects, 
$x_i\in\R^D$ is a vector of covariates for subject $i$, and $\theta\in\R^D$ is a vector of coefficients.
To perform inference for $\theta$ in a way that does not require any modeling of $\lambda_0$ and elegantly handles censoring,
\citet{cox1972regression,cox1975partial} proposed using the \textit{partial likelihood},
$$ \mathcal{L}_n^\mathrm{Cox}(\theta) = \prod_{i=1}^n \bigg( \frac{\exp(\theta^\T x_i)}{\sum_{j=1}^n \exp(\theta^\T x_j)\I(y_j\geq y_i)} \bigg)^{z_i} $$
where $y_i \geq 0$ is the outcome time for subject $i$ and $z_i\in\{0,1\}$ indicates whether $y_i$ is an observed event time ($z_i=1$) or a right-censoring time ($z_i=0$).
When $z_i=1$, the $i$th factor in the partial likelihood can be interpreted as the conditional probability that subject $i$ has an event at time $y_i$,
given the risk set $\{j : y_j \geq y_i\}$ (the set of subjects that have not yet had an event or been censored up until time $y_i$) and
given that some subject has an event at time $y_i$.
See \citet{efron1977efficiency} for an intuitive explanation of the Cox partial likelihood based on a discrete approximation.
Formally, the Cox partial likelihood coincides with the likelihood of a certain generalized linear model with categorical outcomes, 
however, asymptotic analysis is complicated by the dependencies between the factors of the partial likelihood.
A number of authors have studied the asymptotics of the Cox partial likelihood; we mention, in particular,
the result of \citet{lin1989robust} on asymptotic normality of the maximum partial likelihood estimator for the Cox model under misspecification.

The generalized posterior $\pi_n(\theta) \propto \mathcal{L}_n^\mathrm{Cox}(\theta) \pi(\theta)$ based on the Cox partial likelihood has been considered by several authors 
\citep{raftery1996accounting,sinha2003bayesian,kim2009bayesian,ventura2016pseudo}.  \citet{sinha2003bayesian} show that $\pi_n$ approximates the standard posterior under a semiparametric Bayesian model, extending the results of \citet{kalbfleisch1978non}. 
Here, we provide sufficient conditions for $\pi_n$ to exhibit concentration, asymptotic normality, and an asymptotically correct Laplace approximation. 

\begin{theorem}
\label{theorem:cox}
Suppose $(X,Y,Z),(X_1,Y_1,Z_1),(X_2,Y_2,Z_2),\ldots$ are i.i.d., where $X\in\X\subseteq\R^D$, $Y\geq 0$, and $Z\in\{0,1\}$.
Define $f(\theta) = \E\big(h_Y(\theta) Z\big) - \theta^\T \E(X Z)$ for $\theta\in\Theta:=\R^D$
where $h_y(\theta) = \log \E\big(\exp(\theta^\T X) \I(Y \geq y)\big)$.
Assume:
\begin{enumerate}
\item\label{item:cox-bounded} $\X$ is bounded,
\item\label{item:cox-continuous} the c.d.f.\ of $Y$ is continuous on $\R$,
\item\label{item:cox-nonconstant} $\Pr(Z = 1) > 0$ and $\mathrm{Var}(a^\T X) > 0$ for all nonzero $a\in\R^D$,
\item\label{item:cox-positive} $\Pr(Y\geq y \mid X = x) > 0$ for all $x\in\X$, $y \geq 0$, and
\item\label{item:cox-minimum} $f'(\theta_0)=0$ for some $\theta_0\in\R^D$.
\end{enumerate}
Then for any open ball $E$ such that $\theta_0\in E$, 
$f_n(\theta) := -\frac{1}{n}\log \mathcal{L}_n^\mathrm{Cox}(\theta) - \frac{1}{n} \sum_{i=1}^n Z_i \log n$
satisfies the conditions of Theorem \ref{theorem:altogether} with probability~$1$.
\end{theorem}

See Section \ref{section:generalized-proof} for the proof. 
Note that $\exp(-n f_n(\theta)) \propto \mathcal{L}_n^\mathrm{Cox}(\theta)$ since $\frac{1}{n} \sum_{i=1}^n Z_i \log n$
does not depend on $\theta$; the purpose of introducing this term is so that $f_n$ converges.
For interpretation, in words, the assumptions of Theorem~\ref{theorem:cox} are that:
(\ref{item:cox-bounded}) the covariates are bounded,
(\ref{item:cox-continuous}) the time outcome is a continuous random variable,
(\ref{item:cox-nonconstant})  the probability of observing an uncensored outcome is nonzero, and 
there is no lower-dimensional affine subspace that always contains the covariates (which is necessary for identifiability),
(\ref{item:cox-positive}) the survival function is nonzero, and
(\ref{item:cox-minimum}) the maximum partial likelihood estimate exists, asymptotically (or more precisely, the asymptotic partial likelihood function has a critical point).
These conditions are fairly mild and can be expected to hold in many practical applications.

\subsection{Median-based posterior for a location parameter}
\label{section:median}


Suppose we wish to perform robust Bayesian inference for the parameter $\theta$ of a location family model $G_\theta(x) = G(x - \theta)$
where $G$ is a cumulative distribution function (c.d.f.) on $\R$.
If $G$ is misspecified, then the posterior on $\theta$ can be poorly behaved, and may even fail to converge at all.
For instance, if $G_\theta$ is the c.d.f.\ of $\Normal(\theta,\sigma^2)$ 
and the data are $X_1,X_2,\ldots$ i.i.d.\ $\sim \mathrm{Cauchy}(0,1)$, then the posterior on $\theta$
is concentrated near $\frac{1}{n}\sum_{i=1}^n X_i$ when $n$ is large, but $\frac{1}{n}\sum_{i=1}^n X_i \sim \mathrm{Cauchy}(0,1)$;
thus, the posterior does not converge to any fixed value.

\citet{doksum1990consistent} propose to use the conditional distribution of $\theta$ given the sample median (or some other robust estimate of location)
to perform robust Bayesian inference for $\theta$.
More precisely, let $M(x_{1:n})$ be a sample median of $x_{1:n} = (x_1,\ldots,x_n)$ and assume $G_\theta$ has a density $g_\theta$.
Then when $n$ is odd,
\begin{align*}
p\big(\theta \,\big\vert\, M(X_{1:n}) = m\big) 
&\propto \pi(\theta)\, p(M(X_{1:n}) = m \mid \theta) \\
&\propto \pi(\theta) g_\theta(m) G_\theta(m)^{(n-1)/2} \big(1-G_\theta(m)\big)^{(n-1)/2} \\
&= \pi(\theta) \exp\Big(\tfrac{1}{2}(n-1) \log G(m - \theta)(1-G(m - \theta)) + \log g_\theta(m)\Big)
\end{align*}
where $\pi$ is the prior on $\theta$. Here, the conditional densities are under the model in which $\theta\sim\pi$ and $X_1,\ldots,X_n|\theta$ i.i.d.\ $\sim G_\theta$.
\citet{doksum1990consistent} show that $p(\theta \mid M(X_{1:n}) = M(x_{1:n}))$ and generalizations thereof have desirable properties as 
robust posteriors for $\theta$; in particular, they provide consistency and asymptotic normality results.

% ...

With this as motivation, consider the generalized posterior $\pi_n(\theta) \propto \pi(\theta) \exp(-n f_n(\theta))$
where $f_n(\theta) = -\tfrac{1}{2} \log G(m_n-\theta)(1-G(m_n-\theta))$ and $m_n = M(x_{1:n})$;
this approximates $p(\theta \mid M(X_{1:n}) = m_n) $ and is somewhat simpler to analyze.
The following theorem strengthens the \citet{doksum1990consistent} asymptotic normality result by showing convergence in total variation distance, 
rather than convergence in the weak topology. 
Further, our conditions are simpler, but we do assume greater regularity of $G$ and we only consider the median.



\begin{theorem}
Suppose $G:\R\to(0,1)$ is a c.d.f.\ such that $G'''$ exists and is continuous, 
$G(-x) = 1-G(x)$ for all $x\in\R$, $(\log G)''(x) \leq 0$ for all $x\in\R$, and $(\log G)''(0) < 0$.
If $\theta_0\in\R$ and $m_1,m_2,\ldots\in\R$ such that $\theta_0 = \lim_{n\to\infty} m_n$, 
then for any open ball $E$ containing~$\theta_0$,
$f_n(\theta) := -\tfrac{1}{2} \log G(m_n-\theta)(1-G(m_n-\theta))$ satisfies the conditions of Theorem~\ref{theorem:altogether} on $\R$.
\end{theorem}
\begin{proof}
By the chain rule, $f_n(\theta)$ has a continuous third derivative since $\log(x)$ and $G(x)$ have continuous third derivatives and $G(x)\in(0,1)$.
Define $f(\theta) = -\tfrac{1}{2} \log G(\theta_0 - \theta) (1-G(\theta_0 - \theta))$ for $\theta\in\R$.
Then for all $\theta\in\R$, $f_n(\theta) \to f(\theta)$ as $n\to\infty$ since $m_n\to\theta_0$, $\log(x)$ and $G(x)$ are continuous, and $G(x)\in(0,1)$.
Further,
\begin{align*}
f(\theta) &= -\tfrac{1}{2} \log G(\theta_0 - \theta) - \tfrac{1}{2} \log G(\theta - \theta_0), \\
f'(\theta) &= \tfrac{1}{2} (\log G)'(\theta_0 - \theta) - \tfrac{1}{2} (\log G)'(\theta - \theta_0), \\
f''(\theta) &= -\tfrac{1}{2} (\log G)''(\theta_0 - \theta) - \tfrac{1}{2} (\log G)''(\theta - \theta_0).
\end{align*}
Thus, $f'(\theta_0) = 0$ and $f''(\theta_0) = - (\log G)''(0) > 0$.
Similarly, $f_n''(\theta) = -\tfrac{1}{2} (\log G)''(m_n - \theta) - \tfrac{1}{2} (\log G)''(\theta - m_n) \geq 0$ since $(\log G)''(x) \leq 0$.
Thus, $f_n$ is convex.
Finally, for any bounded open interval $E$ containing $\theta_0$,
$(f_n''')$ is uniformly bounded on $E$ by Proposition~\ref{proposition:bounded-reduction} with $h(\theta,s) = -\tfrac{1}{2}\log G(s-\theta)G(\theta-s)$,
$K = \bar E$, and $S = [\inf m_n, \sup m_n] \subseteq \R$.
\end{proof}


In cases where $f_n(\theta) = h(\theta,s_n)$ for some finite-dimensional statistic $s_n$, the following simple proposition can make it easy to 
verify the uniform boundedness condition.

\begin{proposition}
\label{proposition:bounded-reduction}
Let $K\subseteq\R^D$ and $S\subseteq\R^d$ be compact sets. 
Suppose $f_n(\theta) = h(\theta,s_n)$ for $\theta\in K$, $n\in\N$,
where $h:K\times S \to \R$ and $s_1,s_2,\ldots\in S$.
If $(\partial^3 h/\partial\theta_i\partial\theta_j\partial\theta_k)(\theta,s)$ exists and is continuous on $K\times S$ for all $i,j,k\in\{1,\ldots,D\}$,
then $(f_n''')$ is uniformly bounded on $K$.
\end{proposition}
\begin{proof}
Let $h'''(\theta,s)$ denote the tensor of third derivatives with respect to $\theta$, 
and let $c = \sup\{ \|h'''(\theta,s)\| : \theta\in K, s\in S\}$.
For all $\theta\in K$, $n\in\N$, we have $\|f_n'''(\theta)\|=\|h'''(\theta,s_n)\|\leq c$,
and $c<\infty$ since $(\theta,s)\mapsto \|h'''(\theta,s)\|$ is continuous and $K\times S$ is compact.
\end{proof}

\section{Previous work}
\label{section:previous-work}

In this section, we provide a discussion comparing our assumptions, results, and proof techniques with those in previous work.
Our discussion focuses primarily on asymptotic normality (Bernstein--von Mises), 
and we discuss work on posterior consistency in Section~\ref{section:previous-work-consistency}.

\subsection{Overview of previous BvM results}

The origins of the Bernstein--von Mises (BvM) theorem go back to \citet{laplace1810memoire}, \citet{bernstein1917theory}, and \citet{vonmises1931}.
Rigorous formulations of the theorem were developed by \citet{lecam1953some}, \citet{bickel1969some}, \citet{walker1969asymptotic},
and \citet{dawid1970limiting}.
These works employed ``classical conditions'' involving second, third, or even fourth-order derivatives of the log-likelihood; see
the texts by \citet{lehmann2006theory} and \citet{ghosh2003bayesian}.

\citet{lecam1970assumptions} discovered that
the classical differentiability assumptions could be replaced by a 
less stringest condition referred to as differentiability in quadratic mean (DQM),
which yields the benefits of a quadratic expansion while only requiring a certain first-order derivative; 
also see \citet{le1986asymptotic}, \citet{pollard1997another}, \citet{van2000asymptotic}, and \citet{le2000asymptotics} for background.
\citet{lecam1960necessary} and \citet{schwartz1965bayes}
developed the assumption of the existence of uniformly consistent tests (UCTs)
as a way of guaranteeing that $\theta_0$ is distinguishable.
In combination, the DQM and UCT assumptions form the basis for an elegant BvM theorem in the i.i.d.\ setting \citep{van2000asymptotic}.

The works listed above focus on the canonical setting of a correctly specified, i.i.d.\ probabilistic model
in which the dimension of parameter is fixed and finite.
Going beyond this canonical setting, a number of authors have provided extensions of the theory.
For instance, BvM theorems have been established for non-i.i.d.\ models such as 
Markov processes \citep{borwanker1971bernstein} and the Cox proportional hazards model with a prior on the baseline hazard function \citep{kim2006bernstein}.
More recently, \citet{kleijn2012bernstein} provide a BvM for cases in which the assumed model is misspecified, focusing primarily on the i.i.d.\ setting.
\citet{bochkina2014bernstein} provide an interesting BvM result when the true parameter is on the boundary of the parameter space,
and their result is also applicable under misspecification.

For semiparametric and nonparametric models, BvM results have been established by a number of authors
\citep{shen2002asymptotic,kim2004bernstein,leahu2011bernstein,castillo2013nonparametric,bickel2012semiparametric,castillo2013general}.
A very general result is provided by \citet{panov2015finite},
who establish a finite-sample BvM theorem for non-i.i.d.\ semiparametric models 
under misspecification, allowing the dimension of the parameter to grow with the sample size.
While the theorem of \citet{panov2015finite} is very general, their conditions are quite abstract and may be challenging for non-experts to employ. 

For more references of early BvM contributions, see \citet{bernardo2000bayesian}.






\subsection{Categories of BvM conclusions and conditions}

While all BvM results show that ``the posterior converges to a normal distribution'',
each result can be placed along various axes in terms of
the strength of the conclusions obtained and the generality of the conditions assumed.

\textbf{Strength of the conclusions.}
First, the topology with respect to which convergence is shown to occur is usually either 
the weak topology (that is, convergence in distribution) or the ``strong topology'', that is, the metric topology induced by total variation (TV) distance.
Second, in either topology, one can prove convergence in probability or almost sure convergence.
In our results, we prove almost sure convergence in TV distance, which is the stronger conclusion in both respects.


\textbf{Generality of the conditions.}
We group the conditions commonly assumed in BvM theorems as follows.
The key conditions fall into two categories:
\begin{enumerate}[~~(A)]
\setlength\itemsep{0em}
\item\label{item:regularity} regularity of the log-likelihood or generalized log-likelihood, and
\item\label{item:separation} separation conditions enabling $\theta_0$ to be distinguished.
\end{enumerate}
Other conditions often assumed are that:
\begin{enumerate}[~~(A)]\setcounter{enumi}{2}
\setlength\itemsep{0em}
\item\label{item:posdef} the limiting Hessian at $\theta_0$ is positive definite,
\item\label{item:prior} the prior density is continuous and positive at $\theta_0$,
\item\label{item:param} the dimension of parameter is fixed and finite,
\item\label{item:root} a consistent root of the likelihood exists (such as the MLE),
\item\label{item:prob} the posterior arises from a probabilistic model,
\item\label{item:iid} the data are independent and identically distributed, and
\item\label{item:correct} the model is correctly specified.
\end{enumerate}



\subsection{Abstract BvM conditions}
\label{section:previous-abstract-bvm}

To discuss how our Theorem~\ref{theorem:BVM} relates to the existing literature,
we compare with \citet[Theorem 10.1]{van2000asymptotic} (VdV, for short) as an example of a modern BvM theorem.

VdV assumes conditions \ref{item:prob}, \ref{item:iid}, and \ref{item:correct} (that is, the posterior arises from a correctly specified, i.i.d.\ probabilistic model), whereas Theorem~\ref{theorem:BVM} does not require these conditions -- our results hold in misspecified, non-i.i.d.\ settings and do not even require that the posterior arise from a probability model.
Theorem~\ref{theorem:BVM} and VdV both assume 
condition \ref{item:posdef} (the limiting Hessian at $\theta_0$, which coincides with the Fisher information matrix in VdV, is positive definite),
condition \ref{item:prior} (the prior is continous and positive at $\theta_0$), and
condition \ref{item:param} (the parameter dimension is fixed and finite).
Meanwhile, neither theorem explicitly assumes condition \ref{item:root} (consistent root), but both work with a sequence $\theta_n$ that converges to $\theta_0$.

In category \ref{item:regularity} (regularity conditions), VdV assumes differentiability in quadratic mean (DQM)
whereas Theorem~\ref{theorem:BVM} assumes the quadratic approximation in condition \ref{item:BVM-rep}.
DQM implies a particular quadratic expansion referred to as \textit{local asymptotic normality} (LAN) \citep[Theorem 7.2]{van2000asymptotic}.
The LAN property is roughly similar to condition \ref{item:BVM-rep},
except that it is centered at $\theta_0$ and the remainder is formulated probabilistically \citep{pollard1997another}.
In Theorem~\ref{theorem:BVM}, having a deterministic (rather than probabilistic) bound on the remainder 
facilitates applications to generalized posteriors,
since it decouples the deterministic convergence result from the (possibly complex) distribution of the data.

In category \ref{item:separation} (separation conditions), VdV assumes there exists a sequence of uniformly consistent tests (UCTs) for $H_0 : \theta = \theta_0$ versus $H_1 : \|\theta - \theta_0\| \geq \varepsilon$, for every $\varepsilon > 0$.
Correspondingly, Theorem~\ref{theorem:BVM} assumes condition \ref{item:BVM-liminf}.
The UCT assumption is less stringent in the i.i.d.\ setting, but not as broadly applicable in general;
see our discussion of Schwartz's theorem in Section~\ref{section:previous-work-consistency}.

Finally, VdV establishes convergence in probability with respect to TV distance, 
whereas Theorem~\ref{theorem:BVM} enables us to obtain almost sure convergence in TV. 
Thus, while our assumptions in categories \ref{item:regularity} and \ref{item:separation} may be stronger than VdV's in the correctly specified i.i.d.\ model setting, 
our conditions apply to much more general settings and we also obtain the stronger conclusion of almost sure convergence.
Exploring whether DQM and UCT-like conditions can be extended to generalized posteriors is a potential area for future work.

\subsection{Concrete BvM conditions}

To relate our Theorem~\ref{theorem:altogether} to the existing literature,
we compare with \citet[Theorem 1.4.2]{ghosh2003bayesian} (G\&R, for short) 
as an example of a BvM theorem employing more concrete, classical conditions.

G\&R assume conditions \ref{item:root}, \ref{item:prob}, \ref{item:iid}, and \ref{item:correct} (the posterior arises from a correctly specified, i.i.d.\ probabilistic model with a consistent root), whereas Theorem~\ref{theorem:altogether} does not assume any of these conditions.
Theorem~\ref{theorem:altogether} and G\&R both assume conditions \ref{item:posdef}, \ref{item:prior}, and \ref{item:param}.

In category \ref{item:regularity} (regularity), 
G\&R assume the third derivatives of the log-likelihood terms are dominated, uniformly over a neighborhood of $\theta_0$,
by an integrable function of the data.
Similarly, Theorem~\ref{theorem:altogether} assumes the third derivatives of $f_n$ are uniformly bounded in a neighborhood of $\theta_0$.
In the i.i.d.\ setting, the G\&R domination condition is slightly weaker, but it is not clear how to extend it to arbitrary generalized posteriors. 
In category \ref{item:regularity}, G\&R also assume that the densities have common support,
that the first and second derivatives of the log-likelihood at $\theta_0$ are integrable, and
that differentiation under the integral sign is justified;
meanwhile, we only additionally require that $f_n$ converges pointwise and has continuous third derivatives.
In both G\&R and Theorem~\ref{theorem:altogether}, 
the role of these category \ref{item:regularity} conditions is to bound the error term in a second-order Taylor expansion 
as in condition \ref{item:BVM-rep} of Theorem~\ref{theorem:BVM};
this is formalized on page 37 of G\&R and in our Theorem~\ref{theorem:natural-sufficient}.
This establishes that near $\theta_0$, the log posterior density approaches a quadratic form.
Note that condition \ref{item:posdef} is necessary to ensure that, when exponentiated, the limiting quadratic form can be normalized to a probability density.

In category \ref{item:separation} (separation), G\&R assume that for any $\delta > 0$, there exists $\varepsilon > 0$ such that
$\Pr\big(\inf_{|\theta - \theta_0| > \delta} (f_n(\theta) - f_n(\theta_0)) \geq \varepsilon \big) \longrightarrow 1$ as $n\to\infty$,
where the probability $\Pr(\cdot)$ is with respect to the randomness in $f_n$ due to the data.
Meanwhile, in Theorem~\ref{theorem:altogether}, we assume condition~\ref{item:altogether-liminf},
which is more stringent but is helpful in obtaining almost sure convergence rather than just convergence in probability.
The role of the category \ref{item:separation} conditions is to ensure that negligible mass is placed outside a neighborhood of $\theta_0$, asymptotically.

Like VdV, G\&R prove convergence in probability with respect to TV distance, whereas 
Theorem~\ref{theorem:altogether} enables us to obtain a.s.\ convergence in TV.
Thus, overall, we obtain a stronger and more general conclusion while assuming fewer conditions.
In the special case of correctly specified i.i.d.\ models, our third derivative bounds and our separation condition are more stringent than 
the corresponding G\&R conditions, however, our conditions extend readily to generalized posteriors.

\subsection{BvM under misspecification}
\label{section:previous-work-misspecification}

Relatively recently, \citet{kleijn2012bernstein} extended the BvM theory to handle misspecification,
that is, to apply to cases in which the assumed model is incorrect; also see \citet{bochkina2014bernstein}.
\citet[Theorem 2.1]{kleijn2012bernstein} (K\&V, for short) establish a BvM result assuming
(i) the LAN property holds at rate $\delta_n$ and
(ii) the posterior concentrates (in probability) at $\theta_0$ at the same rate, $\delta_n$.
Additionally, like our Theorem~\ref{theorem:BVM}, K\&V assume conditions \ref{item:posdef}, \ref{item:prior}, and \ref{item:param} above.
(Note that our $H_0$ corresponds to their $V_{\theta^*}$.)
Further, K\&V assume condition~\ref{item:prob} (whereas Theorem~\ref{theorem:BVM} does not), 
however, this assumption might not be essential for their proof.

In category \ref{item:regularity}, condition (i) of K\&V is roughly similar to condition \ref{item:BVM-rep} of Theorem~\ref{theorem:BVM}, 
as discussed in Section~\ref{section:previous-abstract-bvm}.
In category \ref{item:separation}, condition (ii) of K\&V roughly corresponds to condition~\ref{item:BVM-liminf} of Theorem~\ref{theorem:BVM}, except that K\&V assume concentration at a particular rate (to match the LAN rate),
whereas our condition does not require a rate.
On the other hand, (ii) is less stringent in the sense that it only involves posterior probabilities and convergence in probability,
whereas we require strict separation of $\theta_0$ in terms of the values of $f_n(\theta)$.

To facilitate application of their theorem,
\citet{kleijn2012bernstein} focus in particular on the i.i.d.\ setting,
providing concrete conditions under which (i) and (ii) hold,
assuming differentiability, a Lipschitz condition, a second-order Taylor expansion,
non-singular Fisher information, and existence of UCTs.
These conditions generalize VdV to the misspecified setting.

Regarding the strength of the conclusions,
like VdV, K\&V show convergence in probability (in TV),
whereas our results show a.s.\ convergence (in TV).
Thus, while our conditions are stronger in some respects, they are weaker in other respects,
and we obtain a stronger conclusion in a more general setting.



\subsection{BvM for generalized posteriors}

Much previous work on generalized posteriors relies on \citet[Proposition 5.14]{bernardo2000bayesian} to establish asymptotic normality;
for instance, see \citet{lazar2003bayesian}, \citet{greco2008robust}, \citet{pauli2011bayesian}, and \citet{ventura2016pseudo}.
Thus, to relate our results to this literature, 
we compare our Theorem~\ref{theorem:BVM} to \citet[Proposition 5.14]{bernardo2000bayesian} (B\&S, for short),
which is originally due to \citet{chen1985asymptotic}.

First, in terms of the strength of the conclusions, B\&S only show convergence in distribution (that is, convergence in the weak topology), 
whereas Theorem~\ref{theorem:BVM} shows convergence in TV distance, which is much stronger.
On the other hand, both theorems consider an arbitrary deterministic sequence of distributions
(playing the role of generalized posteriors indexed by $n$), and thus, both are conducive for establishing almost sure BvM results.

In category \ref{item:regularity} (regularity), B\&S assume that 
(i) $f_n$ is twice differentiable,
(ii) the smallest eigenvalue of $n f_n''(\theta_n)$ tends to $\infty$ as $n\to\infty$, and
(iii) for all $\varepsilon > 0$, there exists $\delta > 0$ such that for all $n$ sufficiently large,
for all $\theta \in B_\delta(\theta_n)$, the Hessian $f_n''(\theta)$ satisfies
$I - A(\varepsilon) \preceq f_n''(\theta) f_n''(\theta_n)^{-1} \preceq I + A(\varepsilon)$
where $I$ is the identity matrix and $A(\varepsilon)$ is a symmetric positive semidefinite matrix 
whose largest eigenvalue tends to $0$ as $\varepsilon\to 0$.
Here, $A \preceq B$ denotes that $B - A$ is positive definite.
Meanwhile, Theorem~\ref{theorem:BVM} assumes condition \ref{item:BVM-rep}, which does not require any differentiability,
but may require stronger control on the remainder term in Equation~\ref{equation:f_n} 
compared to (iii) above (which B\&S use to derive a quadratic approximation).
B\&S's eigenvalue condition in (ii) above is related to our assumption that $H_n \to H_0$ positive definite;
indeed, the latter implies the former if $r_n''(0) \to 0$ as $n\to\infty$, where $r_n$ is the remainder term in Equation~\ref{equation:f_n}.

In category \ref{item:separation} (separation), B\&S assume that 
(iv) $f_n$ has a strict local minimum $\theta_n$, and 
(v) for any $\delta > 0$, the posterior probability of $B_\delta(\theta_n)$ converges to $1$ as $n\to\infty$.
Meanwhile, Theorem~\ref{theorem:BVM} obtains posterior concentration as a conclusion rather than assuming it,
but Theorem~\ref{theorem:BVM} does assume condition~\ref{item:BVM-liminf}.
Both theorems assume condition \ref{item:param} (the dimension of parameter is fixed and finite),
however, neither theorem requires any of the additional conditions~\ref{item:root}, \ref{item:prob}, \ref{item:iid}, or \ref{item:correct}.
Theorem~\ref{theorem:BVM} assumes condition \ref{item:prior},
while B\&S does not since their assumptions are placed directly on the posterior density.

Overall, while Theorem~\ref{theorem:BVM} assumes more stringent conditions than B\&S in terms of remainder control and separation,
Theorem~\ref{theorem:BVM} does not require differentiability and 
yields a considerably stronger result in terms of TV distance rather than convergence in distribution.

\subsection{Posterior consistency}
\label{section:previous-work-consistency}

\citet{Doob_1949} used martingales to prove a very general result on posterior consistency for correctly specified i.i.d.\ models (also see \citealp{miller2018detailed}),
however, it seems difficult to extend this proof technique, especially to generalized posteriors.
Further, Doob's theorem is only guaranteed to hold on a set of prior probability $1$.
\citet{schwartz1965bayes} established a powerful theorem on posterior consistency based on the UCT assumption
along with an assumption that the prior puts positive mass in Kullback--Leibler neighborhoods of $\theta_0$;
see \citet[Theorem 4.4.1]{ghosh2003bayesian} for a clear exposition.
Schwartz's theorem forms the basis for many modern results on posterior consistency in nonparametric Bayesian models \citep{ghosal2010dirichlet}.
Schwartz's approach improves upon Doob's theorem by guaranteeing consistency at all points, and it is also conducive to generalization.
Other notable early results on posterior consistency are due to \citet{lecam1953some} and \citet{freedman1963asymptotic}.


To interpret our posterior consistency theorems in the context of these well-known results,
we first relate our Theorem~\ref{theorem:f-balls} to Schwartz's theorem as presented in \citet[Theorem 4.4.1]{ghosh2003bayesian}.
Schwartz considers the setting of an i.i.d.\ probability model with densities $p(x|\theta)$, 
and assumes the data $X_i$ are i.i.d.\ from some $P_0$.
In this setting, $f_n(\theta) = -\frac{1}{n}\sum_{i=1}^n \log p(X_i | \theta)$ and
$f_n(\theta) \to f(\theta) = \E_{P_0}(-\log p(X_i | \theta))$ almost surely
by the strong law of large numbers (SLLN), assuming $\E_{P_0}|\log p(X_i | \theta)|<\infty$.
Further, if the model is correctly specified, so that $P_0 = P_{\theta_0}$ for some $\theta_0$,
then $f(\theta) - f(\theta_0)$ equals the Kullback--Leibler divergence, $D(P_{\theta_0} \| P_\theta)$.

Thus, in the correctly specified i.i.d.\ setting,
the interpretation of Theorem~\ref{theorem:f-balls} is as follows:
condition~\ref{item:f-balls-cvg} essentially amounts to assuming a SLLN holds,
condition~\ref{item:f-balls-pi} is that the prior puts positive mass on Kullback--Leibler neighborhoods of $\theta_0$ (just like Schwartz assumes), and roughly speaking, condition~\ref{item:f-balls-min} is that outside neighborhoods of $\theta_0$, 
the log-likelihood does not get too close to the log-likelihood at $\theta_0$ when $n$ is sufficiently large.
The main difference between Theorem~\ref{theorem:f-balls} and Schwartz's theorem is that Theorem~\ref{theorem:f-balls} assumes condition~\ref{item:f-balls-min} 
instead of the UCT condition.
In the i.i.d.\ setting, it seems inevitable that the UCT assumption is less stringent than condition \ref{item:f-balls-min},
however, it is not obvious how to extend the UCT approach to our setting of arbitrary generalized posteriors 
where we do not even assume there exists a data distribution. 
Thus, although in the i.i.d.\ setting, our condition is more stringent, we obtain the benefit of much broader applicability in general.









\subsection{Proof techniques}

The proofs in this paper involve several new or non-standard techniques.
First, in our main results in Sections~\ref{section:concentration} and \ref{section:BVM}, a major shift in technique is to study deterministic sequences of posterior distributions, rather than the usual approach of studying sequences of random posterior distributions obtained from random data. 
By showing that the conditions of the theorems hold with probability $1$, we obtain almost sure convergence.
This device enables one to separate the problem into a real analysis part (involving asymptotics of functions) and a probability part (involving the randomness in the data),
which is particularly useful when considering generalized posteriors and misspecified models.
This decoupling technique may facilitate future work on each part, separately.
\citet[Proposition 5.14]{bernardo2000bayesian} and \citet{chen1985asymptotic} also considered deterministic sequences of distributions,
but only for showing weak convergence rather than in total variation.

The proof of Theorem~\ref{theorem:f-balls} has the same core structure as the proof of Schwartz's theorem \citep[Theorem 4.4.1]{ghosh2003bayesian},
however, to handle generalized posteriors, we use condition \ref{item:f-balls-min} rather than the UCT assumption
in order to employ the deterministic sequence technique and enable application to generalized posteriors.

The proof technique for Theorem~\ref{theorem:BVM} differs from previous BvM proofs in some key respects --- specifically,
this formulation of the conditions facilitates a succinct proof using the generalized dominated convergence theorem.
Further, we use the deterministic sequence technique described above.
On the other hand, certain aspects of the proof are adapted from \citet[Theorem 1.4.2]{ghosh2003bayesian},
such as how we break up the integral into regions.

The proof technique for Theorem~\ref{theorem:altogether} involves innovations as well, encapsulated primarily in Theorem~\ref{theorem:regular-convergence}.
Specifically, regularity properties of $f$, $f'$, $f''$, $f_n$, $f_n'$, and $f_n''$ are obtained via Theorem~\ref{theorem:regular-convergence}, rather than having to be assumed. This simplifies several aspects of the proof, such as interchanging the order of derivatives and limits or expectations.
Another advantage of our proof technique is that there is no need for a common support condition, which is sometimes assumed \citep[Theorem 1.4.2]{ghosh2003bayesian}, because we deal with $f_n$ directly, rather than with a probability model.

\acks{I would like to thank David Dunson, Matthew Harrison, Natalia Bochkina, and Basilis Gidas
for helpful conversations.}

\vskip 0.2in
\begin{spacing}{0.9}
\bibliography{refs}
\end{spacing}


\newpage
\setcounter{page}{1}
\setcounter{section}{0}
\setcounter{table}{0}
\setcounter{figure}{0}
\setcounter{equation}{0}
\renewcommand{\theHsection}{SIsection.\arabic{section}}
\renewcommand{\theHtable}{SItable.\arabic{table}}
\renewcommand{\theHfigure}{SIfigure.\arabic{figure}}
\renewcommand{\theHequation}{SIequation.\arabic{section}.\arabic{equation}}
\renewcommand{\thepage}{S\arabic{page}}  
\renewcommand{\thesection}{S\arabic{section}}   
\renewcommand{\thetable}{S\arabic{table}}   
\renewcommand{\thefigure}{S\arabic{figure}}
\renewcommand{\theequation}{S\arabic{section}.\arabic{equation}}

\bigskip
\bigskip
\bigskip
\begin{center}
{\Large\bf Supplementary material for ``Asymptotic normality, concentration, and coverage of generalized posteriors''}
\end{center}
\medskip

\section{Proofs of concentration results}
\label{section:concentration-proof}

\begin{proof}{\bf of Theorem \ref{theorem:f-balls}}~
Let $\epsilon>0$.  Define $\mu_n(E) = \int_E e^{- n f_n(\theta)}\Pi(d\theta)$ for $E\subseteq\Theta$.  
Recall that $\mu_n(\Theta) = z_n < \infty$ by assumption.
For any $\beta\in\R$,
\begin{align*}
1 - \Pi_n(A_\epsilon) = \Pi_n(A_\epsilon^c) = \frac{\mu_n(A_\epsilon^c)}{\mu_n(\Theta)}
= \frac{e^{n(f(\theta_0)+\beta)}\mu_n(A_\epsilon^c)}{e^{n(f(\theta_0)+\beta)}\mu_n(\Theta)},
\end{align*}
so prove the result, it suffices to show that for some $\beta$, the numerator is bounded and the denominator goes to $\infty$.

First, consider the numerator. Condition \ref{item:f-balls-min} implies that there exists $\beta > 0$ such that for all $n$ sufficiently large, $\inf_{\theta\in A_\epsilon^c} f_n(\theta)\geq f(\theta_0) + \beta$. Then for all $n$ sufficiently large, for all $\theta\in A_\epsilon^c$, we have $\exp\big(- n(f_n(\theta) - f(\theta_0) - \beta)\big) \leq 1$. Hence, for all $n$ sufficiently large, 
\begin{align*}
    e^{n(f(\theta_0)+\beta)} \mu_n(A_\epsilon^c)
    =  \int_{A_\epsilon^c} \exp\big(- n(f_n(\theta) - f(\theta_0) - \beta)\big)\Pi(d\theta) 
    \leq  \int_{A_\epsilon^c} \Pi(d\theta) \leq 1.
\end{align*}

Now, consider the denominator.
For any $\theta\in A_{\beta/2}$, $f_n(\theta) - f(\theta_0) - \beta\longrightarrow f(\theta) - f(\theta_0) - \beta < -\beta/2 < 0$, and thus,
$\exp\big(- n(f_n(\theta) - f(\theta_0) - \beta)\big) \longrightarrow \infty$
as $n\to\infty$. Therefore, by Fatou's lemma,
\begin{align*}
    \liminf_{n\to\infty} e^{n(f(\theta_0)+\beta)} \mu_n(A_{\beta/2})
    =  \liminf_{n\to\infty} \int_{A_{\beta/2}} \exp\big(- n(f_n(\theta) - f(\theta_0) - \beta)\big)\Pi(d\theta) = \infty
\end{align*}
since $\Pi(A_{\beta/2})>0$. Hence, $e^{n(f(\theta_0)+\beta)} \mu_n(\Theta) \to \infty$ since $\mu_n(\Theta) \geq \mu_n(A_{\beta/2})$.
\end{proof}

\begin{lemma}\label{lemma:convex-minimum}
Suppose $\Theta\subseteq\R^D$, $E\subseteq\Theta$ is convex and open in $\R^D$, and $\theta_0\in E$.
Let $f_n:\Theta\to\R$ be convex, and assume $f_n\to f$ pointwise on $E$ for some $f:E\to\R$.
\begin{enumerate}
    \item\label{item:convex-minimum-1} If $f'$ exists in a neighborhood of $\theta_0$, $f'(\theta_0) = 0$, and $f''(\theta_0)$ exists and is positive definite,
        then $f(\theta)>f(\theta_0)$ for all $\theta\in E\setminus\{\theta_0\}$.
    \item\label{item:convex-minimum-2} If $f(\theta)>f(\theta_0)$ for all $\theta\in E\setminus\{\theta_0\}$,
        then $\liminf_n\inf_{\theta\in\Theta\setminus B_\epsilon(\theta_0)} f_n(\theta) > f(\theta_0)$ for any $\epsilon>0$.
\end{enumerate}
\end{lemma}
\begin{proof}
(1) As the pointwise limit of convex functions on a convex open set, $f$ is convex on $E$ \citep[][10.8]{rockafellar1970convex}.
Let $R>0$ such that $f'(\theta)$ exists for all $\theta\in B_R(\theta_0)$.
Let $u\in\R^D$ with $|u|= 1$, and define $g(r) = f(\theta_0+ r u)$ for $r\in[0,R)$.
Then $g'(r) = f'(\theta_0+ r u)^\T u$ and $g''(0) = u^\T f''(\theta_0) u$. Since
$$ \frac{g'(r)}{r} =\frac{g'(r) - g'(0)}{r}\xrightarrow[r\to 0]{} g''(0) = u^\T f''(\theta_0) u > 0, $$
then $g'(r)>0$ for all $r>0$ sufficiently small, say, all $r\in(0,\epsilon]$. Then for any $s\in(0,\epsilon]$, we have
\begin{align}\label{equation:convex-minimum}
f(\theta_0+s u) - f(\theta_0) = g(s) - g(0) =\int_0^s g'(r) d r>0.
\end{align}
Meanwhile, for any $s>\epsilon$ such that $\theta_0+ s u\in E$, we have
$$\frac{1}{s}(f(\theta_0+s u) - f(\theta_0))\geq\frac{1}{\epsilon}(f(\theta_0+\epsilon u) - f(\theta_0))>0 $$
by the convexity of $f$ and by Equation \ref{equation:convex-minimum} with $s=\epsilon$.
Hence, for any $s>0$ such that $\theta_0+ s u\in E$, $f(\theta_0+s u)>f(\theta_0)$.  Since $u$ is arbitrary, the result follows.

(2) By \citet[10.8]{rockafellar1970convex}, $f_n\to f$ uniformly on any compact subset of $E$, and $f$ is convex on $E$.
Further, $f$ is continuous on $E$, as a convex function on a convex open set \citep[][Theorem 10.1]{rockafellar1970convex}.
Let $\epsilon>0$ small enough that the $\epsilon$-sphere $S_\epsilon = \{\theta\in\R^D : |\theta - \theta_0|=\epsilon\}$ is contained in $E$.
Let $\alpha_n =\inf_{\theta\in S_\epsilon} f_n(\theta) - f_n(\theta_0)$ and $\alpha = \inf_{\theta\in S_\epsilon} f(\theta) - f(\theta_0)$.
By uniform convergence, $\alpha_n\to\alpha$. Note that $\alpha>0$, as the minimum of the continuous positive function $f(\theta) - f(\theta_0)$
on the compact set $S_\epsilon$. For any $\theta\in\Theta\setminus B_\epsilon(\theta_0)$,
letting $\xi_\theta$ be the point of $S_\epsilon$ on the line from $\theta$ to $\theta_0$, we have, by the convexity of $f_n$,
$$ f_n(\theta) - f_n(\theta_0)\geq |\theta - \theta_0|\frac{f_n(\xi_\theta) - f_n(\theta_0)}{|\xi_\theta - \theta_0|} \geq \alpha_n $$
whenever $\alpha_n\geq 0$, since $|\theta - \theta_0|\geq|\xi_\theta - \theta_0|$.
Since $\alpha_n\to\alpha>0$, then for all $n$ sufficiently large, for all $\theta\in\Theta\setminus B_\epsilon(\theta_0)$,
$f_n(\theta)\geq f_n(\theta_0) +\alpha_n\longrightarrow f(\theta_0) +\alpha.$
Therefore, $\liminf_n \inf_{\theta\in\Theta\setminus B_\epsilon(\theta_0)} f_n(\theta)\geq f(\theta_0) +\alpha>f(\theta_0)$. 
Note that this also implies the same inequality for any $\epsilon'>\epsilon$.
\end{proof}

\begin{proof}{\bf of Theorem \ref{theorem:concentration}}~\\
\indent \textbf{(Part \ref{item:concentration-metric})}
Defining $A_\epsilon$ as in Theorem \ref{theorem:f-balls}, it suffices to show that
\begin{enumerate}
    \item[(a)] for any $\epsilon>0$ there exists $\delta>0$ such that $A_\delta \subseteq N_\epsilon$, and 
    \item[(b)] for any $\delta>0$ there exists $\epsilon'>0$ such that $N_{\epsilon'} \subseteq A_\delta$,
\end{enumerate}
since for any $\epsilon>0$, choosing $\delta$ by (a), we have $\Pi_n(N_\epsilon) \geq \Pi_n(A_\delta)$;
meanwhile, for any $\delta>0$, choosing $\epsilon'$ by (b), we have $\Pi(A_\delta)\geq \Pi(N_{\epsilon'}) > 0$ and
$\liminf_n \inf_{\theta\in A_\delta^c} f_n(\theta) \geq \liminf_n \inf_{\theta\in N_{\epsilon'}^c} f_n(\theta) > f(\theta_0)$, 
and hence, by Theorem \ref{theorem:f-balls}, $\Pi_n(A_\delta)\to 1$.

(a) Let $\epsilon>0$. Pointwise convergence and the liminf condition imply $\inf_{\theta\in N_\epsilon^c} f(\theta)>f(\theta_0)$, hence, letting $\delta =\inf_{\theta\in N_\epsilon^c} f(\theta) - f(\theta_0)$, we have $\delta>0$ and $A_\delta\subseteq N_\epsilon$.

(b) Let $\delta>0$. By the continuity of $f$ at $\theta_0$, choose $\epsilon'>0$ such that $|f(\theta) - f(\theta_0)| < \delta$ for all $\theta\in N_{\epsilon'}$. Then for any $\theta\in N_{\epsilon'}$, $f(\theta) < f(\theta_0) +\delta$, hence, $\theta\in A_\delta$.

\textbf{(Part \ref{item:concentration-equicontinuous})}
We show that 2 implies 1.
By Lemma \ref{lemma:equicontinuous-convergence}, $f_n\to f$ uniformly on $K$.
Consequently, $f|_K$ is continuous, as the uniform limit of continuous functions \citep[][7.12]{rudin1976principles}.
In particular, $f$ is continuous at $\theta_0$, since $\theta_0$ is an interior point of $K$.
For any $\epsilon>0$,
$$\liminf_n \inf_{\theta\in K\setminus N_\epsilon} f_n(\theta) =\inf_{\theta\in K\setminus N_\epsilon} f(\theta)>f(\theta_0), $$
the first step holding since $f_n\to f$ uniformly on $K$, and the second step since $f|_K$ is continuous, $K\setminus N_\epsilon$ is compact, and $f(\theta)>f(\theta_0)$ for all $\theta\in K\setminus\{\theta_0\}$. Therefore, since $N_\epsilon^c\subseteq (K\setminus N_\epsilon)\cup K^c$,
$$\liminf_n \inf_{\theta\in N_\epsilon^c} f_n(\theta) \geq 
\liminf_n \min\Big\{\inf_{\theta\in K\setminus N_\epsilon} f_n(\theta),\,\inf_{\theta\in K^c} f_n(\theta)\Big\} > f(\theta_0). $$

\textbf{(Part \ref{item:concentration-convex})}
We show that 3 implies 1.
Denote $B_\epsilon = \{\theta\in\R^D: |\theta - \theta_0| < \epsilon\}$.
Let $r>0$ small enough that $B_r\subseteq\Theta$. 
As the pointwise limit of convex functions, $f$ is convex,
and thus, it is continuous on $B_r$ \citep[][10.1]{rockafellar1970convex}.
By Lemma \ref{lemma:convex-minimum} with $E = B_r$, in either case (a) or (b), we have
$$ \liminf_n \inf_{\theta\in\Theta\setminus B_\epsilon} f_n(\theta) > f(\theta_0) $$
for any $\epsilon>0$. Since $\Theta\setminus B_\epsilon = \Theta\setminus N_\epsilon = N_\epsilon^c$, this proves the result.
\end{proof}

\section{Proofs of asymptotic normality results}
\label{section:BVM-proof}

\begin{lemma}\label{lemma:BVM-implies-concentration}
    Let $\theta_n\in\R^D$ such that $\theta_n\to\theta_0$ for some $\theta_0\in\R^D$, let $\pi_n$ be a density with respect to Lebesgue 
    measure on $\R^D$, and let $q_n$ be the density of $\sqrt n (\theta - \theta_n)$ when $\theta\sim \pi_n$.
    If $\int |q_n(x) - q(x)|d x \longrightarrow 0$ for some probability density $q$, then $\pi_n$ concentrates at $\theta_0$.
\end{lemma}
\begin{proof}
    Let $\Pi_n$, $Q_n$, and $Q$ denote the probability measures corresponding to $\pi_n$, $q_n$, and $q$, respectively.
    For any $\epsilon>0$ and $\delta>0$,
    $$ Q_n(B_\delta(0)) = \Pi_n(B_{\delta/\sqrt n}(\theta_n)) \leq \Pi_n(B_\epsilon(\theta_0)) $$
    for all $n$ sufficiently large. Hence, since $Q_n\to Q$ in total variation,
    $$ Q(B_\delta(0)) = \lim_n Q_n(B_\delta(0)) \leq \liminf_n \Pi_n(B_\epsilon(\theta_0)). $$
    Taking the limit as $\delta\to\infty$ shows that $\lim_n \Pi_n(B_\epsilon(\theta_0)) = 1$.
\end{proof}

\begin{proof}{\bf of Theorem \ref{theorem:BVM}}~
Note that $q_n(x) = \pi_n(\theta_n + x/\sqrt n) n^{- D/2}$. Define
\begin{align}
g_n(x) &= \exp\big(- n[f_n(\theta_n + x/\sqrt n) - f_n(\theta_n)]\big) \pi(\theta_n + x/\sqrt n) \notag\\
&= q_n(x) e^{n f_n(\theta_n)} n^{D/2} z_n, \label{equation:g_n}
\end{align}
recalling that $z_n<\infty$ by assumption, and define
$$ g_0(x) = \exp(-\tfrac{1}{2} x^\T H_0 x)\pi(\theta_0). $$
Let $\alpha\in(0,\lambda)$, where $\lambda$ is the smallest eigenvalue of $H_0$. Let $\epsilon>0$ small enough that $\epsilon<\alpha/(2 c_0)$,
$\epsilon<\epsilon_0$, and $\pi(\theta)\leq 2\pi(\theta_0)$ for all $\theta\in B_{2\epsilon}(\theta_0)$
(which we can do since $\pi$ is continuous at $\theta_0$).
Let $\delta =\liminf_n \inf_{\theta\in B_\epsilon(\theta_n)^c}\big(f_n(\theta) - f_n(\theta_n)\big)$,
noting that $\delta>0$ by assumption. Letting $A_n = H_n -\alpha I$ and $A_0 = H_0 -\alpha I$, define
\begin{align*}
    & h_n(x) =\branch{\exp(-\tfrac{1}{2}x^\T A_n x) 2\pi(\theta_0)}{\text{if } |x|<\epsilon\sqrt n,}
                     {e^{-n\delta/2} \pi(\theta_n + x/\sqrt n)}{\text{if } |x|\geq\epsilon\sqrt n,}\\
& h_0(x) = \exp(-\tfrac{1}{2} x^\T A_0 x) 2\pi(\theta_0).
\end{align*}
We will show that 
\begin{enumerate}[(a)]
\item $g_n\to g_0$ and $h_n\to h_0$ pointwise,
\item $\int h_n\to\int h_0$,
\item $g_n = |g_n| \leq h_n$ for all $n$ sufficiently large, and
\item $g_n,g_0,h_n,h_0\in L^1$ for all $n$ sufficiently large.
\end{enumerate}
By the generalized dominated convergence theorem, this will imply that $\int g_n\to\int g_0$ and $\int|g_n - g_0|\to 0$
\citep[e.g.,][exercises 2.20, 2.21]{folland2013real}.
Supposing this for the moment, we show how the result follows.
Since $\int q_n = 1$, by Equation \ref{equation:g_n} we have
\begin{align}\label{equation:z_n}
e^{n f_n(\theta_n)} n^{D/2} z_n = \sint g_n\longrightarrow \sint g_0 = \pi(\theta_0)\frac{(2\pi)^{D/2}}{|H_0|^{1/2}},
\end{align}
where $|H_0| = |\det H_0|$, and hence,
$$ z_n\sim \frac{e^{- n f_n(\theta_n)}\pi(\theta_0)}{|H_0|^{1/2}}\Big(\frac{2\pi}{n}\Big)^{D/2} $$
as $n\to\infty$; this proves Equation~\ref{equation:Laplace}. 
For any $a_n\to a \in\R$, we have $\int |a_n g_n - a g_0|\to 0$ since
$$ \int |a_n g_n - a g_0| \leq \int |a_n g_n - a_n g_0| + \int |a_n g_0 - a g_0| = |a_n|\int |g_n - g_0| + |a_n - a|\int |g_0| \longrightarrow 0. $$
Thus, letting
$1/a_n = e^{n f_n(\theta_n)} n^{D/2} z_n$ and $1/a =\pi(\theta_0)\frac{(2\pi)^{D/2}}{|H_0|^{1/2}}$,
we have $a_n\to a$ by Equation~\ref{equation:z_n}, and thus,
$$ \int \Big\vert q_n(x) -\frac{|H_0|^{1/2}}{(2\pi)^{D/2}} \exp(-\tfrac{1}{2}x^\T H_0 x)\Big\vert d x \longrightarrow 0, $$
proving Equation~\ref{equation:normality}.
Equation~\ref{equation:concentration} (concentration at $\theta_0$) follows by Lemma~\ref{lemma:BVM-implies-concentration}, since $\theta_n\to\theta_0$.
It remains to show (a)--(d) above.

(a) Fix $x\in\R^D$. First, consider $h_n$. For all $n$ sufficiently large, $|x|<\epsilon\sqrt n$, and thus,
$$ h_n(x) = \exp(-\tfrac{1}{2}x^\T A_n x) 2\pi(\theta_0) \longrightarrow \exp(-\tfrac{1}{2}x^\T A_0 x) 2\pi(\theta_0) = h_0(x) $$
since $A_n \to A_0$. Now, for $g_n$, first note that $\pi(\theta_n + x/\sqrt n) \to \pi(\theta_0)$ since $\pi$ is continuous at $\theta_0$
and $\theta_n\to\theta_0$, $x/\sqrt n\to 0$. By the assumed representation of $f_n$ (Equation \ref{equation:f_n}),
$$ n (f_n(\theta_n + x/\sqrt n) - f_n(\theta_n)) =\tfrac{1}{2} x^\T H_n x + n r_n(x/\sqrt n) \longrightarrow \tfrac{1}{2} x^\T H_0 x $$
since $H_n\to H_0$ and for all $n$ sufficiently large (to ensure that $|x/\sqrt n|<\epsilon_0$ and the assumed bound on $r_n$ holds),
\begin{align}\label{equation:r_n}
|n r_n(x/\sqrt n)|\leq n c_0|x/\sqrt n|^3 = c_0|x|^3/\sqrt n \to 0
\end{align}
as $n\to\infty$.  Hence, $g_n(x)\to g_0(x)$. 

(b) By the definition of $h_n$, letting $B_n = B_{\epsilon\sqrt n}(0)$, 
$$\sint h_n =\int_{B_n}\exp(-\tfrac{1}{2} x^\T A_n x) 2\pi(\theta_0) d x
+ \int_{B_n^c} e^{- n\delta/2}\pi(\theta_n + x/\sqrt n) d x.$$
Since $A_n\to A_0$ and $A_0$ is positive definite, then for all $n$ sufficiently large, $A_n$ is also positive definite
and the first term equals
$$ 2\pi(\theta_0)\frac{(2\pi)^{D/2}}{|A_n|^{1/2}}\Pr(|A_n^{-1/2} Z|<\epsilon\sqrt n)
\longrightarrow 2\pi(\theta_0)\frac{(2\pi)^{D/2}}{|A_0|^{1/2}}=\sint h_0 $$
where $Z\sim\Normal(0, I)$.
The second term goes to zero, since it is nonnegative and upper bounded by
$$ \int_{\R^D} e^{- n\delta/2}\pi(\theta_n + x/\sqrt n) d x = e^{- n\delta/2} n^{D/2}\longrightarrow 0, $$
using the fact that $\pi(\theta_n + x/\sqrt n) n^{-D/2}$ is the density of $X =\sqrt n(\theta-\theta_n)$ when $\theta\sim\pi$. 

(c) For all $n$ sufficiently large, $|\theta_n - \theta_0| < \epsilon$, the bound on $r_n$ applies, and
$\inf_{\theta\in B_\epsilon(\theta_n)^c} f_n(\theta) - f_n(\theta_n) > \delta/2$.  
Let $n$ large enough that these hold, and let $x\in\R^D$. If $|x|\geq \epsilon \sqrt n$,
then $f_n(\theta_n + x/\sqrt n) - f_n(\theta_n) > \delta/2$, and thus,
$$ g_n(x) \leq e^{-n \delta/2} \pi(\theta_n + x/\sqrt n) = h_n(x). $$
Meanwhile, if $|x| < \epsilon\sqrt n$, then $\pi(\theta_n + x/\sqrt n) \leq 2\pi(\theta_0)$ (by our choice of $\epsilon$, since
$|(\theta_n + x/\sqrt n) - \theta_0| \leq |\theta_n - \theta_0| + |x/\sqrt n| < 2\epsilon$), and 
$$ n (f_n(\theta_n + x/\sqrt n) - f_n(\theta_n)) =\tfrac{1}{2} x^\T H_n x + n r_n(x/\sqrt n)
\geq \tfrac{1}{2} x^\T H_n x - \tfrac{1}{2}\alpha x^\T x = \tfrac{1}{2} x^\T A_n x$$
since $|n r_n(x/\sqrt n)| \leq c_0 |x|^3/\sqrt n \leq c_0\epsilon |x|^2 \leq \tfrac{1}{2}\alpha |x|^2$,
by the fact that $|x/\sqrt n|<\epsilon<\epsilon_0$ and $\epsilon < \alpha/(2 c_0)$.
Therefore, 
$$ g_n(x) \leq \exp(-\tfrac{1}{2}x^\T A_n x) 2\pi(\theta_0) = h_n(x). $$

(d) Since $H_0$ and $A_0$ are positive definite, $\int g_0$ and $\int h_0$ are finite.
By (b) and (c), since $\int h_n\to\int h_0 < \infty$, we have $\int g_n \leq \int h_n < \infty$ for all $n$ sufficiently large.
Measurability of $g_n$ and $h_n$ follows from measurability of $f_n$ and $\pi$.
\end{proof}

\begin{proof}{\bf of Theorem \ref{theorem:altogether}}~
Without loss of generality, we may assume $E$ is convex, since otherwise 
we can choose $E'\subseteq E$ to be an open ball around $\theta_0$, and proceed with $E'$ in place of $E$ throughout the proof.
First, we show that under case \ref{item:altogether-convex}, the conditions for case \ref{item:altogether-liminf} hold. 
By Lemma \ref{lemma:convex-minimum}(\ref{item:convex-minimum-1}), $f(\theta) > f(\theta_0)$ for all $\theta\in E\setminus\{\theta_0\}$
since $f'$ exists on $E$ by Theorem~\ref{theorem:regular-convergence}.
Letting $K = \overline{B_\epsilon(\theta_0)}$ where $\epsilon>0$ is small enough that $K\subseteq E$, 
we have $\liminf_n\inf_{\theta\in\Theta\setminus K} f_n(\theta) > f(\theta_0)$
by Lemma \ref{lemma:convex-minimum}(\ref{item:convex-minimum-2}).
Thus, it suffices to prove the result under case \ref{item:altogether-liminf}.

Consider case \ref{item:altogether-liminf}.
Extend $\pi$, $f_n$, and $f$ to all of $\R^D$ by defining $\pi(\theta) = 0$ and $f(\theta) = f_n(\theta) = f(\theta_0) + 1$ for all
$\theta\in\R^D\setminus\Theta$.  Then all the conditions of Theorem \ref{theorem:altogether} (under case \ref{item:altogether-liminf})
still hold with $\R^D$ in place of $\Theta$.
We will show that:
\begin{enumerate}
    \item[(a)] $(f_n)$ is equicontinuous on $E$, and $f_n''(\theta_0)\to f''(\theta_0)$ as $n\to\infty$,
    \item[(b)] there exist $\theta_n\in E$ such that $\theta_n\to\theta_0$ and $f_n'(\theta_n) = 0$ for all $n$ sufficiently large, and
    \item[(c)] $f_n(\theta_n)\to f(\theta_0)$.
\end{enumerate}
Assuming (a)--(c) for the moment, we show how the result follows.
Letting $H_0 = f''(\theta_0)$, the conditions of Theorem \ref{theorem:natural-sufficient} are satisfied, and thus, 
condition \ref{item:BVM-rep} of Theorem \ref{theorem:BVM} is satisfied for all $n$ sufficiently large.
Condition \ref{item:BVM-liminf} of Theorem \ref{theorem:BVM} holds, since for all $\epsilon>0$,
\begin{align*}
\liminf_n\inf_{\theta\in B_\epsilon(\theta_n)^c} (f_n(\theta) - f_n(\theta_n))
&= \Big(\liminf_n\inf_{\theta\in B_\epsilon(\theta_n)^c} f_n(\theta)\Big) - f(\theta_0) \\
&\geq \Big(\liminf_n\inf_{\theta\in B_{\epsilon/2}(\theta_0)^c} f_n(\theta)\Big) - f(\theta_0) > 0
\end{align*}
the first step holding by (c), the second step since $\theta_n\to\theta_0$ and thus $B_{\epsilon/2}(\theta_0)\subseteq B_\epsilon(\theta_n)$ 
for all $n$ sufficiently large, and the third step by the implication $2\Rightarrow 1$ in Theorem \ref{theorem:concentration}.
Thus, the conditions of Theorem \ref{theorem:BVM} are satisfied (except possibly for some initial sequence of $n$'s, which can be ignored
since the conclusions are asymptotic in nature), establishing 
Equation~\ref{equation:concentration} (concentration at $\theta_0$),
Equation \ref{equation:Laplace} (the Laplace approximation),
and Equation \ref{equation:normality} (asymptotic normality).
To complete the proof, we establish (a), (b), and (c).

(a) By Theorem \ref{theorem:regular-convergence}, $(f_n)$ is equi-Lipschitz (hence, equicontinuous) on $E$ and $f_n''\to f''$ uniformly on $E$.

(b) Let $\epsilon>0$ small enough that $S_\epsilon\subseteq K$ where $S_\epsilon = \{\theta\in\R^D : |\theta - \theta_0| = \epsilon\}$.
By Theorem \ref{theorem:regular-convergence}, $f$ is continuous on $E$ (since $f'$ exists on $E$).
Thus, $f$ attains its minimum on the compact set $S_\epsilon$, and since $f(\theta)>f(\theta_0)$ on $S_\epsilon$, we have
$\inf_{\theta\in S_\epsilon} f(\theta) > f(\theta_0)$. For each $n$, since $f_n$ is continuous on $E$, its minimum over the set
$\overline{B_\epsilon(\theta_0)}$ is attained at one or more points; define $\theta_n^\epsilon$ to be such a minimizer.
Since $f_n\to f$ uniformly on $E$ (by Theorem \ref{theorem:regular-convergence}), then for all $n$ sufficiently large,
any such minimizer cannot be in $S_\epsilon$ (since $\inf_{\theta\in S_\epsilon} f(\theta) > f(\theta_0)$).
Hence, for all sufficiently small $\epsilon>0$, for all $n$ sufficiently large, we have $\theta_n^\epsilon\in B_\epsilon(\theta_0)$ and 
(by Lemma \ref{lemma:smooth-min}) $f_n'(\theta_n^\epsilon) = 0$.

Thus, we can choose a sequence $\epsilon_n>0$ such that (a) $\epsilon_n \to 0$ and (b) for all $n$ sufficiently large,
$\theta_n^{\epsilon_n} \in B_{\epsilon_n}(\theta_0)$ and $f_n'(\theta_n^{\epsilon_n}) = 0$.
Therefore, letting $\theta_n = \theta_n^{\epsilon_n}$, we have $\theta_n\to\theta_0$ and $f_n'(\theta_n) = 0$ for all $n$ sufficiently large.

(c) We have $|f_n(\theta_n) - f(\theta_0)| \leq |f_n(\theta_n) - f_n(\theta_0)| + |f_n(\theta_0) - f(\theta_0)| \to 0$, the first term going to zero
since $\theta_n\to\theta_0$ and $(f_n)$ is equi-Lipschitz on $E$, and the second term since $f_n\to f$ pointwise.
\end{proof}

For tensors $S,T\in\R^{D^3}$, define the inner product $\langle S,T\rangle = \sum_{i,j,k} S_{i j k} T_{i j k}$ 
(noting that this is just the dot product of the vectorized versions of $S$ and $T$).
For $x\in\R^D$, define $x^{\otimes 3} = x\otimes x \otimes x = \big(x_i x_j x_k\big)_{i,j,k=1}^D \in \R^{D^3}$, and note that $\|x^{\otimes 3}\|=|x|^3$.

\vspace{1em}
\begin{proof}{\bf of Theorem \ref{theorem:natural-sufficient}}~
By Lemma \ref{lemma:UBD}, $(f_n'')$ is equi-Lipschitz. Thus,
$$ \|f_n''(\theta_n) - H_0\| \leq 
   \|f_n''(\theta_n) - f_n''(\theta_0)\| + \|f_n''(\theta_0) - H_0\| \leq C|\theta_n - \theta_0| + \|f_n''(\theta_0) - H_0\|\longrightarrow 0, $$
and hence, $H_n\to H_0$. 
Let $C_0 =\sup_n\sup_{\theta\in E}\|f_n'''(\theta)\|$. Let $n$ large enough that $f_n'(\theta_n) = 0$. For $\theta\in E$, by Taylor's theorem,
$$ f_n(\theta) = f_n(\theta_n) +\tfrac{1}{2}(\theta -\theta_n)^\T f_n''(\theta_n)(\theta -\theta_n) + r_n(\theta -\theta_n) $$
where $r_n(\theta -\theta_n) =\tfrac{1}{6}\langle f_n'''(t_n(\theta)),(\theta -\theta_n)^{\otimes 3}\rangle$, and 
$t_n(\theta)$ is a point on the line between $\theta$ and $\theta_n$. Then by Cauchy--Schwarz,
\begin{align}\label{equation:rntt}
|r_n(\theta -\theta_n)|\leq\tfrac{1}{6}\|f_n'''(t_n(\theta))\|\|(\theta -\theta_n)^{\otimes 3}\| 
\leq \tfrac{1}{6}C_0\|(\theta -\theta_n)^{\otimes 3}\| = \tfrac{1}{6}C_0|\theta -\theta_n|^3.
\end{align}
Choose $\epsilon_0>0$ small enough that $B_{2\epsilon_0}(\theta_0)\subseteq E$, and choose $c_0 = C_0/6$. 
For all $n$ sufficiently large, $|\theta_n - \theta_0|\leq \epsilon_0$ 
 and hence for all $x\in B_{\epsilon_0}(0)$, we have $\theta_n + x\in B_{2\epsilon_0}(\theta_0)\subseteq E$;
thus, setting $\theta = \theta_n + x$ in Equation \ref{equation:rntt} yields $|r_n(x)|\leq c_0|x|^3$.
\end{proof}

\section{Proof of regular convergence theorem}
\label{section:regular-convergence-proof}

\begin{lemma}\label{lemma:derivative-bound}
Let $E\subseteq\R^D$ be open.  If $f_n:E\to\R$ has continuous second derivatives, $(f_n)$ is pointwise bounded,
and $(f_n'')$ is uniformly bounded, then $(f_n')$ is pointwise bounded.
\end{lemma}
\begin{proof}
Let $C = \sup\{\|f_n''(x)\| : n\in\N, x\in E\} < \infty$. Fix $x\in E$, and let $\epsilon>0$ small enough that $B_{2\epsilon}(x)\subseteq E$.
By Taylor's theorem, for any $u\in\R^D$ with $|u|= 1$,
$$ f_n(x +\epsilon u) = f_n(x) +\epsilon f_n'(x)^\T u +\tfrac{1}{2}\epsilon^2 u^\T f_n''(z) u $$
for some $z$ on the line between $x$ and $x +\epsilon u$, and therefore,
$$|f_n'(x)^\T u|\leq (1/\epsilon)|f_n(x +\epsilon u) - f_n(x)| + \tfrac{1}{2}\epsilon C $$
since $|u^\T f_n''(z) u|\leq\|f_n''(z)\| |u|^2\leq C$. Thus, $\{f_n'(x)^\T u : n\in\N\}$ is bounded, for any $u$ with $|u|= 1$.
Applying this to each element of the standard basis, we see $f_n'(x)$ is bounded.
\end{proof}

\begin{lemma}\label{lemma:Hessian-bound}
Let $E\subseteq\R^D$ be open.  If $f_n:E\to\R$ has continuous third derivatives, $(f_n)$ is pointwise bounded,
and $(f_n''')$ is uniformly bounded, then $(f_n'')$ is pointwise bounded.
\end{lemma}
\begin{proof}
Let $C = \sup_n \sup_{x\in E} \|f_n'''(x)\| <\infty$. Fix $x\in E$, and let $\epsilon>0$ small enough
that $\overline{B_\epsilon(x)}\subseteq E$. By Taylor's theorem, for any $u\in\R^D$ with $|u|= 1$,
$$ f_n(x +\epsilon u) = f_n(x) +\epsilon f_n'(x)^\T u +\tfrac{1}{2}\epsilon^2 u^\T f_n''(x) u 
+ \tfrac{1}{6}\epsilon^3 \langle f_n'''(z^+), u^{\otimes 3}\rangle$$
for some $z^+$ on the line between $x$ and $x +\epsilon u$. Likewise,
$$ f_n(x -\epsilon u) = f_n(x) -\epsilon f_n'(x)^\T u +\tfrac{1}{2}\epsilon^2 u^\T f_n''(x) u 
- \tfrac{1}{6}\epsilon^3 \langle f_n'''(z^-), u^{\otimes 3}\rangle$$
for some $z^-$ on the line between $x$ and $x -\epsilon u$. Adding these two equations gives
$$ f_n(x +\epsilon u) + f_n(x -\epsilon u) = 2 f_n(x) +\epsilon^2 u^\T f_n''(x) u 
+\tfrac{1}{6}\epsilon^3\langle f_n'''(z^+) - f_n'''(z^-),\, u^{\otimes 3}\rangle. $$
For any tensor $T \in \R^{D^3}$, $|\langle T, u^{\otimes 3}\rangle| \leq \|T\|\|u^{\otimes 3}\| = \|T\|$, by the Cauchy--Schwarz inequality. Therefore,
$$|u^\T f_n''(x) u|\leq (1/\epsilon^2)|f_n(x +\epsilon u) + f_n(x -\epsilon u) - 2 f_n(x)| + \tfrac{1}{3}\epsilon C. $$
Thus, since $(f_n)$ is pointwise bounded, this implies that $\{u^\T f_n''(x) u : n\in\N\}$ is bounded, for any $u$ with $|u|= 1$.
Let $u_1,\ldots,u_k\in\R^D$, with $|u_i|=1$, such that $u_1 u_1^\T,\ldots,u_k u_k^\T$ is a basis for the vector space
$V\subseteq\R^{D\times D}$ of symmetric matrices.
(This is possible since $\lspan\{u u^\T : |u|=1\} = V$ by the spectral decomposition theorem.)
With $\langle A,B \rangle := \sum_{i,j} A_{ij} B_{ij}$, $V$ is an inner product space.
Since $\{u_i^\T f_n''(x) u_i : n\in\N\}$ is bounded for each $i$, and $u_i^\T f_n''(x) u_i = \langle u_i u_i^\T, f_n''(x)\rangle$, 
then by Lemma \ref{lemma:basis-bound}, $\{f_n''(x):n\in\N\}$ is bounded.
Since $x$ is arbitrary, $(f_n'')$ is pointwise bounded.
\end{proof}

\begin{lemma}\label{lemma:basis-bound}
Suppose $V$ is a finite-dimensional inner product space over $\R$, and let $e_1,\ldots,e_k\in V$ be a basis.
If $S\subseteq V$ such that $\{\langle e_i,x \rangle : x\in S\}$ is bounded for each $i = 1,\ldots,k$, then $S$ is bounded.
\end{lemma}
\begin{proof}
Let $G$ be the Gram matrix of $(e_i)$, i.e., $G_{ij} = \langle e_i,e_j \rangle$. Note that $G$ is positive definite, since for any $a\in\R^k$,
\begin{align}\label{equation:aGa}
a^\T G a =\sum_{i,j} a_i a_j G_{ij} =\sum_{i,j} \langle a_i e_i, a_j e_j \rangle 
= \textstyle \big\langle\sum_i a_i e_i,\sum_j a_j e_j \big\rangle =\|\sum_i a_i e_i\|^2\geq 0,
\end{align}
with equality if and only if $\sum_i a_i e_i = 0$, that is, if and only if $a = 0$ (since $(e_i)$ is a linearly independent set).
For $x\in V$, define $a(x)\in\R^k$ by the property that $\sum_i a_i(x) e_i = x$ (noting that $a(x)$ always exists and is unique, since $(e_i)$ is a basis). Define $b(x)\in\R^k$ such that $b_i(x) = \langle e_i,x \rangle$. Then for any $x\in V$,
$$ b_i(x) = \big\langle e_i, {\textstyle\sum_j} a_j(x) e_j \big\rangle = \sum_j a_j(x)\langle e_i,e_j\rangle = \sum_j a_j(x)G_{ij}, $$
and thus, $b(x) = G a(x)$. Hence, $a(x) = G^{-1} b(x)$, so by Equation \ref{equation:aGa},
$$ \|x\|^2 = a(x)^\T G a(x) = b(x)^\T G^{-1} b(x) \leq \|G^{-1}\| |b(x)|^2. $$
By assumption, $\{|b(x)| : x \in S\}$ is bounded, hence, $\{\|x\| : x \in S\}$ is bounded.
\end{proof}

\begin{lemma}\label{lemma:derivative-sandwich}
Let $E\subseteq \R^D$ be open, convex, and bounded.  Let $f_n:E\to\R$ have continuous second derivatives.
If $f_n\to f$ pointwise for some $f:E\to\R$, and $(f_n'')$ is uniformly bounded, then $f'$ exists and is continuous, and $f_n'\to f'$ uniformly.
\end{lemma}
\begin{proof}
First, we show that $(f_n')$ converges pointwise.
Let $C =\sup_n\sup_{x\in E}\|f_n''(x)\|<\infty$. Let $x\in E$, and let $\epsilon>0$ small enough that $\overline{B_{\epsilon}(x)}\subseteq E$.
Then for any $u\in\R^D$ with $|u|= 1$, for any $m,n$, by applying Taylor's theorem to $f_m - f_n$,
$$ f_m(x +\epsilon u) - f_n(x +\epsilon u) 
= f_m(x) - f_n(x) + (f_m'(x) - f_n'(x))^\T(\epsilon u) +\tfrac{1}{2}(\epsilon u)^\T(f_m''(z) - f_n''(z))(\epsilon u) $$
for some $z$ on the line between $x$ and $x +\epsilon u$. Thus,
$$|(f_m'(x) - f_n'(x))^\T u|\leq\frac{1}{\epsilon}|f_m(x +\epsilon u) - f_n(x +\epsilon u)|+\frac{1}{\epsilon}|f_m(x) - f_n(x)|
+\tfrac{1}{2}\epsilon\|f_m''(z) - f_n''(z)\|.$$
The first two terms on the right go to zero as $m, n\to\infty$ (by pointwise convergence of $f_n$), and 
$\|f_m''(z)-f_n''(z)\|\leq\|f_m''(z)\|+\|f_n''(z)\|\leq 2 C$, therefore, $\limsup_{m,n\to\infty} |(f_m'(x) - f_n'(x))^\T u|\leq\epsilon C$. Since $\epsilon$ can be arbitrarily small, $|(f_m'(x) - f_n'(x))^\T u| \to 0$ as $m,n\to\infty$. Choosing $u =(1,0,0,\ldots,0)^\T$, then $u =(0,1,0,\ldots,0)^\T$, and so on, this implies $|f_m'(x) - f_n'(x)| \to 0$ as $m,n\to\infty$, and hence, $f_n'(x)$ converges.

Next, by Lemma \ref{lemma:UBD}, $(f_n')$ is equi-Lipschitz, and hence, equicontinuous.
Thus, in fact, $(f_n')$ converges uniformly, by Lemma \ref{lemma:equicontinuous-convergence}.
Finally, we show that $f'$ exists and $f_n'\to f'$ uniformly; it will follow that $f'$ is continuous,
as the limit of a uniformly convergent sequence of continuous functions.

Let $C_{m n} =\sup_{x\in E}|f_m'(x) - f_n'(x)|$. Then $C_{m n}\to 0$ as $m,n\to\infty$, by uniform convergence.
To establish the result, it suffices to show that for any $x_0\in E$, $f'(x_0)$ exists and $f_n'(x_0)\to f'(x_0)$. 
Fix $x_0\in E$, and let $B = B_\epsilon(x_0)\setminus\{x_0\}$ where $\epsilon>0$ is small enough that $B\subseteq E$.
For $x\in B$, define $\varphi_n(x) = (f_n(x) - f_n(x_0))/|x-x_0|$ and $\varphi(x) = (f(x)-f(x_0))/|x-x_0|$,
noting that $\varphi_n \to \varphi$ pointwise. 
For any $x \in B$, by Taylor's theorem applied to $f_m - f_n$, 
$$ f_m(x) - f_n(x) = f_m(x_0) - f_n(x_0) + (f_m'(z) - f_n'(z))^\T(x - x_0) $$
for some $z$ on the line between $x$ and $x_0$, and hence, 
$$|\varphi_m(x) - \varphi_n(x)| \leq |f_m'(z) - f_n'(z)| \leq C_{m n} \longrightarrow 0 $$
as $m,n\to\infty$. Therefore, $\varphi_n \to \varphi$ uniformly (on $B$) \citep[by e.g.,][7.8]{rudin1976principles}. 

Now, define $\psi_n(x) = f_n'(x_0)^\T(x-x_0)/|x-x_0|$ and $\psi(x) = v^\T(x-x_0)/|x-x_0|$ for $x\in B$, where $v = \lim_n f_n'(x_0)$.
Since $|\psi_n(x) - \psi(x)|\leq |f_n'(x_0)-v| \to 0$ as $n\to\infty$, then $\psi_n\to\psi$ uniformly as well.
Hence, $|\varphi_n - \psi_n| \to |\varphi - \psi|$ uniformly (on $B$). 

By the definition of the derivative $f_n'(x_0)$,
$$ |\varphi_n(x) - \psi_n(x)| = \frac{|f_n(x) - f_n(x_0) - f_n'(x_0)^\T(x-x_0)|}{|x-x_0|} \xrightarrow[x\to x_0]{} 0. $$
Therefore \citep[by e.g.,][7.11]{rudin1976principles},
\begin{align*}
0 = \lim_{n\to\infty} \lim_{x\to x_0} |\varphi_n(x) - \psi_n(x)|
&= \lim_{x\to x_0} \lim_{n\to\infty} |\varphi_n(x) - \psi_n(x)| = \lim_{x\to x_0} |\varphi(x) - \psi(x)| \\
&= \lim_{x\to x_0} \frac{|f(x) - f(x_0) - v^\T(x-x_0)|}{|x-x_0|}.
\end{align*}
Hence, $f'(x_0)$ exists and equals $v = \lim_n f_n'(x_0)$.
\end{proof}

\begin{proof}{\bf of Theorem \ref{theorem:regular-convergence}}~
First, suppose $(f_n)$ is pointwise bounded.
By Lemma \ref{lemma:UBD} with $k=3$, $(f_n'')$ is equi-Lipschitz,
and by Lemma \ref{lemma:Hessian-bound}, $(f_n'')$ is pointwise bounded.
Thus, since $E$ is bounded, it follows that $(f_n'')$ is uniformly bounded. 
Therefore, by Lemma \ref{lemma:UBD} with $k=2$, $(f_n')$ is equi-Lipschitz,
and by Lemma \ref{lemma:derivative-bound}, $(f_n')$ is pointwise bounded.
Thus, likewise, $(f_n')$ is uniformly bounded. 
And lastly, applying Lemma \ref{lemma:UBD} with $k=1$, we have that $(f_n)$ is equi-Lipschitz,
and hence, uniformly bounded, 
since it is pointwise bounded by assumption.

Now, assume $f_n\to f$ pointwise. 
Then in fact, $f_n\to f$ uniformly, by Lemma \ref{lemma:equicontinuous-convergence}, 
since $(f_n)$ is equi-Lipschitz (as just established), and hence, equicontinuous. 
By Lemma \ref{lemma:derivative-sandwich}, $f'$ exists and $f_n'\to f'$ uniformly.
To complete the proof, we show that $f''$ exists and $f_n''\to f''$ uniformly.
For any $i\in\{1,\ldots,D\}$, if we define $h_n(x) = f_n'(x)_i$ and $h(x) = f'(x)_i$, then $h_n\to h$ pointwise and
$(h_n'')$ is uniformly bounded (since $(f_n''')$ is uniformly bounded and $\|h_n''(x)\|\leq\|f_n'''(x)\|$);
hence, by Lemma \ref{lemma:derivative-sandwich}, $h'$ exists and is continuous, and $h_n'\to h'$ uniformly.
Since this holds for each coordinate $i$, then $f''$ exists, and $f_n''\to f''$ uniformly.
\end{proof}

\section{Proofs of coverage results}
\label{section:coverage-proof}

\begin{proof}{\bf of Theorem \ref{theorem:coverage}}~
Letting $X_n = -\sqrt{n}(\theta_n - \theta_0)$ and $X\sim Q$,
$$ \Pr(\theta_0\in S_n) \overset{\text{(a)}}{=} \Pr(\sqrt{n}(\theta_0 - \theta_n)\in R_n) = \Pr(X_n\in R_n) \overset{\text{(b)}}{\longrightarrow} \Pr(X\in R) = Q(R)$$
where step (a) is by the definition of $R_n$, and
(b) is by Lemma~\ref{lemma:convergent-set}, using conditions \ref{item:coverage1} ($X_n\xrightarrow[]{\mathrm{D}} X$), \ref{item:coverage3}, and \ref{item:coverage4}.
To see that $Q(R) = \rho$, note that $\Pi_n(S_n) \xrightarrow[]{\mathrm{a.s.}} \rho$ by assumption and also
$\Pi_n(S_n) = Q_n(R_n) \xrightarrow[]{\mathrm{a.s.}} Q(R)$ since
\begin{align*}
|Q_n(R_n) - Q(R)| &\leq |Q_n(R_n) - Q(R_n)| + |Q(R_n) - Q(R)| \\
&\leq \sup_{A\in\mathcal{B}}|Q_n(A) - Q(A)| + |Q(R_n) - Q(R)| \xrightarrow[]{\mathrm{a.s.}} 0
\end{align*}
by condition \ref{item:coverage2} and condition \ref{item:coverage3} plus the dominated convergence theorem \citep[Theorem 2.24]{folland2013real}.
\end{proof}

\begin{proof}{\bf of Lemma \ref{lemma:convergent-set}}~
For each $k=1,2,\ldots$, define $A_k = \{x\in \R^D : d(x,R^c) > 1/k\}$ and $B_k = \{x\in\R^D : d(x,R) \leq 1/k\}$.
Note that $A_k$ is open and $B_k$ is closed since $x\mapsto d(x,R)$ and $x\mapsto d(x,R^c)$ are continuous.
For any $k$, by Lemma~\ref{lemma:squeeze} we have that with probability $1$, for all $n$ sufficiently large, $A_k\subseteq R_n \subseteq B_k$.
Thus, with probability $1$, $\liminf_n \big(\I(X_n\in R_n) - \I(X_n\in A_k)\big) \geq \liminf_n \inf_x \big(\I(x\in R_n) - \I(x\in A_k)\big) \geq 0$.
It follows that 
$$\liminf_n \E\big(\I(X_n \in R_n) - \I(X_n\in A_k)\big) \geq \E \liminf_n \big(\I(X_n\in R_n) - \I(X_n\in A_k)\big) \geq 0$$
by Fatou's lemma applied to $\I(X_n\in R_n) - \I(X_n\in A_k) + 1$.
(The $+1$ is added to make the function nonnegative, so that Fatou's lemma applies directly.)
If $\liminf_n (a_n - b_n) \geq 0$ then $\liminf a_n = \liminf (a_n - b_n + b_n) \geq \liminf (a_n - b_n) + \liminf b_n \geq \liminf b_n$.
Therefore,
$\liminf_{n\to\infty} \Pr(X_n\in A_k) \leq \liminf_{n\to\infty} \Pr(X_n \in R_n).$
Similarly, by reverse Fatou's lemma,
$$\limsup_n \E\big(\I(X_n \in R_n) - \I(X_n\in B_k)\big) \leq \E \limsup_n \big(\I(X_n\in R_n) - \I(X_n\in B_k)\big) \leq 0,$$
and therefore, $\limsup_n \Pr(X_n\in R_n) \leq \limsup_n \Pr(X_n\in B_k)$.
Hence, by the portmanteau theorem \citep[Theorem 11.1.1]{dudley2002real}, for all $k$,
\begin{align*}
\Pr(X\in A_k) & \leq \liminf_n \Pr(X_n\in A_k) \leq \liminf_n \Pr(X_n\in R_n) \\
&\leq \limsup_n \Pr(X_n\in R_n) \leq \limsup_n \Pr(X_n\in B_k) \leq \Pr(X \in B_k).
\end{align*}
Taking limits as $k\to\infty$ and using the fact that $\bigcup_{k=1}^\infty A_k = R^\circ$ and $\bigcap_{k=1}^\infty B_k = \bar{R}$,
we have
$\Pr(X\in R^\circ) = \lim_k \Pr(X\in A_k) \leq \liminf_n \Pr(X_n\in R_n) \leq \limsup_n \Pr(X_n\in R_n) \leq \lim_k \Pr(X\in B_k) = \Pr(X\in \bar{R})$
by \citet[Theorem 1.8]{folland2013real}.
Further, $\Pr(X\in R^\circ) = \Pr(X\in R) = \Pr(X\in \bar{R})$ since $\Pr(X\in \partial R) = 0$.
Therefore, $\lim_n \Pr(X_n\in R_n) = \Pr(X\in R)$.
\end{proof}

\begin{proof}{\bf of Lemma \ref{lemma:squeeze}}~
First, we establish some initial facts.  
It is straightforward to check that $R$ is convex.
$R^\circ$ is nonempty
since $m(\bar R) \geq m(R) > 0$ and $m(\partial R) = 0$ \citep{lang1986note}.
It follows that $R$, $A$, and $B$ are bounded.
For any open cube $E$ such that $\bar{E}\subseteq R$, we have $E\subseteq R_n$ for all $n$ sufficiently large,
since $\I(x\in R_n) \to \I(x\in R)$ for each corner $x$ of the cube $E$.

Next, we show that $A\subseteq R_n$ for all $n$ sufficiently large.
For each $x\in \bar{A}$, let $E_x$ be a nonempty open cube centered at $x$ such that $\bar{E_x} \subseteq R$.
Then $\{E_x : x\in \bar{A}\}$ is an open cover of $\bar{A}$.
Since $\bar{A}$ is compact, there is a finite subcover $E_{x_1},\ldots,E_{x_k}$.
Thus, for all $n$ sufficiently large, $A\subseteq \bar{A} \subseteq \bigcup_{i=1}^k E_{x_i} \subseteq R_n$.

Now, we show that $R_n\subseteq B$ for all $n$ sufficiently large.
Let $S_\delta = \{x\in \R^D : d(x,R) = \delta\}$ for $\delta>0$.
Let $E\subseteq R$ be a nonempty open cube such that $E\subseteq R_n$ for all $n$ sufficiently large.
For each $x\in S_{\epsilon/2}$, define $C_x = \bigcup_{t>1} \{t x + (1-t) z : z\in E\}$.
Then $C_x$ is open, as a union of open sets.
Note that $y\in C_x$ if and only if $x = s y + (1-s) z$ for some $s\in(0,1)$, $z\in E$,
i.e., if and only if $x$ is a (strict) convex combination of $y$ and some point of $E$.
Thus, $\{C_x : x\in S_{\epsilon/2}\}$ is an open cover of $S_\epsilon$
(since for any $y\in S_\epsilon$, the line between $y$ and any $z\in E$ must pass through $S_{\epsilon/2}$ 
by the intermediate value theorem applied to $s\mapsto d(s x + (1-s)z, R)$).
Since $S_\epsilon$ is compact, there is a finite subcover $C_{x_1},\ldots,C_{x_k}$ for some $x_1,\ldots,x_k \in S_{\epsilon/2}$.
Since $x_i\in R^c$ for each $i=1,\ldots,k$, there exists $N$ such that for all $n\geq N$, $x_1,\ldots,x_k\in R_n^c$ and $E\subseteq R_n$.
Then for all $n\geq N$, by the convexity of $R_n$, we have $S_\epsilon \subseteq \bigcup_{i=1}^k C_{x_i} \subseteq R_n^c$ and hence $R_n \subseteq B$.
\end{proof}

\section{Proofs for generalized applications}
\label{section:generalized-proof}

\begin{proof}{\bf of Theorem~\ref{theorem:grf-ergodic}}~
Define $B_r = \{j\in\Z^m : R(j) \leq r\}$, that is, $B_r = \{-r,\ldots,-1,0,1,\ldots,r\}^m$ for $r\in\N$. 
Let $r_n = R(v(n))$, $L_n = |B_{r_n-1}|$, and $M_n = |B_{r_n}|$, for $n\in\N$.
Observe that $L_n < n \leq M_n$ since $R(v(1))\leq R(v(2)) \leq \cdots$.
Further, $M_n/n\to 1$ as $n\to\infty$, since
$$ 1 \leq \frac{M_n}{n} \leq \frac{M_n}{L_n} = \frac{|B_{r_n}|}{|B_{r_n-1}|} = \Big(\frac{2 r_n + 1}{2 r_n - 1}\Big)^m \longrightarrow 1 $$
as $n\to\infty$ since $r_n\to\infty$.

Fix $k,\ell\in\{1,\ldots,m\}$, 
and define $Z_i = Y_i X_{i k} - \E(Y_i X_{i k})$ or $Z_i = X_{i k} X_{i \ell} - \E(X_{i k}X_{i \ell})$ where $X_i$ is defined as in Theorem~\ref{theorem:grf};
the proof is the same in either case.
Then $Z_1,Z_2,\ldots$ are identically distributed, and in fact,
by Lemma~\ref{lemma:stationary-ergodic},
$(Z_{v^{-1}(j)} : j\in\Z^m)$ is stationary with respect to $T_1,\ldots,T_m$ and ergodic with respect to at least one of $T_1,\ldots,T_m$.
Note that $\E Z_i = 0$ and $\E|Z_i|^2 = \mathrm{Var}(Z_i) < \infty$,
since for all $i,j$, $\mathrm{Var}(Y_i Y_j) \leq \E|Y_i Y_j|^2 \leq (\E|Y_i|^4 \E|Y_j|^4)^{1/2} < \infty$ by the Cauchy--Schwarz inequality.

To prove the result, we need to show that $\frac{1}{n}\sum_{i=1}^n Z_i \to 0$ almost surely as $n\to\infty$.
The key part is showing that (A) $\frac{1}{M_n} \sum_{i=1}^{M_n} Z_i \to 0$ by the multivariate ergodic theorem;
the remainder of the proof is showing that (B) the difference between $\frac{1}{n}\sum_{i=1}^n Z_i$ and $\frac{1}{M_n} \sum_{i=1}^{M_n} Z_i$ is negligible.

(A) For the first part, letting $\mathcal{F}_k$ denote the invariant sigma-field with respect to $T_k$,
\begin{align}
\label{equation:ergodic}
\frac{1}{M_n} \sum_{i=1}^{M_n} Z_i &= \frac{1}{|B_{r_n}|} \sum_{j\in B_{r_n}} Z_{v^{-1}(j)} \notag\\
&= \frac{1}{(2 r_n + 1)^m} \sum_{j_1=-r_n}^{r_n} \cdots \sum_{j_m=-r_n}^{r_n} Z_{v^{-1}(j_1,\ldots,j_m)} \\
&\xrightarrow[n\to\infty]{\mathrm{a.s.}} \E^{\mathcal{F}_m}\cdots\E^{\mathcal{F}_1} Z_1 = \E(Z_1 \mid \cap_k \mathcal{F}_k) = \E Z_1 = 0  \notag
\end{align}
by applying the multivariate ergodic theorem \citep[Theorems 10.12 and 10.13]{kallenberg2002foundations}
to each of the $2^m$ orthants of $\Z^m$.
A few clarifying remarks.
The subset $\{j\in\Z^m : \min\{|j_1|,\ldots,|j_m|\}=0\}$ can be handled by shifting each orthant to ensure that, collectively, they form a partition of $\Z^m$.
\citet[Theorem 10.12]{kallenberg2002foundations} shows that the limit is $\E^{\mathcal{F}_m}\cdots\E^{\mathcal{F}_1} Z_1$,
and since $T_1,\ldots,T_m$ are commutative, we have
$\E^{\mathcal{F}_m}\cdots\E^{\mathcal{F}_1} Z_1 = \E(Z_1\mid \cap_k \mathcal{F}_k)$ 
by \citet[Theorem 10.13]{kallenberg2002foundations}.
Ergodicity with respect to $T_k$ means that the corresponding invariant sigma-field $\mathcal{F}_k$ is trivial, 
that is, the probability of any set $A\in\mathcal{F}_k$ is either $0$ or~$1$.
So by assumption, at least one of $\mathcal{F}_1,\ldots,\mathcal{F}_m$ is trivial,
and hence, $\cap_k \mathcal{F}_k$ is trivial as well.
Therefore, $\E(Z_1 \mid \cap_k \mathcal{F}_k) = \E Z_1 = 0$.

(B) Now, for the remainder of the proof,
$$ \Big\vert\frac{1}{M_n}\sum_{i=1}^n Z_i\Big\vert 
\leq \Big\vert\frac{1}{M_n}\sum_{i=1}^n Z_i - \frac{1}{M_n}\sum_{i=1}^{M_n} Z_i \Big\vert + \Big\vert \frac{1}{M_n}\sum_{i=1}^{M_n} Z_i\Big\vert. $$
As for the second term, we have $\frac{1}{M_n}\sum_{i=1}^{M_n} Z_i \to 0$ a.s.\ by Equation~\ref{equation:ergodic}.
As for the first term, we have 
\begin{align*}
\Big\vert\frac{1}{M_n}\sum_{i=1}^n Z_i - \frac{1}{M_n}\sum_{i=1}^{M_n} Z_i \Big\vert
\leq \frac{1}{M_n}\sum_{i=n+1}^{M_n} |Z_i|
\leq \frac{1}{M_n} \sum_{i=L_n+1}^{M_n} |Z_i|
= c_{r_n} W_{r_n}
\end{align*}
where $c_r = |B_r\setminus B_{r-1}| / |B_r|$, $W_r = \frac{1}{|S_r|} \sum_{i\in S_r} |Z_i|$, and $S_r = \{L_n+1,\ldots,M_n\}$.
If we can show that $c_{r_n} W_{r_n} \to 0$ almost surely as $n\to\infty$, then this will prove the result, since then
$\frac{1}{n}\sum_{i=1}^n Z_i = \big(\frac{M_n}{n}\big) \frac{1}{M_n}\sum_{i=1}^n Z_i \to 0$ a.s.,
using the fact that $M_n/n\to 1$ as shown above.

We show that $c_{r_n} W_{r_n} \to 0$ using the Borel--Cantelli lemma.
For all $\varepsilon > 0$, $r\in\N$, by Markov's inequality we have
\begin{align}
\label{equation:ergodic-markov}
\Pr(c_r W_r \geq \varepsilon) = \Pr(W_r^2 \geq (\varepsilon/c_r)^2) \leq (c_r/\varepsilon)^2 \E(W_r^2) \leq (c_r/\varepsilon)^2 \mathrm{Var}(Z_1),
\end{align}
where the last step holds since $\E(W_r^2) \leq \E\big(\frac{1}{|S_r|}\sum_{i\in S_r} |Z_i|^2 \big) = \E |Z_1|^2 = \mathrm{Var}(Z_1)$
by Jensen's inequality.
Now, for all $r\in\N$, we have the bound
\begin{align}
\label{equation:ergodic-bound}
c_r = \frac{(2 r + 1)^m - (2 r - 1)^m}{(2 r + 1)^m} \leq \frac{(2 r - 1)^{m-1} 3^m}{(2 r + 1)^m} \leq \frac{3^m}{r}
\end{align}
where the first inequality holds 
by the following application of the binomial theorem, taking $x = 2 r - 1$ and $y = 2$:
for all $m\in\N$, $x\geq 1$, $y\geq 0$, 
$$ (x + y)^m - x^m = \sum_{k=1}^m {m \choose k} x^{m-k} y^k \leq x^{m-1}\sum_{k=0}^m {m \choose k} y^k = x^{m-1} (1+y)^m. $$
Therefore, combining Equations~\ref{equation:ergodic-markov} and \ref{equation:ergodic-bound},
we have that for all $\varepsilon>0$, 
$$ \sum_{r=1}^\infty \Pr(c_r W_r \geq \varepsilon) \leq \sum_{r=1}^\infty (c_r/\varepsilon)^2 \mathrm{Var}(Z_1) 
\leq (3^m/\varepsilon)^2 \mathrm{Var}(Z_1) \sum_{r=1}^\infty \frac{1}{r^2} < \infty. $$
Hence, by the Borel--Cantelli lemma, 
for all $\varepsilon>0$, $\Pr(\limsup_r c_r W_r \leq \varepsilon) = 1$.
Therefore, $\Pr(\lim c_r W_r = 0) = \Pr(\cap_{k=1}^\infty \{\limsup_r c_r W_r \leq 1/k\}) = 1$,
that is, $c_r W_r \to 0$ a.s.\ as $r\to\infty$.
\end{proof}

In the following lemma, we adopt the notational conventions and definitions of \citet{kallenberg2002foundations}, page 181.

\begin{lemma}
\label{lemma:stationary-ergodic}
Let $\xi$ be a random element in $S$ with distribution $\mu$,
and let $T:S\to S$ be $\mu$-preserving (that is, $T(\xi) \overset{d}{=} \xi$).
Suppose $f:S\to S$ is a measurable function such that $f\circ T = T \circ f$.
Then $T(f(\xi)) \overset{d}{=} f(\xi)$, that is, $T$ preserves the distribution of $f(\xi)$.
If, further, $\xi$ is $T$-ergodic, then $f(\xi)$ is $T$-ergodic.
\end{lemma}
\begin{proof}
The first part is immediate, since $T(f(\xi)) = f(T(\xi)) \overset{d}{=} f(\xi)$.
Suppose $\xi$ is $T$-ergodic.  In other words, suppose that for any measurable set $A$ such that $T^{-1}(A) = A$,
we have $\Pr(\xi \in A) \in \{0,1\}$.  (This is also equivalent to saying that the $T$-invariant sigma-algebra is trivial under $\mu$.)
To show that $f(\xi)$ is $T$-ergodic, let $A$ such that $T^{-1}(A) = A$.
Then $T^{-1}(f^{-1}(A)) = f^{-1}(T^{-1}(A)) = f^{-1}(A)$ since $f\circ T = T\circ f$, so we have $\Pr(f(\xi)\in A) = \Pr(\xi \in f^{-1}(A)) \in \{0,1\}$.
Hence, $f(\xi)$ is $T$-ergodic.
\end{proof}

In the proof of Theorem~\ref{theorem:grf-ergodic}, we apply Lemma~\ref{lemma:stationary-ergodic} in the following way.
Suppose $\xi$ is a real-valued stochastic process on $\Z^m$, that is, $\xi = (\xi(i_1,\ldots,i_m) : i \in \Z^m)$
where $\xi(i_1,\ldots,i_m)$ is a real-valued random variable.
Suppose $T_k$ is the shift transformation in coordinate $k$, that is, $T_k\xi = T_k(\xi) = (\xi(i_1,\ldots,i_k+1,\ldots,i_m) : i \in \Z^m)$.
Let $\varphi(\xi)\in\R$ be a measurable function of $\xi$, and 
define $f(\xi) = (\varphi(T_1^{j_1} \cdots T_m^{j_m} \xi) : j\in\Z^m)$.
Then $f(T_k \xi) = (\varphi(T_1^{j_1}\cdots T_k^{j_k+1}\cdots T_m^{j_m}\xi) : j\in\Z^m) = T_k f(\xi)$
and thus, $f\circ T_k = T_k\circ f$.
Hence, if $\xi$ is stationary with respect to $T_k$ (that is, $T_k\xi \overset{d}{=} \xi$)
and $\xi$ is $T_k$-ergodic, then by Lemma~\ref{lemma:stationary-ergodic}, $f(\xi)$ is stationary with respect to $T_k$ and is $T_k$-ergodic.

\vspace{1em}
\begin{proof}{\bf of Theorem \ref{theorem:cox}}~
For $\theta\in\R^D$,
$$ f_n(\theta) = -\frac{1}{n}\log \mathcal{L}_n^\mathrm{Cox}(\theta) - \frac{1}{n} \sum_{i=1}^n Z_i \log n
= \frac{1}{n}\sum_{i=1}^n H_{Y_i}^n(\theta) Z_i - \theta^\T \big({\textstyle\frac{1}{n} \sum_{i=1}^n X_i Z_i}\big)$$
where $H_y^n(\theta) = \log\big(\frac{1}{n}\sum_{j=1}^n \exp(\theta^\T X_j) \I(Y_j\geq y)\big)$.
Note that $f_n$ is $C^\infty$, as a composition of $C^\infty$ functions. 
Further, $f_n$ is convex on $\R^D$, since $H_{Y_i}^n(\theta)$ is convex by Lemma~\ref{lemma:K-lemma} with $\mu = \frac{1}{n}\sum_{j:Y_j\geq Y_i} \delta_{X_j}$.
By Lemma~\ref{lemma:cox-posdef}, $f''(\theta_0)$ is positive definite.

By the strong law of large numbers, $\frac{1}{n}\sum_{i=1}^n X_i Z_i \overset{\mathrm{a.s.}}{\longrightarrow} \E(X Z)$ as $n\to\infty$,
and by Lemma~\ref{lemma:cox-pointwise}, for all $\theta\in\R^D$, $\E|h_Y(\theta)Z| < \infty$ and
$\frac{1}{n}\sum_{i=1}^n H_{Y_i}^n(\theta) Z_i \overset{\mathrm{a.s.}}{\longrightarrow} \E\big(h_Y(\theta) Z\big)$ as $n\to\infty$.
Therefore, for all $\theta\in\R^D$, with probability $1$, $f_n(\theta)\to f(\theta)$.
Due to convexity, this implies that with probability $1$, for all $\theta\in\R^D$, $f_n(\theta)\to f(\theta)$. 

Let $m = \sup\{|x| : x\in\X\} < \infty$.
Then by Lemma~\ref{lemma:K-lemma}, $\big\vert (\partial^3/\partial\theta_j\partial\theta_k\partial\theta_\ell) H_{Y_i}^n(\theta)\big\vert \leq (2 m)^3 = 8 m^3$ for all $\theta\in\R^D$.  
Thus, $\|f_n'''(\theta)\|^2 = \sum_{j,k,\ell} \big\vert (\partial^3/\partial\theta_j\partial\theta_k\partial\theta_\ell) f_n(\theta)\big\vert^2 
\leq D^3 (8 m^3)^2$ for all $\theta\in\R^D$, $n\in\N$.  Hence, $(f_n''')$ is a.s.\ uniformly bounded on all of $\R^D$.
Thus, for any open ball $E$ containing $\theta_0$, the conditions of Theorem~\ref{theorem:altogether} are satisfied  with probability $1$.
\end{proof}

Note that $H_{Y_1}^n(\theta),H_{Y_2}^n(\theta),\ldots$ are not i.i.d., which is why the next lemma is not trivial.

\begin{lemma}
\label{lemma:cox-pointwise}
Suppose $(X,Y,Z),(X_1,Y_1,Z_1),(X_2,Y_2,Z_2),\ldots$ are i.i.d., where $X\in\X\subseteq\R^D$, $Y\geq 0$, and $Z\in\{0,1\}$.
Define $h_y(\theta) = \log \E\big(\exp(\theta^\T X) \I(Y \geq y)\big)$ and
$H_y^n(\theta) = \log\big(\frac{1}{n}\sum_{j=1}^n \exp(\theta^\T X_j) \I(Y_j\geq y)\big)$
for $\theta\in\R^D$, $y\geq 0$.
If $\X$ is bounded and the c.d.f.\ of $Y$ is continuous on $\R$,
then for all $\theta\in\R^D$,
$\E|h_Y(\theta)Z| < \infty$ and
$$\frac{1}{n}\sum_{i=1}^n H_{Y_i}^n(\theta)Z_i  \xrightarrow[n\to\infty]{\mathrm{a.s.}} \E\big(h_Y(\theta)Z\big).$$
\end{lemma}
\begin{proof}
Let $F(y) = \Pr(Y\leq y)$, $c^* = \sup\{y\in\R : F(y) < 1\}$, and $m = \sup\{|x| : x\in\X\} < \infty$.
Since $|X|\leq m$ and $F$ is continuous,
$\E|h_Y(\theta)Z| \leq m|\theta| - \E \log(1-F(Y)) = m|\theta| + 1$ 
because $F(Y)\sim \mathrm{Uniform}(0,1)$.
Fix $\theta\in\R^D$ and define $g(y) = h_y(\theta)$ and $G_n(y) = H_y^n(\theta)$.

First, we show that for all $c\in(0,c^*)$, 
\begin{align}\label{equation:cox-first}
\sup_{y\in[0,c]} |G_n(y) - g(y)| \xrightarrow[n\to\infty]{\mathrm{a.s.}} 0.
\end{align}
Let $S$ be a countable dense subset of $[0,c]$ such that $0,c\in S$.
For all $y\in S$, $G_n(y)\overset{\mathrm{a.s.}}{\longrightarrow} g(y)\in\R$ by the strong law of large numbers
since $0 < \E(e^{\theta^\T X}\I(Y\geq y)) \leq e^{m|\theta|} < \infty$.
Next, $G_n$ is a non-increasing function on $[0,c]$ (that is, if $0 \leq y < y' \leq c$ then $G_n(y)\geq G_n(y')$) since $y\mapsto \I(Y_j \geq y)$ is non-increasing.
Further, $g(y)$ is continuous on $[0,c]$ by the dominated convergence theorem, 
since $|e^{\theta^\T X}\I(Y\geq y)| \leq e^{m|\theta|}$ and $\Pr(Y=y)=0$ by the continuity of $F$.
Thus, with probability $1$, for all $n$ sufficiently large, $G_n$ is finite on $[0,c]$ 
since $G_n(0)\overset{\mathrm{a.s.}}{\longrightarrow} g(0)$
and $G_n(c)\overset{\mathrm{a.s.}}{\longrightarrow} g(c)$.
It follows that $\sup_{y\in[0,c]} |G_n(y) - g(y)| \overset{\mathrm{a.s.}}{\longrightarrow} 0$
by Lemma~\ref{lemma:uniform-convergence}.

Second, we show that for all $c\in(0,c^*)$, 
\begin{align}\label{equation:cox-second}
\frac{1}{n}\sum_{i=1}^n G_n(Y_i)Z_i\I(Y_i \leq c) \xrightarrow[n\to\infty]{\mathrm{a.s.}} \E\big(g(Y)Z\I(Y\leq c)\big).
\end{align}
To see this, observe that by Equation~\ref{equation:cox-first},
\begin{align*}
\bigg\vert \frac{1}{n}\sum_{i=1}^n & G_n(Y_i)Z_i\I(Y_i \leq c) - \frac{1}{n}\sum_{i=1}^n g(Y_i)Z_i\I(Y_i\leq c)\bigg\vert \\
&\leq \frac{1}{n}\sum_{i=1}^n |G_n(Y_i) - g(Y_i)|\I(Y_i\leq c) 
\leq \sup_{y\in[0,c]} |G_n(y) - g(y)| 
\xrightarrow[n\to\infty]{\mathrm{a.s.}} 0
\end{align*}
and $\frac{1}{n}\sum_{i=1}^n g(Y_i)Z_i\I(Y_i\leq c) \xrightarrow[n\to\infty]{\mathrm{a.s.}} \E\big(g(Y)Z\I(Y\leq c)\big)$ by the strong law of large numbers.

Third, we show that for all $c\in(0,c^*)$, 
\begin{align}\label{equation:cox-third}
\limsup_{n\to\infty} \bigg\vert\frac{1}{n}\sum_{i=1}^n G_n(Y_i) Z_i \I(Y_i > c)\bigg\vert 
\overset{\mathrm{a.s.}}{\leq} m|\theta|p_c - p_c\log p_c + p_c
\end{align}
where $p_c = \Pr(Y > c)$. This follows from the fact that
\begin{align*}
\bigg\vert\frac{1}{n}\sum_{i=1}^n G_n(Y_i) Z_i \I(Y_i > c)\bigg\vert 
&\leq \frac{1}{n} \sum_{i=1}^n |G_n(Y_i)|\I(Y_i > c) \\
&\leq \frac{1}{n}\sum_{i=1}^n \Big(m|\theta| - \log\big({\textstyle\frac{1}{n}\sum_{j=1}^n \I(Y_j \geq Y_i)}\big)\Big) \I(Y_i > c) \\
&\overset{\mathrm{a.s.}}{=} m|\theta| K_n/n - \frac{1}{n} \sum_{k=1}^{K_n} \log(k/n) \\
&\xrightarrow[n\to\infty]{\mathrm{a.s.}} m|\theta| p_c - \int_0^{p_c} (\log x) d x
= m|\theta| p_c - p_c\log p_c + p_c
\end{align*}
where $K_n = \sum_{i=1}^n \I(Y_i > c)$, using that $\Pr(Y_i = Y_j)=0$ for $i\neq j$ by continuity of $F$.

Now, we put these pieces together to obtain the result.
Writing $\frac{1}{n}\sum_{i=1}^n G_n(Y_i)Z_i = \frac{1}{n}\sum_{i=1}^n G_n(Y_i)Z_i\I(Y_i \leq c) + \frac{1}{n}\sum_{i=1}^n G_n(Y_i)Z_i\I(Y_i > c)$,
for all $c\in(0,c^*)$ we have
\begin{align*}
\bigg\vert \frac{1}{n}\sum_{i=1}^n & G_n(Y_i)Z_i - \E(g(Y)Z)\bigg\vert
\leq \bigg\vert \frac{1}{n}\sum_{i=1}^n G_n(Y_i)Z_i\I(Y_i \leq c) - \E\big(g(Y)Z\I(Y\leq c)\big) \bigg\vert  \\
& + \bigg\vert \E\big(g(Y)Z\I(Y\leq c)\big) - \E(g(Y)Z) \bigg\vert + \bigg\vert \frac{1}{n}\sum_{i=1}^n G_n(Y_i)Z_i\I(Y_i > c) \bigg\vert,
\end{align*}
and therefore, by Equations~\ref{equation:cox-second} and \ref{equation:cox-third},
\begin{align}
\limsup_{n\to\infty} & \bigg\vert \frac{1}{n}\sum_{i=1}^n G_n(Y_i)Z_i - \E(g(Y)Z)\bigg\vert \notag \\
&\overset{\mathrm{a.s.}}{\leq} 
\big\vert \E\big(g(Y)Z\I(Y\leq c)\big) - \E(g(Y)Z) \big\vert + m|\theta| p_c - p_c\log p_c + p_c.  \label{equation:cox-limsup}
\end{align}
Let $c_1,c_2,\ldots\in(0,c^*)$ such that $c_k\to c^*$.
Then $p_{c_k} \to p_{c^*} = 0$ by continuity of $F$, and thus, $m|\theta| p_{c_k} - p_{c_k}\log p_{c_k} + p_{c_k} \to 0$ as $k\to\infty$.
Further, $\E\big(g(Y)Z\I(Y\leq c_k)\big) \to \E(g(Y)Z)$ by the dominated convergence theorem,
since $|g(Y)Z\I(Y\leq c_k)|\leq |g(Y)Z|$, $\E|g(Y)Z|<\infty$, and $\I(Y\leq c_k) \overset{\mathrm{a.s.}}{\to} 1$ as $k\to\infty$.
Applying Equation~\ref{equation:cox-limsup} to each $c_k$ and taking limits as $k\to\infty$,
we have that $\limsup_{n\to\infty} \big\vert \frac{1}{n}\sum_{i=1}^n G_n(Y_i)Z_i - \E(g(Y)Z)\big\vert = 0$ almost surely. 
\end{proof}

\begin{lemma}
\label{lemma:cox-posdef}
Under the conditions of Theorem~\ref{theorem:cox}, $f''(\theta)$ is positive definite for all $\theta\in\R^D$.
\end{lemma}
\begin{proof}
Recall that 
$f(\theta) = \E\big(h_Y(\theta) Z\big) - \theta^\T \E(X Z)$
where $h_y(\theta) = \log \E(e^{\theta^\T X} \I(Y \geq y))$  for $\theta\in\R^D$.
First, we put $h_y(\theta)$ in the form of $\kappa(\theta)$ in Lemma~\ref{lemma:K-lemma} by noting that
$h_y(\theta) = \log \E(e^{\theta^\T X}\Pr(Y \geq y \mid X)) = \log\int \exp(\theta^\T x) \mu_y(d x)$
where $\mu_y(d x) = \Pr(Y\geq y \mid X = x) P(d x)$ 
and $P$ is the distribution of $X$
\citep[10.2.1-10.2.2]{dudley2002real}.
Let $m = \sup\{|x| : x\in\X\} < \infty$.
We have $|h_y(\theta)|<\infty$ for all $\theta\in\R^D$ and all $y\geq 0$ because
$\exp(-m|\theta|) \leq \exp(\theta^\T X) \leq \exp(m|\theta|)$,
and thus $-\infty < -m|\theta| + \log \Pr(Y \geq y) \leq h_y(\theta) \leq m|\theta| + \log\Pr(Y \geq y) < \infty$
due to conditions \ref{item:cox-bounded} and \ref{item:cox-positive} of Theorem~\ref{theorem:cox}.

For any given $\theta\in\R^D$ and $y\geq 0$, following Lemma~\ref{lemma:K-lemma}, we define a probability measure $\tilde P = \tilde P_{\theta,y}$ on $\X$ by
$\tilde P(d x) = \exp(\theta^\T x - h_y(\theta)) \Pr(Y \geq y \mid X = x) P(d x)$.
Note that $P$ and $\tilde P$ are mutually absolutely continuous
since $\exp(\theta^\T x - h_y(\theta)) \Pr(Y \geq y \mid X = x)$ is strictly positive for all $x\in\X$.
By Lemma~\ref{lemma:K-lemma}, $h_y'(\theta) = \E(\tilde X)$ and $h_y''(\theta) = \mathrm{Cov}(\tilde X)$ where $\tilde X\sim \tilde P$.
We claim that for any nonzero $a\in\R^D$, $a^\T h_y''(\theta) a > 0$. 
To see this, suppose $a\in\R^D$ such that $a^\T h_y''(\theta) a = 0$. 
Since $a^\T h_y''(\theta) a = \mathrm{Var}(a^\T \tilde X)$,
it follows that $\Pr(a^\T \tilde X = \E(a^\T \tilde X)) = 1$.
But then $\Pr(a^\T X = \E(a^\T \tilde X)) = 1$ since $P \ll \tilde P$.
Hence, $a^\T X$ is a.s.\ equal to a constant, so $\mathrm{Var}(a^\T X) = 0$, which implies $a = 0$ by condition  
\ref{item:cox-nonconstant} of Theorem~\ref{theorem:cox}.

To justify differentiating under the expectation in $\E(h_Y(\theta) Z)$,
we apply \citet[Theorem 2.27b]{folland2013real} using the following bounds.
First, $\E|h_Y(\theta)Z| < \infty$ by Lemma~\ref{lemma:cox-pointwise}.
Next, $|\tilde X| \leq m$ because $\tilde P$ is supported on $\X$.
Thus, $|\frac{\partial}{\partial\theta_j} h_y(\theta)z| = |\E(\tilde X_j)z| \leq \E|\tilde X_j| \leq \E|\tilde X| \leq m$ 
and $|\frac{\partial^2}{\partial\theta_j\partial\theta_k} h_y(\theta)z| = |\mathrm{Cov}(\tilde X_j,\tilde X_k)z| \leq \E|\tilde X_j||\tilde X_k| + \E|\tilde X_j|\E|\tilde X_k| \leq 2 m^2$ for $z\in\{0,1\}$.

Hence, $f''(\theta) = \E\big(h_Y''(\theta)Z\big)$, and we have that for any nonzero $a\in\R^D$, $a^\T f''(\theta) a = \E\big(a^\T h_Y''(\theta) a Z\big) > 0$
because $a^\T h_Y''(\theta) a > 0$ and $\Pr(Z = 1) > 0$ due to condition \ref{item:cox-nonconstant} of Theorem~\ref{theorem:cox}.
Therefore, $f''(\theta)$ is positive definite.
\end{proof}

\section{Supporting results}
\label{section:supporting-results}

This section contains miscellaneous supporting results used in the proofs.
A metric space $E$ is \textit{totally bounded} if for any $\delta>0$, there exist $x_1,\ldots,x_k\in E$,
for some $k\in\N$, such that $E =\bigcup_{i = 1}^k\{x\in E: d(x,x_i)<\delta\}$. 
In particular, any bounded subset of a Euclidean space is totally bounded.

\begin{lemma}\label{lemma:equicontinuous-convergence}
Suppose $h_n:E\to F$ for $n\in\N$, where $E$ is a totally bounded metric space and $F$ is a normed space.
If $(h_n)$ converges pointwise and is equicontinuous, then it converges uniformly.
\end{lemma}
\begin{proof}
Let $\epsilon>0$. Choose $\delta>0$ by equicontinuity, so that for any $n\in\N$, $x,y\in E$,
if $d(x,y)<\delta$ then $\|h_n(x)-h_n(y)\|<\epsilon$.
Choose $x_1,\ldots,x_k\in E$ by totally boundedness, and by pointwise convergence, let $N$ such that for all $m,n>N$,
for all $i\in\{1,\ldots,k\}$,
$\|h_m(x_i) - h_n(x_i)\|<\epsilon$. Then, for any $x\in E$, there is some $i\in\{1,\ldots,k\}$ such that $d(x,x_i)<\delta$, and thus
$$\|h_m(x) - h_n(x)\|\leq\|h_m(x) - h_m(x_i)\|+\|h_m(x_i) - h_n(x_i)\|+\|h_n(x_i) - h_n(x)\| < 3\epsilon $$
for any $m,n>N$. Therefore, $(h_n)$ converges uniformly \citep[by e.g.,][7.8]{rudin1976principles}. 
\end{proof}

When all the $k$th order partial derivatives of $f$ exist,
let $f^{(k)}(x)$ denote the $k$-way tensor of $k$th derivatives;
in particular, $f^{(1)} = f'$, $f^{(2)}=f''$, and so on.
When these derivatives are continuous, the order of differentiation does not matter
\citep[exercise 9.29]{rudin1976principles}.

\begin{lemma}\label{lemma:UBD}
Let $E\subseteq\R^D$ be open and convex, and let $f_n:E\to\R$ for $n\in\N$.
For any $k\in\N$, if each $f_n$ has continuous $k$th-order derivatives and $(f_n^{(k)})$ is uniformly bounded,
then $(f_n^{(k-1)})$ is equi-Lipschitz.
\end{lemma}
\begin{proof}
First, we prove the case of $k = 1$.
Let $C = \sup_n \sup_{x\in E}|f_n'(x)|<\infty$. By Taylor's theorem, for any $n\in\N$, $x,y\in E$,
$f_n(x) = f_n(y) + f_n'(z)^\T (x - y)$ for some $z$ on the line between $x$ and $y$, and therefore, 
$$ |f_n(x) - f_n(y)|\leq|f_n'(z)| \,|x - y|\leq C|x - y|. $$
Thus, $(f_n)$ is equi-Lipschitz.


For notational clarity, we prove the case of $k = 3$, and observe that the extension from this to the general case is immediate.
For any $i,j\in\{1,\ldots,D\}$, if we define $h_n(x) = f_n''(x)_{i j} = \frac{\partial^2}{\partial x_i \partial x_j} f_n(x)$,
then $(h_n')$ is uniformly bounded (since $|h_n'(x)|\leq\|f_n'''(x)\|$ and $(f_n''')$ is uniformly bounded), and
hence, $(h_n)$ is equi-Lipschitz by the case of $k = 1$ just proven.
Thus, $(f_n'')$ is equi-Lipschitz, since if $C_{i j}$ is the equi-Lipschitz constant for entry $(i,j)$, then 
$$\|f_n''(x) - f_n''(y)\|^2 = \sum_{i,j} |f_n''(x)_{i j} - f_n''(y)_{i j}|^2 \leq C^2|x-y|^2$$
where $C^2 = \sum_{i,j} C_{i j}^2$.  
\end{proof}

\begin{lemma}\label{lemma:smooth-min}
Let $B\subseteq\R^D$ be open and let $f:B\to\R$ be differentiable. If $x_0\in B$ such that $f(x)\geq f(x_0)$ for all $x\in B$, then $f'(x_0) = 0$.
\end{lemma}
\begin{proof}
For any $u\in\R^D$ with $|u|=1$, $f'(x_0)^\T u = \lim_{\epsilon\to 0} (f(x_0 + \epsilon u) - f(x_0)) \geq 0$. 
If $f'(x_0)\neq 0$, then choosing $u = -f'(x_0)/|f'(x_0)|$, we have $0 \leq f'(x_0)^\T u = -|f'(x_0)| < 0$, a contradiction.
\end{proof}


\begin{lemma}
\label{lemma:uniform-convergence}
Let $a,b\in\R$ such that $a<b$, let $g:[a,b]\to\R$ be continuous, and for $n\in\N$, let $g_n:[a,b]\to\R$ be a non-increasing function.
If there is a dense subset $S\subseteq[a,b]$ such that $a,b\in S$ and $g_n(y)\to g(y)$ for all $y\in S$, 
then $\sup_{y\in[a,b]} |g_n(y) - g(y)| \longrightarrow 0$ as $n\to\infty$.
\end{lemma}

Lemma~\ref{lemma:uniform-convergence} is straightforward to verify, so we omit the proof.
Lemmas~\ref{lemma:K-differentiation} and \ref{lemma:K-lemma} are standard well-known results, but we provide precise statements and proofs for completeness.
We write $S^\circ$ to denote the interior of $S$.

\begin{lemma}
\label{lemma:K-differentiation}
Let $\mu$ be a Borel measure on $\R^D$ and define $G(\theta) = \int_{\R^D} \exp(\theta^\T x) \mu(d x)$ for $\theta\in\R^D$.
Let $S = \{\theta\in\R^D : G(\theta) < \infty\}$.
Then $G$ is $C^\infty$ on $S^\circ$ and for all $\theta\in S^\circ$, $k\in\{0,1,2,\ldots\}$, $i_1,\ldots,i_k\in\{1,\ldots,D\}$,
we have
\begin{equation}
\label{equation:K-differentiation}
\frac{\partial}{\partial\theta_{i_1}} \cdots \frac{\partial}{\partial\theta_{i_k}} G(\theta) = \int x_{i_1}\cdots x_{i_k} \exp(\theta^\T x) \mu(d x).
\end{equation}
\end{lemma}
\begin{proof}
We proceed by induction.  
By construction, for all $\theta\in S^\circ$, $\int |e^{\theta^\T x}| \mu(d x) < \infty$ and Equation~\ref{equation:K-differentiation} holds when $k=0$.
Fix $i_1,\ldots,i_k\in\{1,\ldots,D\}$ and suppose that for all $\theta\in S^\circ$, $\int |x_{i_1}\cdots x_{i_k} e^{\theta^\T x}| \mu(d x) < \infty$
and Equation~\ref{equation:K-differentiation} holds.
Let $j\in\{1,\ldots,D\}$ and $\theta_0\in S^\circ$. Define $u = (0,\ldots,0,1,0,\ldots,0)\in\R^D$ where the $1$ is in the $j$th position.
Choose $\epsilon > 0$ such that $\theta_0 + t u \in S^\circ$ for all $t\in [-2\epsilon,2\epsilon]$.
Define $f(x,t) = x_{i_1}\cdots x_{i_k} e^{(\theta_0 + t u)^\T x}$ and $F(t) = \int f(x,t) \mu(d x)$ for $x\in\R^D$, $t\in[-2\epsilon,2\epsilon]$.
Note that $\int |f(x,t)|\mu(d x) < \infty$ for all $t\in[-2\epsilon,2\epsilon]$ by the induction hypothesis.
Define $g(x) = |f(x,2\epsilon)|/\epsilon + |f(x,-2\epsilon)|/\epsilon$.
It is straightforward to verify that $|\frac{\partial f}{\partial t}(x,t)| = |x_j f(x,t)| \leq g(x)$ for all $x\in\R^D$, $t\in[-\epsilon,\epsilon]$,
by using the inequality $|x_j| \leq e^{\epsilon |x_j|}/\epsilon$. 
Further, $\int|g(x)|\mu(d x) < \infty$ by the induction hypothesis.
Therefore, $F$ is differentiable and $F'(t) = \int \frac{\partial f}{\partial t}(x,t) \mu(d x)$ for all $t \in (-\epsilon,\epsilon)$ 
by \citet[Theorem 2.27b]{folland2013real}.

Putting these pieces together, we have
\begin{align*}
\frac{\partial}{\partial\theta_j}\bigg\vert_{\theta=\theta_0} \frac{\partial}{\partial\theta_{i_1}} \cdots \frac{\partial}{\partial\theta_{i_k}} G(\theta)
&= \frac{\partial}{\partial\theta_j}\bigg\vert_{\theta=\theta_0} \int x_{i_1}\cdots x_{i_k} \exp(\theta^\T x) \mu(d x) \\
&= \frac{\partial}{\partial t}\bigg\vert_{t = 0} \int f(x,t)\mu(d x) 
= F'(0) = \int \frac{\partial f}{\partial t}(x,0)\mu(d x) \\
&= \int x_j f(x,0)\mu(d x)
= \int x_j x_{i_1}\cdots x_{i_k} \exp(\theta_0^\T x) \mu(d x)
\end{align*}
and $\int |x_j x_{i_1}\cdots x_{i_k} e^{\theta_0^\T x}| \mu(d x) = \int |\frac{\partial f}{\partial t}(x,0)|\mu(d x) \leq \int |g(x)|\mu(d x) < \infty$.
Since $j\in\{1,\ldots,D\}$ and $\theta_0\in S^\circ$ are arbitrary, this completes the induction step.
\end{proof}

\begin{lemma}
\label{lemma:K-lemma}
Let $\mu$ be a Borel measure on $\R^D$ and define $\kappa(\theta) = \log \int_{\R^D} \exp(\theta^\T x) \mu(d x)$ for $\theta\in\R^D$.
Let $\Theta = \{\theta\in\R^D : |\kappa(\theta)| < \infty\}$, and 
define $P_\theta(A) = \int_A \exp(\theta^\T x - \kappa(\theta)) \mu(d x)$ for $\theta\in\Theta$ and $A\subseteq\R^D$ Borel measurable.
Then $\Theta$ is a convex set and $\kappa$ is convex on $\Theta$.  Further, for all $\theta$ in the interior of $\Theta$, for all $i,j,k\in\{1,\ldots,D\}$, if $X\sim P_\theta$ then
\begin{enumerate}
\item\label{K-lemma-1} $\displaystyle \frac{\partial\kappa}{\partial\theta_i}(\theta) = \E(X_i)$,
\item\label{K-lemma-2} $\displaystyle \frac{\partial^2\kappa}{\partial\theta_i\partial\theta_j}(\theta) = \E\big((X_i - \E X_i)(X_j - \E X_j)\big) = \mathrm{Cov}(X_i,X_j)$, and
\item\label{K-lemma-3} $\displaystyle \frac{\partial^3\kappa}{\partial\theta_i\partial\theta_j\partial\theta_k}(\theta) = \E\big((X_i - \E X_i)(X_j - \E X_j)(X_k - \E X_k)\big)$.
\end{enumerate}
\end{lemma}
\noindent More succinctly, items \ref{K-lemma-1} and \ref{K-lemma-2} state that $\kappa'(\theta) = \E(X)$ and $\kappa''(\theta) = \mathrm{Cov}(X)$ where $X\sim P_\theta$.
\begin{proof}
Convexity of $\Theta$ and $\kappa$ is a straightforward application of H\"older's inequality.
Define $G(\theta) = \int \exp(\theta^\T x) \mu(d x)$ for $\theta\in\R^D$.
By Lemma~\ref{lemma:K-differentiation}, $G$ is $C^\infty$ on the interior of $\Theta$ and
its partial derivatives are given by Equation~\ref{equation:K-differentiation}.
The identities in items \ref{K-lemma-1} - \ref{K-lemma-3} are straightforward to derive using 
Equation~\ref{equation:K-differentiation} and the chain rule.
\end{proof}


\end{document}